\documentclass[12pt]{amsart}
\usepackage{physics,amsmath,amssymb}
\usepackage{enumitem}
\usepackage{graphicx}
\usepackage{nicefrac}
\usepackage{colonequals}
\usepackage{booktabs}
\usepackage{placeins}

\newtheorem{proposition}{Proposition}
\newtheorem{theorem}[proposition]{Theorem}

\newtheorem{lemma}[proposition]{Lemma}

\newtheorem{conjecture}[proposition]{Conjecture}

\theoremstyle{definition}
\newtheorem{definition}[proposition]{Definition}
\newtheorem{remark}[proposition]{Remark}

\usepackage{hyperref}

\newcommand{\Aut}{\operatorname{Aut}}
\newcommand{\SL}{\operatorname{SL}}
\newcommand{\GL}{\operatorname{GL}}

\begin{document}

\author{Will Sawin and Andrew V. Sutherland}

\title{Murmurations for elliptic curves ordered by height}

\begin{abstract}He, Lee, Oliver, and Pozdnyakov~\cite{HLOP} have empirically observed that the average of the $p$th coefficients of the $L$-functions of elliptic curves of particular ranks in a given range of conductors $N$ appears to approximate a continuous function of $p$, depending primarily on the parity of the rank. Hence the sum of $p$th coefficients against the root number also appears to approximate a continuous function, dubbed the murmuration density. However, it is not clear from this numerical data how to obtain an explicit formula for the murmuration density. Convergence of similar averages was proved by Zubrilina~\cite{Zubrilina} for modular forms of weight $2$ (of which elliptic curves form a thin subset) and analogous results for other families of automorphic forms have been obtained in further work~\cite{BBLLD,LOP}. Each of these works gives an explicit formula for the murmuration density. We consider a variant problem where the elliptic curves are ordered by naive height, and the $p$th coefficients are averaged over $p/N$ in a fixed interval. We give a conjecture for the murmuration density in this case, as an explicit but complicated sum of Bessel functions. This conjecture is motivated by a theorem about a variant problem where we sum the $n$th coefficients for $n$ with no small prime factors against a smooth weight function. We test this conjecture for elliptic curves of naive height up to $2^{28}$ and find good agreement with the data. The theorem is proved using the Voronoi summation formula, and the method should apply to many different families of $L$-functions. By a similar approach, we give a prediction murmuration density for elliptic curves of prime conductor, ordered by conductor, again matching the data but lacking a motivating theorem. This is the first work to give an explicit formula for the murmuration density of a family of elliptic curves, in any ordering.  \end{abstract}

\maketitle

\section{Introduction}

We begin with some notation. For $E$ an elliptic curve over the rational numbers, let $a_n(E)$ be the coefficient of $n^{-s}$ in the $L$-function of $E$. Then $L(E,s) = \sum_{n=1}^\infty a_n(E) n^{-s}$ satisfies the functional equation \[ (2\pi)^{-s} \Gamma(s) L(E,s)  =   \epsilon(E)  N(E)^{1-s}  (2\pi)^{s-2} \Gamma(2-s) L(E,2-s)\] for a unique positive integer $N(E)$, the conductor of $E$, and a unique $\epsilon(E) \in \{-1,1\}$, the root number of $E$.

For integers $A, B$, let $E_{A,B}$ be the curve with equation $y^2= x^3 +A x + B$. Every elliptic curve over $\mathbb Q$ can be expressed uniquely as $E_{A,B}$ for integers $A,B$ such that no prime $p$ simultaneously satisfies $p^4|A$ and $p^6|B$.  The \emph{naive height} of an elliptic curve $E/\mathbb Q$ is
\[
H(E)\colonequals H(E_{A,B})\colonequals\max ( 4 \abs{A}^3, 27 \abs{B}^2),
\]
where $A$ and $B$ are uniquely determined by $E_{A,B}\simeq E$ and $p^4\nmid A$ or $p^6\nmid B$ for all primes $p$.

For $S$ a set, let $\mathbb E_{E \in S} [f (E)] \colonequals \frac{1}{\abs{S}} \sum_{E \in S} f(E) $.

Murmurations for elliptic curves ordered by naive height are the averages
\begin{equation}\label{murmurations-average} \mathbb E_{ \{E:H(E)\le X\} }  \Bigl[  \frac{\log  ( N(E)  \frac{C_1+ C_2}{2}  ) }{ N(E) }\sum_{  \substack{ p \in (C_1 N(E) , C_2 N(E)] \\ p \textrm{ prime }}} \epsilon(E) a_p (E)  \Bigr] \end{equation}
The consideration of these averages is inspired by the original work of He, Lee, Oliver, and Pozdnyakov~\cite{HLOP}, who considered a somewhat different average. We briefly explain the motivation for these averages now, and give a more detailed explanation later, after we state our prediction for \eqref{murmurations-average}. Our goal is to understand the correlation between the $p$th coefficient $a_p$ and the root number $\epsilon$. When elliptic curves are ordered by conductor, this correlation seems to depend primarily on the ratio $p/N$ between $p$ and the conductor, and it was suggested by Jonathan Bober (in personal communication) to sum over $p/N$ in an interval $(C_1,C_2)$ when elliptic curves are not necessarily ordered by conductor. The prime number theorem implies that the number of primes with $p/N\in (C_1,C_2)$ is close to $\frac{ (C_2-C_1) N}{ \log (N(C_1+C_2)/2)}$ and we divide by this quantity (except for the constant $C_2-C_1$) before averaging over $E$ to normalize.

To state our conjecture, we need some additional notation. Let $J_1$ be the Bessel function of the first kind, and let $v_p(m)$ denote the $p$-adic valuation of an integer $m$. Our prediction also relies on local terms $\ell_{p,\nu}$ and $\hat{\ell}_{p,\nu}$ depending on a prime $p$ and a nonnegative integer $\nu$ defined, respectively, in Definitions \ref{lp} and \ref{lphat}. We give explicit formulas for these local terms in Lemmas \ref{all-lf} and \ref{prime-lf}.

\begin{conjecture}\label{main-conjecture} For real numbers $0<C_1<C_2$ we have 
\begin{equation}\label{main-prediction} \begin{aligned} & \lim_{X \to \infty}  \mathbb  E_{ \{E:H(E)\le X\} } \Bigl[  \frac{\log  ( N(E)  \frac{C_1+ C_2}{2}  ) }{ N(E) }\sum_{  \substack{ p \in (C_1 N(E) , C_2 N(E)] \\ p \textrm{ prime }}} \epsilon(E) a_p (E)  \Bigr] \\
 = &\int_{C_1}^{C_2}2\pi \sqrt{u}\!\!  \sum_{ \substack{ q \in \mathbb N \\ \emph{squarefree}}}\! \sum_{\substack{m \in \mathbb N }} \frac{ \mu(\gcd(m,q))}{q m \phi\left( \frac{q}{ \gcd(m,q) } \right) }  J_1  \left(4 \pi  \frac{\sqrt{u}m }{ q}  \right)   \prod_{p\mid q} \hat{\ell}_{p,2 v_p(m)} \!\! \prod_{p\mid m, p\nmid q} \ell_{p,{2v_p(m)}}\, du . \end{aligned}\end{equation} \end{conjecture}
We will check (in Lemma \ref{prime-absolute-convergence}) that the sum over $q$ and $m$ on the right-hand side of \eqref{main-prediction} is absolutely convergent uniformly on compact intervals and hence gives a continuous function. The integrand in the right-hand side of \eqref{main-prediction} is the murmuration density in the sense of \cite{SarnakLetter}.

The primary motivation for Conjecture \ref{main-conjecture} is the following theorem.

\begin{theorem}\label{prime} Let $W$ be a smooth, compactly-supported function on $(0,\infty)$. The limit 
\[ \lim_{P\to\infty}   \lim_{X\to\infty}  \mathbb E_{ \{E:H(E)\le X\} } \Bigl[  \frac{\prod_{p \leq P} (1-1/p)^{-1}  }{  N(E)  }\sum_{\substack{ n \in \mathbb N \\ p \nmid n \textrm{ for } p \leq P}} W \left( \frac{n}{ N(E )} \right) \epsilon(E) a_n (E)  \Bigr]\]
exists and is equal to
\[ \int_0^{\infty}\! W(u) 2\pi \sqrt{u}\!\!\sum_{\substack{ q\in\mathbb N \\ \emph{squarefree} }}\! \sum_{\substack{m \in \mathbb N }} \frac{ \mu(\gcd(m,q))}{q m \phi\left( \frac{q}{ \gcd(m,q) } \right) }  J_1  \left(4 \pi  \frac{\sqrt{u}m }{ q}  \right)   \prod_{p\mid q} \hat{\ell}_{p,2 v_p(m)} \!\! \prod_{p\mid m, p\nmid q} \ell_{p,{2v_p(m)}} du .  \]
\end{theorem}

Theorem \ref{prime} differs from Conjecture \ref{main-conjecture} in two ways. First, the sum over primes is replaced by a sum over natural numbers $n$ that have no prime factors $\leq P$, and the inverse density $\log (N \frac{C_1+C_2}{2} ) $ is replaced by the inverse density $\prod_{p\leq P } (1-1/p)^{-1}$ of numbers that have no prime factors $\leq P$. Second, the sum over $n$ with $n/N(E)$ in an interval $(C_1,C_2)$ is replaced with a sum over $n$ weighted by a smooth weight function $W(n/N(E))$, and the integral on the right-hand side is correspondingly made against $W$.

Thus, \eqref{main-prediction} is motivated by Theorem \ref{prime} and the heuristic that these two changes do not affect the density function. For the first change, this heuristic is a version of Cram\'er's random model of the primes: If we think of the primes as a random subset of the natural numbers with no prime factors $\leq P$, then a sum over the primes can be approximated by a sum over natural numbers with no prime factors $\leq P$, normalized appropriately by a density, and we should get better approximations as $P$ grows. For the second change, one generally expects that an estimate that holds for a smooth weight function should hold for more general weight functions, albeit likely with a worse error term.

It might be possible to prove a version of Theorem \ref{prime} with a sharp cutoff with additional analytic effort. On the other hand it seems very difficult to prove a version of Theorem \ref{prime} where the sum over $n$ is restricted to $n$ prime.

We tested \eqref{main-prediction} with a large amount of numerical data, summarized in Figure~\ref{fig:pred} below. Since the left-hand side is a limit as $X$ goes to $\infty$, we are only able to evaluate it by truncating to a particular value of $X$, i.e. by evaluating \eqref{murmurations-average}. Similarly, we truncate the infinite sum on the right-hand side, though at least in this case we know the sum converges and so can establish that the truncation is a good approximation for the infinite sum. (It should be possible to extract from our proof an explicit bound on the convergence rate.)

Figure~\ref{fig:pred} compares \eqref{murmurations-average} and the right-hand side of \eqref{main-prediction} for $C_1 = \nicefrac{j}{2000}, C_2 = \nicefrac{j+1}{2000}$ where $j$ ranges from $0$ to $1999$. Increasing values of $X$ yield progressively better fits, supporting \eqref{main-prediction}.
\begin{figure}[hbp!]
\includegraphics[width=\textwidth]{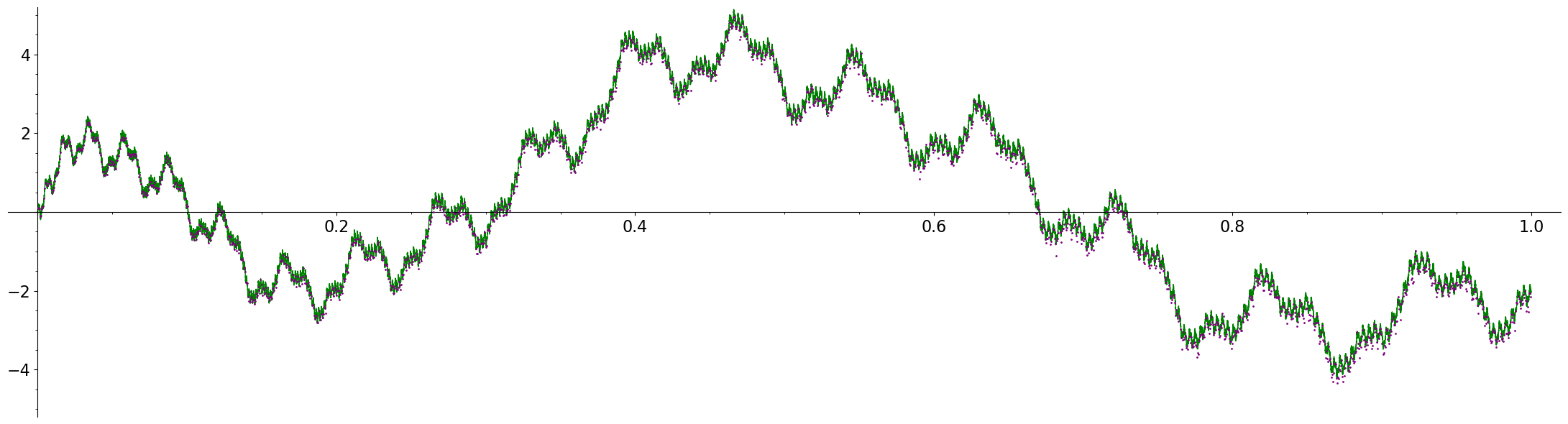}
\caption
{Plot comparing \eqref{murmurations-average} with $X=2^{28}$ (purple dots) to the RHS of \eqref{main-prediction} truncated at $q,m\le B=10^5$ (green curve) for intervals $(\nicefrac{j}{2000},\nicefrac{j+1}{2000}] \subseteq (0,1]$, normalized to have area one (via multiplication by $2000$).}\label{fig:pred}
\end{figure}

An interesting feature is that part of the discrepancy between the purple and green dots comes from a persistent downward bias, where \eqref{murmurations-average} is slightly less than its predicted limiting value. As part of the proof of Theorem~\ref{prime}, we prove a variant statement without the limit as $P$ goes to $\infty$. Graphs of the left-hand side and right-hand side of this statement do not show the downward bias, suggesting it is genuinely a property of the primes. This downward bias seems to decay as $X$ goes to $\infty$, suggesting it will go away in the limit, though it may decay more slowly than other sources of discrepancy between \eqref{murmurations-average} and the right-hand side of \eqref{main-prediction}.

A possible explanation for \eqref{murmurations-average} is related to elliptic curves of rank~$\geq 2$. Curves of larger rank have been observed since work of Birch and Swinnerton-Dyer~\cite{BSD} to have smaller $a_p$ values on average. Since, conjecturally, more curves have rank $2$ than any other rank $>1$, and, conjecturally, curves with rank $2$ have root number $+1$, with both conjectures empirically checked in a range that includes our dataset, this could push the average of $a_p$ times the root number downwards. (On the other hand, it is not obvious that this effect is not already accounted for by Theorem \ref{prime}.) However, while the proportion of curves with naive height $\leq X$ that have rank $2$ is expected to decrease to $0$ as $X\to\infty$, it actually increases for the range of $X$ we consider (from $2^{16}$ to $2^{28}$), as can be seen in \cite{BHKSSW}. Hence, to use curves of rank $\geq 2$ to explain the downward bias, one would have to explain why the bias decreases over this range, rather than increasing as one might expect.

We now give our results computing the local factors $\ell_{p,\nu}$ and $\hat{\ell}_{p,\nu}$ appearing in \eqref{main-prediction} and Theorem \ref{prime}.  We first introduce some further notation.

Let $S^{ k}_0 (\SL_2(\mathbb Z))$ be the space of holomorphic cusp forms of weight $k$ and level $1$. Let $U_\nu$ be the $\nu$th Chebyshev polynomial of the second kind (with $U_0=1$).

\begin{lemma}\label{all-lf}[Lemma \ref{all-lf-detailed}] Fix a prime $p$ and positive integer $\nu$. Let $\ell_{p,\nu}$ be the quantity defined in Definition \ref{lp}. If $p>3$ we have
\[ \ell_{p,\nu}= - \frac{p^{-1}- p^{-2}}{1-p^{-10}}  \sum_{ \substack{ f \in S^{ \nu+2}_0 (\SL_2(\mathbb Z))\\ \emph{eigenform} \\ a_1(f)=1 }} a_p (f) .\]
If $p=3$ we have
\[ \ell_{3,\nu}=  3^{\frac{\nu}{2} -2} \Bigl(   U_\nu \left( \frac{3}{ 2 \sqrt{3} }\right) +4 U_\nu \left( 0\right) +  U_\nu \left( \frac{-3}{ 2 \sqrt{3}} \right) \Bigr)  -   \frac{3^{-10}- 3^{-11}}{1-3^{-10}}  \sum_{ \substack{ f \in S^{ \nu+2}_0 (\SL_2(\mathbb Z))\\ \emph{eigenform} \\ a_1(f)=1 }} a_3 (f) . \]
If $p=2$ we have
\[ \ell_{2,\nu}=-  \frac{2^{-10}}{1-2^{-10}}   \sum_{ \substack{ f \in S^{ \nu+2}_0 (\SL_2(\mathbb Z)) \\ \emph{eigenform} \\ a_1(f)=1 }} a_2 (f) .\] 
\end{lemma}

\begin{lemma}\label{prime-lf}[Lemma \ref{prime-lf-detailed}] Fix a prime $p$ and a nonnegative even integer $\nu$.  Let $\hat{\ell}_{p,\nu}$ be the quantity defined in Definition \ref{lphat}. If $p>3$ we have
\[ \hat{\ell}_{p, 0 } =   \frac{1- p^{-1}}{ 1- p^{-10}}  \]
\[ \hat{\ell}_{p,2} =  - \frac{p-p^{-1} +p^{-2}- p^{-8} }{ (1-p^{-10}) (p-1) }   \]
and for $\nu>2$
\[\hat{\ell}_{p,\nu} = -  \frac{p^{-1}- p^{-2}}{1-p^{-10}} \Bigl(  p+1  + \!\!\!\!\! \sum_{ \substack{ f \in S^{ \nu+2}_0 (\SL_2(\mathbb Z))\\ \emph{eigenform} \\ a_1(f)=1 }}\!\!\!\! a_p (f) \Bigr).\]
If $p=3$ we have
\[ \hat{\ell}_{3, 0 } =   (1-3^{-10})^{-1}  \left( \frac{2}{3} + \frac{4}{3^{11}} \right) \]
\[ \hat{\ell}_{3,2} = -\frac{ 3 - 3^{-7}+16 \cdot 3^{-11}   } {2  (1-3^{-10} )} \]
and for $\nu>2$
\[\hat{\ell}_{3,\nu} = \frac{2}{9} (1-3^{-10} )^{-1}  \Bigl( 3^{ \frac{\nu}{2}} \Bigl(  U_\nu\! \left( \frac{3}{ 2 \sqrt{3} }\right)\! +2 U_\nu\! \left( 0\right)  +3^{-9} \Bigl(  U_\nu\! \left( \frac{2}{ 2 \sqrt{3} }\right)\!+U_\nu\! \left( \frac{1}{ 2 \sqrt{3} }\right) \Bigr) \Bigr)  -3^{-8}  \Bigr) . \]
If $p=2$ we have
\[ \hat{\ell}_{2, 0 } =  \frac{2^{-9}}{ 1- 2^{-10}} \]
\[ \hat{\ell}_{2,2} =  - \frac{ 4 - 2^{-6}  + 3  \cdot 2^{-10} }{1- 2^{-10}} \]
and for $\nu>2$
\[\hat{\ell}_{2,\nu}  = -  \frac{1}{ 2^{10}-1} \Bigl(  3  +\!\!\!\!\!\!\!\! \sum_{ \substack{ f \in S^{ \nu+2}_0 (\SL_2(\mathbb Z))\\ \emph{eigenform} \\ a_1(f)=1 }}\!\!\!\!\!\! a_2 (f) \Bigr).\]
\end{lemma}

\subsection{History and prior work}

We now explain how our work relates to prior work on murmurations.

Let $\operatorname{rank}(E)$ be the $\mathbb Q$-rank of the group of rational points of $E$. He, Lee, Oliver, and Pozdnyakov~\cite{HLOP} originally considered averages of the form
\begin{equation}\label{original-murmurations}  \mathbb E_{\substack{  N(E) \in [X, 2X] \\ \operatorname{rank}(E) = r}}  [ a_p (E) ]  \end{equation} and plotted them as a function of $p$ for fixed $r$, observing continuous oscillations, depending on the rank, that they dubbed ``murmurations". (They actually considered increasing intervals of length $1000$, but subsequent work has usually used dyadic intervals or slightly smaller intervals like $[X, X+ X^{1-\delta}]$ for $\delta>0$ small, as longer averages make analysis more tractable and data smoother.) This raised the question of finding a number-theoretic explanation for the murmurations, which would presumably also give a prediction for their shape.

In particular, later investigations~\cite{SutherlandLetter} revealed that the oscillations in these averages seem to reflect a continuous function of the ratio $p/X$ that becomes more apparent after further averaging over $p$ with $p/X $ in an interval. Thus, the murmurations problem in its original form might be to evaluate the limit of \eqref{original-murmurations} as $X$ goes to $\infty$ with $r$ fixed and $p/X$ converging to a fixed value, if it exists, or the limit of the average over $p$ such that $p/X$ lies in a fixed interval of \eqref{original-murmurations} as $X \to\infty$, otherwise.

Subsequent work has studied various modifications of \eqref{original-murmurations}. We explain the origin of the modifications that led to \eqref{murmurations-average}.

\begin{enumerate}[label=(\roman*)]
\item The unusual oscillatory behavior of \eqref{original-murmurations} seems to depend primarily on the parity of the rank. On the other hand, that elliptic curves of higher rank tend to have more points over small finite fields was previously observed and can be seen in the original form of the Birch and Swinnerton-Dyer conjecture~\cite{BSD}. Thus it makes sense to break into even and odd rank cases rather than considering each rank separately, especially because there are (empirically and conjecturally) many fewer elliptic curves of ranks greater than $1$. This was done in early follow-up work~\cite{SutherlandLetter,HLOPS}. Equivalently under the parity conjecture, we break into cases where the $\epsilon$ factor $\epsilon(E)$ of the functional equation of the $L$-function of $E$ is $+1$ or $-1$. In nearly all subsequent work, such as \cite{Zubrilina}, rather than considering these two cases separately, we subtract them from each other, as in
\[ \mathbb E_{\substack{  N(E) \in [X, 2X] }}  [ \epsilon(E) a_p (E) ] .\]
This can be motivated in multiple different ways. Most simply, by subtracting two data sets which are empirically approximately mirror images from each other, we have one function to work with instead of two. Second, analytic methods like the trace formula behave well on sums involving epsilon factors, while restricting to one parity would introduce extra terms.

Third, and most subtly, empirically the averages over elliptic curves with fixed parity are less smooth and slower to converge. It would be interesting to explain this theoretically, which seems challenging in our setting of elliptic curves, but may be doable for averages over modular forms or other settings where the trace formula can be applied, if it is possible to estimate these additional terms and show they lead to a slower rate of convergence. 

\item Bober (in personal communication) hypothesized that, rather than $p/X$, the crucial ratio is really $p/N(E)$, the ratio of the prime to the conductor of the elliptic curve. He suggested that rather than averaging $\epsilon(E)a_p(E) $ over many elliptic curves with different conductors, we first sum $a_p(E)$ over all primes $p$ with $p/N$ in a certain interval, and then sum over elliptic curves, producing formulas like
\[ \mathbb E_{\substack{  N(E) \in [X, 2X] }} \Bigl [ \sum_{ \substack{ p \in (C_1 N(E) , C_2 N(E)) \\ p \textrm{ prime }} }[ \epsilon(E) a_p (E) ]\Bigr ] .\]
As there are more primes $p$ in that interval when $N(E)$ is larger, this average over elliptic curves will be biased towards those with large conductor. To mitigate this, in this paper we divide by the expected number of $p$ in the range, as in
\begin{equation}\label{conductor-locally-averaged} \mathbb E_{\substack{  N(E) \in [X, 2X] }}  \Bigl[  \frac{\log \left( N(E)\frac{C_1+C_2}{2}\right) }{  N(E) }\sum_{  \substack{ p \in (C_1 N(E) , C_2 N(E)) \\ p \textrm{ prime }}} \epsilon(E) a_p (E) \Bigr ] .\end{equation}
An alternative, which should have similar behavior, and might be slightly better-motivated but harder to work with analytically, is to simply average $\epsilon(E)a_p(E)$ over all $p$ in the interval, as suggested in \cite{SutherlandTalk}, producing
\[ \mathbb E_{\substack{  N(E) \in [X, 2X] }} \Bigl [ \mathbb E_{ \substack{ p \in (C_1 N(E) , C_2 N(E)) \\ p \textrm{ prime }} }[ \epsilon(E) a_p (E) ]\Bigr ]. \]

\item We sum over elliptic curves ordered by naive height instead of conductor (though see \S\ref{ss-conductor} for partial progress on the conductor ordering). Naive height is also a natural way to order elliptic curves, and carries the advantage that many statistical results are known for elliptic curves ordered by naive height (while essentially none are known for elliptic curves ordered by conductor, so prospects for proving even a partial result towards the murmuration density for elliptic curves ordered by conductor are grim). However, the second author \cite{SutherlandLetter} observed that no murmuration patterns appeared when averaging $\epsilon(E) a_p(E)$ for fixed $p$ over elliptic curves $E$ ordered by naive height. Local averaging, as in (ii) above, fixes this, and makes murmurations visible when curves are ordered by naive height. Replacing the conductor ordering in \eqref{conductor-locally-averaged} with the naive height ordering, we obtain  \eqref{murmurations-average}.

\end{enumerate}

In a different direction, murmurations have been generalized from elliptic curves to many other families of $L$-functions, after the observation that murmurations are a general phenomenon that is not specific to elliptic curves \cite{HLOPS,SutherlandLetter}.

Work on murmurations that rigorously establishes murmuration densities has focused on families of automorphic forms, rather than elliptic curves, because the various trace formulae available for automorphic forms enable the calculation of averages of $\epsilon a_p$ over the family (though, in the case $\epsilon \neq \pm 1$, one usually considers $\epsilon^{-1} a_p$ instead). The closest to the setting of elliptic curves would be modular forms of fixed weight and varying level, as elliptic curves correspond to modular forms of fixed weight $2$ and varying level with rational coefficients under the modularity theorem. This case was studied by Zubrilina~\cite{Zubrilina}, obtaining an explicit formula for the murmuration density. Modular forms of level 1 and varying weight were studied by Bober, Booker, Lee, and Lowry-Duda~\cite{BBLLD} while Maass forms of level 1 were studied by Booker, Lee, Lowry-Duda, Seymour-Howell, and Zubrilina~\cite{BLLDSHZ}. In the setting of automorphic forms of rank $1$, murmurations for Dirichlet characters were studied by Lee, Oliver, and Pozdnyakov~\cite{LOP} and murmurations for Hecke $L$-functions of imaginary quadratic fields were studied by Wang~\cite{Wang}.

Cowan~\cite{Cowan2} established a form of murmurations for elliptic curves conditional on the ratios conjectures~\cite{ratios} and other hypotheses. (Cowan~\cite{Cowan3} also established an unconditional version of murmurations for families of quadratic twists of Dirichlet characters, although this, like Proposition \ref{all} below, does not restrict to primes.) Note that Cowan does not use the local average (ii), instead simply summing over $p$ in an interval depending on height. The ratios conjectures give a general recipe for conjecturing averages of ratios of $L$-functions. To apply this, Cowan uses an explicit formula relating $a_p$ for primes $p$ to the one-level density of the zeros of the $L$-function of the elliptic curve, which itself can be expressed in terms of the logarithmic derivative of the $L$-function of the elliptic curve, which is the sort of ratio covered by the ratios conjecture. Cowan's conjectural formula for the murmuration density has not yet been computed explicitly and compared against empirical data, but this likely could be done.

Cowan has indicated that the approach via the ratios conjecture can also be applied to the locally averaged version of murmurations. If this is done, the predicted murmurations density will presumably agree with ours, obtained via the Cram\'{e}r-type heuristic.  The ratios conjecture method involves transforming the $L$-functions in the numerators using functional equations in all possible ways to absorb any epsilon-factors that appear, expanding as a sum and replacing all $L$-function coefficients by their long-run average, and then summing the resulting terms. In our case, there is only one $L$-function in the numerator and only one way to transform it under the functional equation to absorb the epsilon-factor, so the first step is equivalent to applying the Voronoi summation formula, and then replacing Fourier coefficients with their long-run averages should be equivalent to applying the Cram\'er's heuristic, which can be expressed as an exchange of limits and hence as replacing quantities with their long-run averages.

This fits with Sarnak's suggestion \cite{SarnakLetter} that murmurations occur around a phase transition in the one-level density. Frequently in statistics of $L$-functions problems, a phase transition is where the contributions of small primes are most apparent. The murmurations would then arise from the contributions of small primes to the epsilon-factor-twisted average of the one-level density, refracted through the functional equation.

Finally, we discuss earlier work. While the sum of $\epsilon(E) a_p(E)$ over elliptic curves $E$ seems to have not attracted much study before the discovery of murmurations, sums of $\epsilon(E)$ and $a_p(E)$ alone had both been studied before: Showing cancellation in the sum of $\epsilon(E)$ over $E$ in a given family of elliptic curves is equivalent to showing that the root number equidistributes for that family, a question that was studied by Helfgott~\cite{Helfgott04}, proving conditional positive results for some families of elliptic curves and negative results for others. Estimating the sum of $a_p(E)$ over both $E$ and $p$, with smooth averaging in $p$, is roughly equivalent to calculating the one-level density of the $L$-functions of the family of elliptic curves. This one-level density has been studied by many authors, with some of the strongest results obtained by Young~\cite{Young}  and Baier and Zhao~\cite{BZ}. Both of these may be compared to Conjecture \ref{main-conjecture}, though the techniques required to study them are somewhat different.

Note that to estimate the average of $a_p(E)$ over elliptic curves with a given root number, it would suffice to estimate the average of $a_p(E) \epsilon(E)$ as well as the average of $a_p(E)$ and, to find the denominator, $\epsilon(E)$, so studying that form of murmurations could require combining all these directions.

\subsection{Definitions of the local factors}

We are now ready to define the local factors. We first review the concrete description of the coefficients of the $L$-function of the elliptic curve:

For $p$ prime, $a_p(E)$ is equal to $p+1 - \abs{E(\mathbb F_p)}$ if $E$ has good reduction at $p$, equal to $1$ if $E$ has split multiplicative reduction at $p$, equal to $-1$ if $E$ has non-split multiplicative reduction at $p$, and equal to $0$ if $E$ has additive reduction at $p$.  For $p^\nu$ a prime power, $a_{p^\nu}(E) = p^{ \frac{\nu}{2}} U_\nu \left( \frac{a_p(E)}{2\sqrt{p}}\right) $ for $E$ with good reduction at $p$ and $a_{p^\nu}(E) = (a_p(E))^\nu$ for $E$ with bad reduction at $p$. Finally we have $a_{nm}(E) =a_n(E) a_m(E)$ for $n,m$ coprime and this uniquely determines $a_n(E)$ for all $n$.

The formulas for $a_p$ and $a_{p^\nu}$ make sense for an elliptic curve $E$ over $\mathbb Q_p$, allowing us to define $a_{ p^\nu}(E)$ for $E$ an elliptic curve over $\mathbb Q_p$ by the same formulas. The local factors will be expressed as averages of $a_{p^\nu}(E)$ for elliptic curves $E$ over $\mathbb Q_p$.

Note that we define the local factors for all nonnegative integers $\nu$, even though only the local factors for even $\nu$ appear in our results. This is because the local factors for odd $\nu$ appear in the proofs, but all turn out to vanish, allowing us to state our final formulas using only terms with even $\nu$.

\begin{definition}\label{lp} For a prime $p$ and nonnegative integer $\nu$, let \[\ell_{p,\nu} \colonequals \frac{1}{1-p^{-10}} \int_{  (A,B) \in \mathbb Z_p^2 \setminus ( p^4 \mathbb Z_p \times  p^6 \mathbb Z_p)}  a_{p^\nu}( E_{A,B}), \] with the integral taken against the uniform measure on $\mathbb Z_p^2 $ with total mass $1$.\end{definition}

\begin{definition}\label{lphat} For a prime $p$ and a nonnegative integer $\nu$, define  $\hat{\ell}_{p, \nu}$ by
\[ \hat{\ell}_{p, 0 } \colonequals  (1-p^{-10} )^{-1}   \int_{ \substack{A, B\in \mathbb Z_p \\ p^4\nmid A \textrm{ or } p^6 \nmid B \\ p \nmid N(E_{A,B}) }} 1 \]  
\[ \hat{\ell}_{p,1} \colonequals (1-p^{-10} )^{-1}  \int_{ \substack{A, B\in \mathbb Z_p \\ p^4\nmid A \textrm{ or } p^6 \nmid B \\ p^2 \nmid N(E_{A,B}) }}  a_p(E_{A,B}) \cdot \begin{cases} (1-1/p )^{-1} & \textrm{if } p \mid N(E_{A,B}) \\ 1 & \textrm{if } p\nmid N(E_{A,B}) \end{cases}\]
\[ \hat{\ell}_{p,2} \colonequals (1-p^{-10} )^{-1}  \int_{ \substack{A, B\in \mathbb Z_p \\ p^4\nmid A \textrm{ or } p^6 \nmid B  }} \cdot \begin{cases}  a_{p^2}(E_{A,B}) & \textrm{if }E_{A,B}\textrm{ has good reduction} \\ - p a_{p^2}(E_{A,B}) & \textrm{if }E_{A,B}\textrm{ has multiplicative reduction} \\  - \frac{p^2}{p-1} & \textrm{if }E_{A,B} \textrm{ has additive reduction} \end{cases} \]
and for $\nu>2$
\[\hat{\ell}_{p,\nu} \colonequals (1-p^{-10} )^{-1}  \int_{ \substack{A, B\in \mathbb Z_p \\ p^4\nmid A \textrm{ or } p^6 \nmid B \\ p^2 \nmid N(E_{A,B}) }}   a_{p^\nu} (E_{A,B})  \cdot \begin{cases} 1 & \text{if } p \nmid N(E_{A,B}) \\ -p & \textrm{if } p \mid N(E_{A,B}) \end{cases} .\]  \end{definition}

\subsection{Strategy of the proof}

We now explain the key ideas in the proof of Theorem \ref{prime}. A common tool for estimating sums of coefficients of modular forms $a_n$ against smooth weight functions is the Voronoi summation formula. This formula relates two different sums over $a_n$ against two different smooth weight functions, and there is an uncertainty principle phenomenon: If one sum has a very smooth weight function supported on many different $n$, the other sum will have a very concentrated weight function supported on very few $n$. This will be helpful to us as the concentrated weight function is easier to estimate. 

The usual statement of the Voronoi formula does not involve the condition that $p\nmid n$ for $p\leq P$, instead allowing us to weight the sum by additive characters. Thus we express the condition $p\nmid n$ in terms of additive characters before applying the Voronoi formula. (Actually, we do this only for $p$ of good reduction, because there are other ways to express the condition $p\nmid n$ when $E$ has good reduction at $n$ that turn out to be simpler.)  
 
The resulting sum with a concentrated weight function supported on very few $n$ converges rapidly enough that we are able to exchange it with the expectation and the limit over $X$, requiring us to calculate the asymptotic averages of $a_n$ over elliptic curves ordered by naive height, together with some additional terms at places of bad reduction. This is not so hard, since $a_n(E_{A,B})$ depends only on the congruence classes of $E_{A,B}$ modulo some power of $n$, and it is easy to estimate the number of $A,B$ in each congruence class. This leads to the local factors $\ell_{p,\nu}$ and $\hat{\ell}_{p,\nu}$. Since this average vanishes unless $n$ is a perfect square, we introduce a change of variables $n=m^2$ that simplifies the expression.
 
A similar proof should apply to elliptic curves ordered by other reasonable height functions, although the local factors may be different, especially at the primes $2$ and $3$ --- for the Faltings height it seems likely that $2$ and $3$ need not be special cases at all, and instead the same formulas which here calculate $\ell_p$ and $\hat{\ell}_p$ for $p>3$ should work for all primes.
  
More generally, a similar proof should be applicable to many different families of $L$-functions, producing predicted murmuration densities. In particular, this should apply to geometric families in the sense of~\cite{SST} (arising from families of algebraic varieties), while the trace formula method used in \cite{BBLLD,BLLDSHZ,LOP,Wang,Zubrilina} is only applicable to harmonic families (sets of automorphic forms defined by local conditions).   Since the method depends on the Voronoi summation formula, which depends on the functional equation of the $L$-function, one needs either to work in a case where the functional equation of the $L$-function is known or include the functional equation as another heuristic assumption. (In cases where the epsilon factor is not just $\pm 1$, one should always use the inverse of the $\epsilon$ factor to define the initial sum, as in~\cite{LOP}, so that it cancels the epsilon factor appearing in the summation formula.) One can then obtain averages of $L$-function coefficients on the other side by a similar method of counting points in congruence classes for geometric families. The method could also apply to harmonic families, where the averages of the $L$-function coefficients on the other side can be estimated by the trace formula, though it is not clear if there are cases where this is better than applying the trace formula directly. 

There has been interest in the murmurations of hypergeometric motives~\cite{CKR}, which depend on some character data and a rational number. If ordered by the Weil height of this rational number, a similar argument would give a murmuration density whose local factors $\ell_{p,\nu}$ can be expressed in terms of the trace of Frobenius on the $\nu$th symmetric power of hypergeometric sheaves over finite fields. The greater the weights of the Frobenius eigenvalues appearing in these cohomology groups for small $\nu$, the larger the $\ell_{p,\nu}$ will be, which should cause the murmuration density to be less smooth as the series defining it converges more slowly (or not at all). For example, this suggests a less smooth density when the trivial character does not appear in the character data, as then the first symmetric power has nontrivial $H^1$, or when the monodromy is orthogonal, as then the second symmetric power has trivial $H^2$.  A similar orthogonal monodromy case was studied in \cite{SarnakLetter}, where indeed the murmuration density is not continuous, and not even a measure. The case of modular forms of weight $1$ may be similar as well.

\subsection{Additional results} 

We also prove two variants of Theorem \ref{prime}. The first one simply removes the restriction that $p\nmid n$ for $p \leq P$. Replacing the average over $a_p$ with an average over $a_n$ for all $n$ in the study of murmurations was suggested by  Bober, Booker, Lee, and Lowry-Duda~\cite[Remark (9) on p. 5]{BBLLD}.

\begin{proposition}\label{all} Let $W$ be a smooth, compactly-supported function on $(0,\infty)$. The limit 
\begin{equation}\label{main-limit} \lim_{X \to\infty}  \mathbb E_{ \{E:H(E)\le X\} }  \Bigl[  \frac{1 }{  N(E)  }\sum_{n=1}^{\infty}W \left( \frac{n}{ N(E)} \right) \epsilon(E) a_n (E)  \Bigr]\end{equation}
exists and is equal to
\begin{equation}\label{main-evaluation}  \int_0^{\infty}\! W(u) 2\pi \sqrt{u}\sum_{m=1}^{\infty}  \frac{1}{m} J_1\bigl( 4 \pi \sqrt{u}m\bigr)  \prod_{p \mid m } \ell_{p,{ 2v_p(m)}}  du. \end{equation} \end{proposition}

Here the murmuration density in the sense of \cite{SarnakLetter} is $2\pi\sum_{m=1}^{\infty}  \frac{ \prod_{p \mid m } \ell_{p,{ 2v_p(m)}} }{m}   \sqrt{u} J_1( 4 \pi \sqrt{u }m )$.

\begin{figure}
\includegraphics[width=\textwidth]{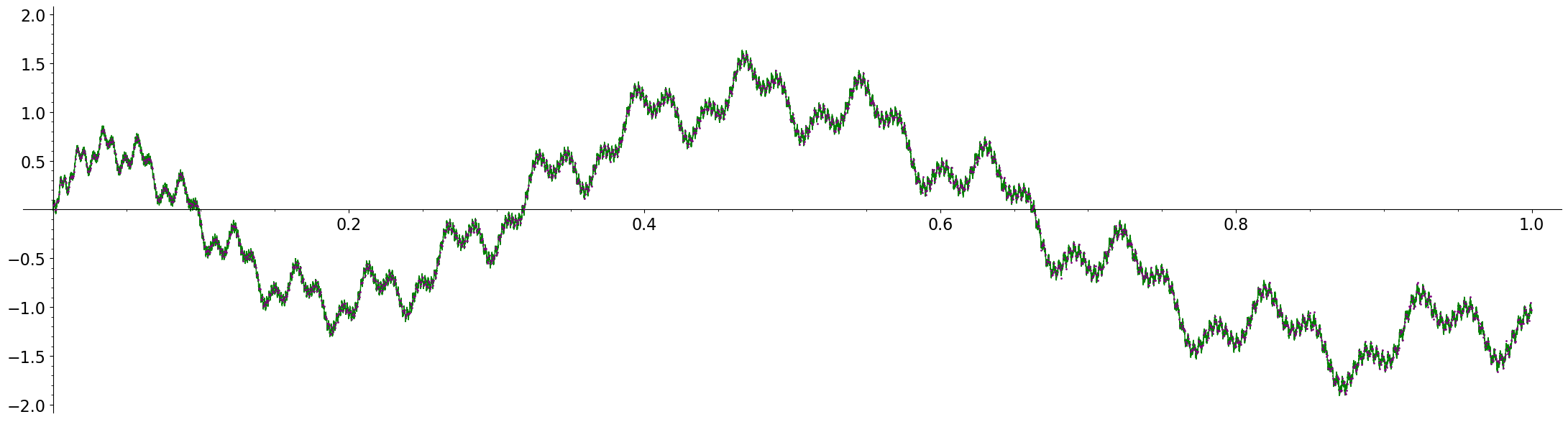}
\caption{Plot comparing truncated values of \eqref{main-limit} with $X=2^{28}$ (purple dots) and \eqref{main-evaluation} (green curve) truncated at $m\le B=10^5$ in Proposition~\ref{all} for normalized indicator functions $W_i(u)$ on $(\nicefrac{i}{2000},\nicefrac{(i+1)}{2000}) \subset [0,1]$ of area 1.}\label{fig:all}
\end{figure}

Second, we add the restriction that $p\nmid N(E)$ for $p\leq P$. A similar heuristic based on Cram\'er's random model for primes suggests that this case could model the murmuration density for a sum over elliptic curves, ordered by naive height, with prime conductor, or in other words that
\[ \lim_{X \to\infty}  \mathbb E_{ \substack{\{E:H(E)\le X\} \\ N(E) \textrm{ prime} }}  \Bigl[  \frac{ \log( N(E) ) }{  N(E)  }\sum_{p \textrm{ prime}} W \left( \frac{p}{ N(E)} \right) \epsilon(E) a_p (E)  \Bigr]\]
\[ \approx \lim_{P\to\infty}  \lim_{X \to\infty}  \mathbb E_{ \substack{\{E:H(E)\le X\} \\ p\nmid N(E)\textrm{ for } p \leq P}}  \Bigl[  \frac{\prod_{p \leq P} (1-1/p)^{-1}  }{  N(E)  }\!\!\!\sum_{\substack{ n \in \mathbb N \\ p \nmid n \textrm{ for } p \leq P}}\!\!\! W \left( \frac{n}{ N(E)} \right) \epsilon(E) a_n (E) \Bigr].\]

To handle this case, we define new local terms.

\begin{definition}\label{lptilde} Let $\tilde{\ell}_{p,\nu}$ be the integral of $a_{p^\nu}(E_{A,B})$ over the set of $A,B \in \mathbb Z_p^2$ with $p^4 \nmid A$ or $p^6 \nmid B$ and such that $E_{A,B}$ has good reduction at $p$, divided by the total measure of that set. \end{definition}

\begin{proposition}\label{primes}  Let $W$ be a smooth, compactly-supported function on $(0,\infty)$. The limit 
\[ \lim_{P\to\infty}  \lim_{X \to\infty}  \mathbb E_{ \substack{ \{E:H(E)\le X\} \\ p\nmid N(E)\textrm{ for } p \leq P}}  \Bigl[  \frac{\prod_{p \leq P} (1-1/p)^{-1}  }{  N(E)  }\sum_{\substack{ n \in \mathbb N \\ p \nmid n \textrm{ for } p \leq P}}\!\!\! W \left( \frac{n}{ N(E)} \right) \epsilon(E) a_n (E)  \Bigr]\]
exists and is equal to 
\[  \int_0^\infty \! W(u) 2\pi \sqrt{u}  \!\! \sum_{ \substack{ q \in \mathbb N \\ \emph{squarefree}}}   \sum_{m=1}^{\infty}  \frac{ \mu ( \gcd(m,q))}{q m\phi \left( \frac{q}{\gcd(m,q)}\right)} J_1\left( 4 \pi \frac{\sqrt{u }m}{q} \right) \prod_{p\mid m} \tilde{\ell}_{p,{2v_p(m)}}  du . \]
\end{proposition}
In this case, the (predicted) murmuration density is
\[
2\pi\sqrt{u}\!\! \sum_{ \substack{ q \in \mathbb N \\ \textrm{squarefree}}}   \sum_{m=1}^{\infty}  \frac{ \mu ( \gcd(m,q))}{qm \phi \left( \frac{q}{\gcd(m,q)}\right)} J_1\left( 4 \pi \frac{\sqrt{u }m}{q} \right)\prod_{p\mid m } \tilde{\ell}_{p,{2v_p(m)}} .
\]

\begin{lemma}\label{primes-lf}[Lemma \ref{primes-lf-detailed}] Fix a prime $p$ and positive even integer $\nu$. If $p\neq3$ we have
\[ \tilde{\ell}_{p,\nu}  = - p^{-1} \Bigl(1 + \sum_{ \substack{ f \in S^{ \nu+2}_0 (\SL_2(\mathbb Z))\\ \emph{eigenform} \\ a_1(f)=1 }} a_p (f) \Bigr).\] 
If $p=3$ we have
\[ \tilde{\ell}_{3,\nu} =   \frac{ 3^{\frac{\nu}{2}}}{ 3(1 + 2 \cdot 3^{-10})} \Bigl(  U_\nu \left( \frac{3}{ 2 \sqrt{3} }\right) +2 U_\nu \left( 0\right)  +3^{-9} \Bigl(  U_\nu \left( \frac{2}{ 2 \sqrt{3} }\right)+U_\nu \left( \frac{1}{ 2 \sqrt{3} }\right)  \Bigr)\Bigr) . \]
\end{lemma}

\begin{figure}[!h]
\includegraphics[width=\textwidth]{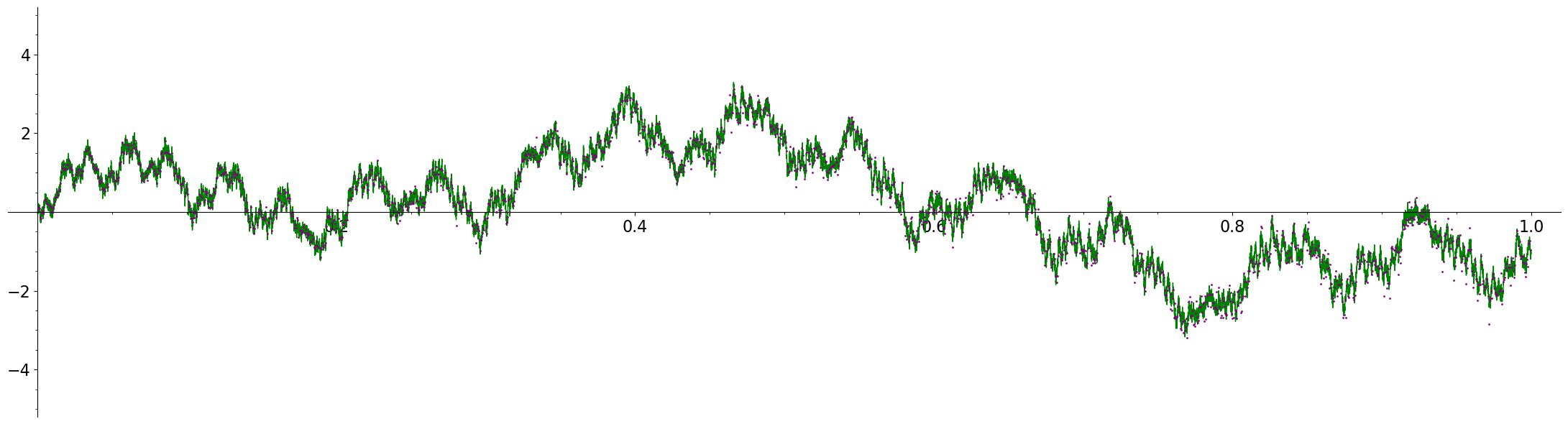}
\caption{Plot comparing truncated values of \eqref{main-limit} restricted to elliptic curves of prime conductor with height bounded by $X=2^{40}$ (purple dots) and the right-hand side of Proposition~\ref{primes} truncated at $q,m\le B=10^5$ (green curve) for normalized indicator functions $W_i(u)$ on $(\nicefrac{i}{2000},\nicefrac{(i+1)}{2000}) \subset [0,1]$ of area 1.}\label{fig:primes}
\end{figure}

\subsection{Ordering by conductor}\label{ss-conductor}

Finally, we discuss the possibility of applying a similar approach to elliptic curves ordered by conductor. The primary difficulty with this case is that the proofs of Theorem \ref{prime}, Proposition \ref{all}, and Proposition \ref{primes} all rely on computing averages over elliptic curves ordered by height. These include the average $\mathbb E_{  \{E:H(E)\le X\}} a_n(E)$ of the $n$th coefficient of the $L$-function, as well as more complicated averages including local factors at primes of bad reduction. To prove similar results about elliptic curves in the height ordering, it would be necessary to compute similar averages over elliptic curves of conductor less than $X$. However, such results are not known, and it does not seem possible to prove them at present levels of understanding.

Still, one can try to guess the behavior of averages in the conductor ordering such as $\mathbb E_{  \{E:N(E)\le X\}} a_n(E)$, prove analogues of Theorem \ref{prime}, Proposition \ref{all}, and Proposition \ref{primes} conditional on these guesses, and then use these to state an analogue of Conjecture \ref{main-conjecture}. While this approach relies on multiple heuristics, if supported by sufficient numerical evidence it could give a convincing explanation for murmurations in the conductor ordering and formula for the murmuration density. 

To show the potential viability of this approach, we give a prediction for the murmuration density of elliptic curves ordered by conductor, restricting to curves of prime conductor. This restriction simplifies things because, since all our elliptic curves have good reduction at all small primes, we do not have to guess the probability of bad reduction of a given type at a small prime. Since the reasoning is non-rigorous in multiple steps, we do not give any proofs and simply state a prediction, an informal justification, and give numerical evidence for it.

It seems likely that for real numbers $0<C_1<C_2$ we have 
\begin{equation}\label{prime-conductor-prediction} \begin{aligned} & \lim_{X \to \infty}  \mathbb  E_{ \{E:N(E)\le X, N(E) \textrm{prime}\} } \Bigl[  \frac{\log  ( N(E)  \frac{C_1+ C_2}{2}  ) }{ N(E) }\sum_{  \substack{ p \in (C_1 N(E) , C_2 N(E)] \\ p \textrm{ prime }}} \epsilon(E) a_p (E)  \Bigr] \\
 = &\int_{C_1}^{C_2}2\pi \sqrt{u}\!\!  \sum_{ \substack{ q \in \mathbb N \\ \emph{squarefree}}} \sum_{\substack{m \in \mathbb N }} \frac{ \mu(\gcd(m,q))}{q m \phi\left( \frac{q}{ \gcd(m,q) } \right) }  J_1  \left(4 \pi  \frac{\sqrt{u}m }{ q}  \right)  \prod_{p\mid m } \tilde{\ell}'_{p,{2v_p(m)}} \, du, \end{aligned}\end{equation}
where
\begin{equation}\label{lp-prime-conductor} \tilde{\ell}'_{p,\nu}  = - p^{-1} \Bigl(1 + \sum_{ \substack{ f \in S^{ \nu+2}_0 (\SL_2(\mathbb Z))\\ \emph{eigenform} \\ a_1(f)=1 }} a_p (f) \Bigr)\end{equation} for all $p$.
 
The prediction in \eqref{prime-conductor-prediction}  agrees with Conjecture \ref{main-conjecture}, except that we use the conductor ordering, restrict to curves of prime conductor, and use the local factors $\tilde{\ell}'_{p,\nu}$ instead of $\ell_{p,\nu}$ and $\hat{\ell}_{p,\nu}$. Note that the definition of $\tilde{\ell}'_{p,\nu}$ agrees with $\tilde{\ell}_{p,\nu}$ for all $p\neq 3$.
 
The motivation for this depends on the following observation used in the proofs of Proposition \ref{primes} and Lemma \ref{primes-lf}: Let $E_p$ be an elliptic curve over $\mathbb F_p$: Let $E$ be a random sample from the elliptic curves of height $<X$ whose conductors are not divisible by $p$. Then the probability that the reduction of $E$ mod $p$ is isomorphic to $E_p$ is $\frac{1}{p \abs{\operatorname{Aut}(E_p)}}$, unless $p=3$, in which case the probability is much larger for supersingular $E_p$ and much smaller for ordinary~$E_p$.
 
The reason that supersingular elliptic curves appear more frequently than ordinary elliptic curves at $3$ is because elliptic curves with good supersingular reduction modulo $3$ can be expressed by a short Weierstrass equation whose discriminant is not divisible by $3$, while elliptic curves with good ordinary reduction modulo $3$ can only be expressed by a short Weierstrass equation whose discriminant is divisible by $3^{12}$. Since the naive height bounds the discriminant of the short Weierstrass equation, elliptic curves with good supersingular reduction have larger naive heights and thus occur less frequently. This phenomenon is a consequence of the fact that naive height is defined using the short Weierstrass equation, and should not be expected to occur for intrinsic invariants like the conductor.
 
Thus, when we order by conductor, it is natural to guess that for random elliptic curves $E$ of prime conductor $<X$, the probability that the reduction of $E$ mod $p$ is isomorphic to $E_p$ is $ \frac{1}{p \abs{\operatorname{Aut}(E_p)}}$ for all $p$. This leads immediately to formula \eqref{lp-prime-conductor} for the average value of $a_{p^\nu}$ over elliptic curves of prime conductor, and then to \eqref{prime-conductor-prediction}.

Figure~\ref{fig:primescond} compares the two sides of the prediction in \eqref{prime-conductor-prediction} with $X=2^{30}$.

\begin{figure}[!h]
\includegraphics[width=\textwidth]{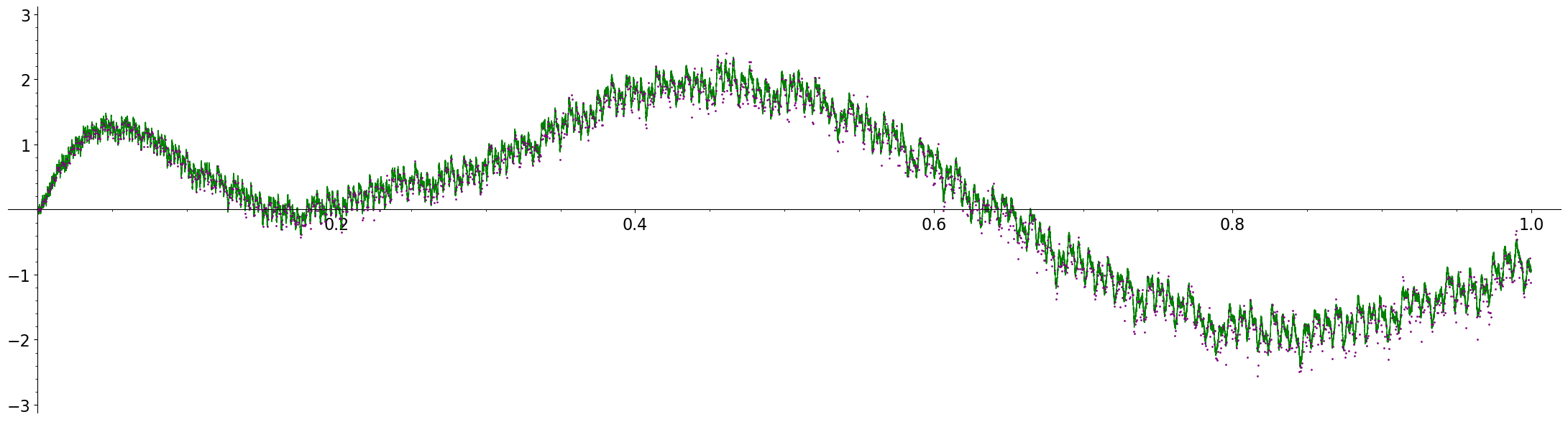}
\caption{Plot comparing truncated values of \eqref{prime-conductor-prediction} restricted to elliptic curves of prime conductor bounded by $X=2^{30}$ (purple dots), with the sums on the right-hand side restricted to $q,m\le B=10^5$ (green curve), for normalized indicator functions $W_i(u)$ on $(\nicefrac{i}{2000},\nicefrac{(i+1)}{2000}) \subset [0,1]$ of area 1.}\label{fig:primescond}
\end{figure}

\subsection{Acknowledgments} 

The authors would like to thank Peter Sarnak and Nina Zubrilina for helpful conversations. While working on this project the first author was supported by NSF grant DMS-2101491 and then by NSF grant DMS-2502029, and was a Sloan Research Fellow, and the second author was supported by Simons Foundation grant 550033.

\section{Proofs}

We will prove the main results in order of increasing difficulty, beginning with Proposition~\ref{all} followed by Proposition \ref{primes} and Theorem \ref{prime}, so that each can serve as a warmup to the next. After these are all established, we will prove Lemmas \ref{all-lf-detailed}, \ref{primes-lf-detailed}, and \ref{prime-lf-detailed}, giving formulas for the local factors.

We start with some preparatory lemmas.

\begin{lemma}\label{elliptic-Voronoi} Let $E$ be an elliptic curve over $\mathbb Q$, $N(E)$ the conductor of $E$, and $\epsilon(E)$ the epsilon factor of $E$. Let $q$ be a positive integer, $a$ an integer coprime to $q$, and $W \colon (0,\infty) \to \mathbb R$ a smooth compactly supported function. Let $\overline{a N(E)}$ denote the modular inverse of $aN(E)$ modulo~$q$. We have
\[  \frac{\epsilon(E)}{N(E)} \sum_{n=1}^\infty a_n(E) \sqrt{ \frac{n}{N(E)}} W \left( \frac{n}{N(E)}\right) e \left( \frac{an}{q} \right) \] \[ = \frac{  1}{q} \sum_{n=1}^{\infty} \frac{ a_n(E)}{\sqrt{n}} e \left( \frac{ \overline{a N(E)} n}{q}\right) \int_0^\infty 2\pi \sqrt{u} W(u)  J_1\left( 4 \pi  \frac{\sqrt{un}}{q}\right)du.\]
\end{lemma}

\begin{proof}  By the modularity theorem \cite{BCDT} for elliptic curves over $\mathbb Q$,  $f= \sum_{n=1}^\infty a_n(E) q^{n}$ is a modular form of weight $2$, level $N(E)$, and root number $\epsilon(E)$.

We now apply the Voronoi summation formula to this modular form.  A form of the Voronoi summation formula convenient for us is stated in \cite[Lemma 2.21]{BFKMMS} except that the factor of $i^{k_f}$ in the third displayed equation is erroneous (it should have been absorbed into the global $\epsilon$ factor when specializing the statement of \cite[Theorem A.4]{KMVK}, since $i^{k_f}$ is the local $\epsilon$ factor at $\infty$ and \cite[Theorem A.4]{KMVK} writes the factor at $\infty$ separately from the factors at all finite places while \cite[Lemma 2.21]{BFKMMS} groups them into a global factor.) Hence we use \cite[Lemma 2.21]{BFKMMS} without the $i^{k_f}$ factor.

We now specialize \cite[Lemma 2.21]{BFKMMS}. We note that $\lambda_f(n) = \frac{a_n}{\sqrt{n}}$ we have $|r|=N(E)$, we take $N =N(E)$ also, we have $k_f=2$ since $f$ is holomorphic and has weight $2$. This gives

\[   \sum_{n=1}^\infty \frac{ a_n(E) }{ \sqrt{n}}  W \left( \frac{n}{N(E)}\right) e \left( \frac{an}{q} \right) \] \[ =  \epsilon(E) \frac{ \sqrt{N(E)}}{q} \sum_{n=1}^{\infty} \frac{ a_n(E)}{\sqrt{n}} e \left( \frac{ \overline{a N(E)} n}{q}\right) \int_0^\infty 2\pi W(u)  J_1\left( 4 \pi  \frac{\sqrt{un}}{q}\right) du.\]
We now replace $W(u)$ with $\sqrt{u} W(u)$ on both the left and right sides. On the left-hand side, this produces a factor of $\frac{ \sqrt{n}}{ \sqrt{N(E)}}$. We divide both sides by $\sqrt{N(E)}$, producing a factor of $\frac{1}{N(E)}$ on the left-hand side, and divide both sides by $\epsilon(E)$, which does not need to be inverted as $\epsilon(E)^2=1$. \end{proof}

\begin{lemma}\label{Hankel-decreasing} For any smooth compactly supported function $W\colon (0,\infty) \to \mathbb R$, natural number~$m$, and $n \in (0,\infty)$, we have $\int_0^{\infty}W(u) 2\pi  \sqrt{u} J_1( 4 \pi \sqrt{un } ) du = O_{m, W} ( n^{-\frac{m}{2}} ) $.\end{lemma} 
The statement is a simpler version of \cite[Lemma 2.4]{BFKMM} and we use the same method of proof. \begin{proof} We have \cite[8.472.3]{ZMTables}  \[tJ_1(t) =  \left(  \frac{d}{t dt} \right)^m ( t^{m+1} J_{m+1}(t)),\] and setting $t= 4 \pi \sqrt{un}$ so that $\frac{d}{dt} =\frac{1}{2\pi} \sqrt{ \frac{u}{n}}  \frac{d}{du}$  we obtain  \[2 \pi \sqrt{u} J_1(4 \pi \sqrt{un} ) = \frac{1}{2\sqrt{n}}   \left(  \frac{1}{8\pi^2 n}  \frac{d}{du} \right)^m (4 \pi \sqrt{un})^{m+1} J_{m+1} \left( 4 \pi \sqrt{un} \right). \] Substituting in and integrating by parts $m$ times yields
\[  \int_0^{\infty}W(u)  2\pi \sqrt{u} J_1( 4 \pi \sqrt{un } ) du\] \[= \frac{1}{ 2\sqrt{n}}  \int_0^{\infty} W(u)   \left(  \frac{1}{ 8 \pi^2 n} \frac{d}{du} \right)^m \left( (4 \pi \sqrt{un})^{m+1} J_{m+1} ( 4 \pi \sqrt{un})  \right) du\] \[= \frac{ (-1)^m}{ 2 \sqrt{n}} \int_0^{\infty} \frac{d^m W(u)}{ du^m}   \frac{  (4 \pi \sqrt{un})^{m+1} }{ (8 \pi^2 n)^m }  J_{m+1} ( 4 \pi \sqrt{un} ) du  \] which since $J_{m+1}( t) = O_m(1)$ is 
\[ \ll_m  \frac{1}{\sqrt{n}}  \int_0^{\infty}  \abs{ \frac{d^m W(u)}{ du^m} }  \frac{  u^{ \frac{m+1}{2}} n^{\frac{m+1}{2}}}{ n^m}  du \ll n^{- \frac{m}{2}}, \] 
since $W$ is compactly supported so the integral over $u$ is bounded.\end{proof}

\begin{lemma}\label{congruence-counting} For a natural number $q$ and residue classes $a \bmod q^4, b \bmod q^6$, we have
\[\lim_{X\to\infty}  \mathbb E_{ \substack{ A, B \in \mathbb Z \\ \max ( 4 \abs{A}^3, 27 \abs{B}^2) \le X \\ p^4 \nmid A \textrm{ or } p^6\nmid B \textrm{ for all } p \\ 4 A^3+27 B^2 \neq 0 }}  [ 1 _{\substack{A \equiv a \bmod q^4\\ B \equiv b \bmod q^6}} ] =  \begin{cases} 0 & \textrm{if }p^4\mid a \textrm{ and }p^6 \mid b \textrm{ for some } p \mid q \\  \frac{1}{ q^{10} \prod_{p\mid q} (1- p^{-10} )} &\textrm{otherwise} \end{cases}.\]\end{lemma}
\begin{proof}
It suffices to show that 
\begin{equation}\label{cc-sum} \sum_{ \substack{ A, B \in \mathbb Z \\ \max ( 4 \abs{A}^3, 27 \abs{B}^2) \le X \\ p^4 \nmid A \textrm{ or } p^6\nmid B \textrm{ for all } p \\ A \equiv a \bmod q^4 \\ B\equiv b \bmod q^6 \\ 4 A^3 + 27B^2 \neq 0  }}\!\!\!\!\!\!\!\!\!\! 1  =  \begin{cases} 0 & \textrm{if }p^4\mid a \textrm{ and }p^6 \mid b \textrm{ for some } p \mid q \\ (1+o(1))   \frac{ 2 (X/4)^{1/3} \cdot 2 (X/27)^{1/2}  \prod_{p\nmid q} (1-p^{-10}) }{ q^{10} } &\textrm{otherwise}\end{cases}, \end{equation}
since dividing \eqref{cc-sum} by the special case $q=1$ of \eqref{cc-sum} gives the desired statement. Furthermore, we may ignore the condition $4A^3 + 27B^2 \neq 0$ because the number of solutions to this equation is $O(X^{1/6})$ and hence can be absorbed into the error term. 

Note that \eqref{cc-sum} is trivial if $p^4 \mid a$ and $p^6 \mid b$ for some $p\mid q$ as the sum on the left-hand side of \eqref{cc-sum} is empty. Otherwise, inclusion-exclusion and the fact that we need only check the $p^4 \nmid A, p^6 \nmid B$ condition on $p$ relatively prime to $q$ give 
\[ \sum_{ \substack{ A, B \in \mathbb Z \\ \max ( 4 \abs{A}^3, 27 \abs{B}^2) \le X \\ p^4 \nmid A \textrm{ or } p^6\nmid B \textrm{ for all } p \\ A \equiv a \bmod q^4 \\ B\equiv b \bmod q^6  }}  1 = \sum_{ \substack{ m < X^{1/6} \\ \gcd(m,q)=1}  }\mu(m) \sum_{ \substack{ A, B \in \mathbb Z \\ \max ( 4 \abs{A}^3, 27 \abs{B}^2) \le X \\ m^4 \mid A \\ m^6 \mid B \\ A \equiv a \bmod q^4 \\ B\equiv b \bmod q^6  } } 1\]
and we have 
\[  \sum_{ \substack{ A, B \in \mathbb Z \\ \max ( 4 \abs{A}^3, 27 \abs{B}^2) \le X \\ m^4 \mid A \\ m^6 \mid B \\ A \equiv a \bmod q^4 \\ B\equiv b \bmod q^6  } } 1 = \frac{ 2 (X/4)^{1/3} \cdot 2 (X/27)^{1/2}}{ m^{10 }q^{10}}  +  O \left( \frac{ X^{1/2}}{m^6 q^6} \right), \] since it is the product of the number of values of $A$ in an arithmetic progression of index $(mq)^4$ inside an interval of length $2 (X/4)^{1/3}$ with the number of values of $B$ in an arithmetic progression of index $(mq)^6$ inside an interval of length $2(X/27)^{1/2}$. By summing over $m$ we obtain \eqref{cc-sum}. \end{proof}

\begin{proof}[Proof of Proposition \ref{all} ] We begin by applying the Voronoi summation formula (Lemma \ref{elliptic-Voronoi}), taking $q=1$ and $a$ an arbitrary integer, which causes the $e\left(\frac{an}{q}\right)$ and $e\left(\frac{ \overline{a N(E)}n}{q}\right)$ terms to drop out. This yields
\[ \frac{1 }{  N(E)  }\sum_{n=1}^{\infty}W \left( \frac{n}{ N(E)} \right) \epsilon(E) a_n (E)  \]
\begin{equation}\label{key-n-sum} = \sum_{n=1}^{\infty}  \frac{a_n (E) }{\sqrt{n}}   \int_0^{\infty}W(u)  2\pi \sqrt{u} J_1( 4 \pi \sqrt{un } ) du .\end{equation} 
 Taking $m=3$ in Lemma \ref{Hankel-decreasing} and using $\frac{a_n (E) }{\sqrt{n}} \ll  n^{\epsilon}$ uniformly in $E$, we see that the sum \eqref{key-n-sum} is absolutely convergent in $n$ uniformly in $E$. This implies
\[ \lim_{X \to\infty}  \mathbb E_{ \{E:H(E)\le X\} }  \Bigl[  \frac{1 }{  N(E)  }\sum_{n=1}^{\infty}W \left( \frac{n}{ N(E)} \right) \epsilon(E) a_n (E)  \Bigr]\]
\[= \lim_{X \to\infty}  \mathbb E_{ \{E:H(E)\le X\} }  \Bigl[\sum_{n=1}^{\infty}  \frac{a_n(E)}{\sqrt{n}}   \int_0^{\infty}W(u)  2\pi \sqrt{u} J_1( 4 \pi \sqrt{un } ) du\Bigr] \]
\[ =  \lim_{X \to\infty} \sum_{n=1}^{\infty} \frac{1}{\sqrt{n}}  \int_0^{\infty}W(u)  2\pi \sqrt{u} J_1( 4 \pi \sqrt{un } ) du \mathbb E_{ \{E:H(E)\le X\} }  \left[ a_n (E)  \right] \]
\[ =   \sum_{n=1}^{\infty}  \frac{1}{\sqrt{n}} \int_0^{\infty}W(u)  2\pi \sqrt{u} J_1( 4 \pi \sqrt{un } ) du \lim_{X \to\infty}  \mathbb E_{ \{E:H(E)\le X\} }  \left[ a _n (E) \right], \]
since the absolute convergence is preserved under averaging and then lets us commute the sum with the limit. We now focus on computing the inner expectation \[ \lim_{X \to\infty}  \mathbb E_{ \{E:H(E)\le X\} }  \left[ a _n (E) \right]= \lim_{X \to\infty}  \mathbb E_{ \substack{A, B\in \mathbb Z \\ \max (4 \abs{A}^3, 27 \abs{B}^2) \le X\\ p^4 \nmid A \textrm{ or } p^6 \nmid B \textrm{ for all } p  \\ 4 A^3+27 B^2 \neq 0 } } \left[ a _n (E_{A,B}) \right] ,\]
 since each elliptic curve $E$ may be expressed uniquely as $E_{A,B}$ with $p^4 \nmid A$ or $p^6\nmid B$ for all~$p$, and by definition $H(E) =  \max (4 \abs{A}^3, 27 \abs{B}^2)$. 

We have $a_n(E) = \prod_{p \mid n } a_{p^{ v_p(n)}} (E_{A,B})$ and each $ a_{p^{ v_p(n)}} (E_{A,B})$ is a $p$-adically continuous function of $A,B$ that takes finitely many values and thus can be written as a linear combination of indicator functions of congruence classes modulo powers of $p$, so $a_n$ may be written as a linear combination of indicator functions of congruence classes modulo natural numbers.   We apply Lemma \ref{congruence-counting} and observe that summing such a function over congruence classes $a$ mod $q^4$ and $b$ mod $q^6$ with $p^4 \nmid a$ or $p^6\nmid b$  and then dividing by $q^{10} \prod_{p\mid q} (1-p^{-10})$ is equivalent to integrating the function on $\prod_{p} \mathbb Z_p^2 \setminus (p^4 \mathbb Z_p \times  p^6 \mathbb Z_p) $ against the product over $p$ of the uniform measure on $ \mathbb Z_p^2 \setminus (p^4 \mathbb Z_p \times  p^6 \mathbb Z_p) $ with total mass $1$, so that
\[\lim_{X\to\infty}  \mathbb E_{ \substack{ A, B \in \mathbb Z \\ \max ( 4 \abs{A}^3, 27 \abs{B}^2) \le X \\ p^4 \nmid A \textrm{ or } p^6\nmid B \textrm{ for all } p}}  \left[ a _n (E_{A,B} )   \right] = \int_{ \prod_p \mathbb Z_p^2 \setminus (p^4 \mathbb Z_p \times  p^6 \mathbb Z_p)  }  \prod_{p \mid n } a_{p^{ v_p(n)}} (E_{A,B}) = \prod_{p \mid n } \ell_{p,{v_p(n)}}, \] since the uniform measure on $\mathbb Z_p^2 \setminus (p^4 \mathbb Z_p \times  p^6 \mathbb Z_p) $ with total mass $1$ is $(1-p^{-10})$ times the restriction of the uniform measure on $\mathbb Z_p^2$ with total mass $1$. This implies the limit \eqref{main-limit} exists and is equal to 
\begin{equation}\label{main-alternate}   \sum_{n=1}^{\infty}  \frac{ \prod_{p \mid n } \ell_{p,{v_p(n)}} }{\sqrt{n}} \int_0^{\infty}W(u)  2\pi \sqrt{u} J_1( 4 \pi \sqrt{un } ) du.  \end{equation}

We next demonstrate that the sum  $\sum_{n=1}^{\infty}  \frac{ \prod_{p \mid n } \ell_{p,{ v_p(n)}} }{\sqrt{n}}   \sqrt{u} J_1( 4 \pi \sqrt{un } ) $ is absolutely convergent uniformly for $u$ in a compact interval in $(0,\infty)$. We first observe that $a_{p^\nu} \ll (p^\nu)^{\frac{1}{2}+\epsilon}$ so that $\ell_{p,\nu} \ll (p^\nu)^{\frac{1}{2}+\epsilon}$ and thus $ \frac{ \prod_{p \mid n } \ell_{p,{ v_p(n)}} }{\sqrt{n}}  \ll n^\epsilon$.  We furthermore have $J_1( 4 \pi \sqrt{un } )  \ll (4 \pi \sqrt{un})^{-1/2} \ll n^{-1/4}$. Finally,  since the lowest-weight cusp form of level $1$ has weight $12$, Lemma \ref{all-lf} implies that a necessary condition for $\prod_{p \mid n } \ell_{p,{ v_p(n)}} \neq 0$ is that every prime, with the possible exception of $3$, dividing $n$ must divide $n$ with multiplicity at least $10$. Such $10$-powerful away from $3$ numbers have density $\ll n^{-9/10}$. Thus, we are summing a term of size $n^{ -\frac{1}{4}+\epsilon}$ over a set of density $\ll n^{-9/10}$, which is absolutely convergent since $\frac{9}{10} + \frac{1}{4} >1$. 

This absolute convergence allows us to exchange the sum and the integral, showing that~\eqref{main-alternate} is equal to 
\[    \int_0^{\infty}W(u) \sum_{n=1}^{\infty}  \frac{ \prod_{p \mid n } \ell_{p,{v_p(n)}} }{\sqrt{n}} 2\pi \sqrt{u} J_1( 4 \pi \sqrt{un } ) du.\]  Lemma \ref{all-lf} implies that $\ell_{p,{v_p(n)}}$ vanishes when $v_p(n)$ is odd, so the product vanishes unless $n$ is a perfect square. We thus introduce the change of variables $n=m^2$, producing \eqref{main-evaluation}. \end{proof}

\begin{proof}[Proof of Proposition \ref{primes}] We proceed similarly to the proof of Proposition \ref{all}. We start the analysis with the contribution of a single elliptic curve. We begin by introducing normalized Fourier coefficients, then using additive characters $e(x) = e^{ 2\pi i x}$ to detect the conditions $p \nmid n$, before applying the Voronoi summation formula (Lemma \ref{elliptic-Voronoi}) and rearranging.  We have
\[\frac{\prod_{p \leq P} (1-1/p)^{-1}  }{  N(E)  }\sum_{\substack{ n \in \mathbb N \\ p \nmid n \textrm{ for } p \leq P}} W \left( \frac{n}{ N(E)} \right) \epsilon(E) a_n (E)  \]
\[ = \frac{1}{ N(E) } \sum_{n=1}^{\infty}  \sum_{ \substack{ q \in \mathbb N \\  P\textrm{-smooth}}} \sum_{ a \in (\mathbb Z/q\mathbb Z)^\times}  \frac{\mu(q)}{\phi(q)} e \left( \frac{a n}{q} \right)W \left( \frac{n}{ N(E)} \right) a_n (E)  \]
\[ = \!\sum_{ \substack{ q \in \mathbb N \\P\textrm{-smooth}}}  \frac{\mu(q)}{\phi(q)}\sum_{ a \in (\mathbb Z/q\mathbb Z)^\times}  \frac{1}{  N(E) } \sum_{n=1}^{\infty}  e \left( \frac{a n}{q} \right)W \left( \frac{n}{ N(E)} \right) \epsilon(E) a_n(E)\]
\[ = \sum_{ \substack{ q \in \mathbb N \\P\textrm{-smooth}}}  \frac{\mu(q)}{\phi(q)}\sum_{ a \in (\mathbb Z/q\mathbb Z)^\times} \frac{1}{q} \sum_{n=1}^{\infty} \frac{ a_n (E)}{\sqrt{n}} e \left( \frac{ \overline{aN(E)} n}{q} \right) \int_0^\infty  W(u) \sqrt{u}  2 \pi J_1\left( 4 \pi \frac{\sqrt{u n}}{q} \right) du \]
\[ =  \sum_{ \substack{ q \in \mathbb N \\P\textrm{-smooth}}}  \frac{\mu(q)}{q\phi(q)} \sum_{ a \in (\mathbb Z/q\mathbb Z)^\times}  \sum_{n=1}^{\infty} \frac{ a_n (E)}{\sqrt{n}} e \left( \frac{ a n}{q} \right) \int_0^\infty  2\pi W(u)  \sqrt{u} J_1\left( 4 \pi \frac{\sqrt{u n}}{q} \right)  du, \]
where $\overline{aN(E)}$ denotes the modular inverse of $aN(E)$ mod $q$, and in the last line we use the fact that $ a\mapsto \overline{aN(E)}$ is a bijection. For fixed $P$, only finitely many $q$ contribute to the sum over $q$ since they must be squarefree and $P$-smooth. The sum over $a$ is finite for each $q$. Thus absolute convergence of this sum follows from absolute convergence of the inner sum over $n$, which follows from Lemma \ref{Hankel-decreasing} with $m=3$. This lets us insert the evaluation of Ramanujan's sum in the squarefree case
\[ \sum_{ a \in (\mathbb Z/q\mathbb Z)^\times}  e \left( \frac{ a n}{q} \right) = \mu \left( \frac{q}{\gcd(n,q)}\right)  \phi( \gcd(n,q) ) \] to obtain
\[\frac{\prod_{p \leq P} (1-1/p)^{-1}  }{  N(E)  }\sum_{\substack{ n \in \mathbb N \\ p \nmid n \textrm{ for } p \leq P}} W \left( \frac{n}{ N(E)} \right) \epsilon(E) a_n (E)  \]
\[ = \sum_{ \substack{ q \in \mathbb N \\P\textrm{-smooth}\\ \textrm{squarefree}}}   \sum_{n=1}^{\infty} \frac{ \mu ( \gcd(n,q))}{q \phi \left( \frac{q}{\gcd(n,q)}\right)} \frac{ a_n (E)}{\sqrt{n}} \int_0^\infty  2\pi W(u) \sqrt{u}  J_1\left( 4 \pi \frac{\sqrt{u n}}{q} \right)  du,\]
which preserves the absolute convergence so we may write
\[ \lim_{X\to\infty}  \mathbb E_{ \substack{ \{E:H(E)\le X\} \\ p\nmid N(E)\textrm{ for } p \leq P}}  \Bigl[  \frac{\prod_{p \leq P} (1-1/p)^{-1}  }{  N(E)  }\sum_{\substack{ n \in \mathbb N \\ p \nmid n \textrm{ for } p \leq P}} W \left( \frac{n}{ N(E)} \right) \epsilon(E) a_n (E)  \Bigr]\]
\[ = \lim_{X\to\infty}  \mathbb E_{ \substack{ \{E:H(E)\le X\} \\ p\nmid N(E)\textrm{ for } p \leq P}}  \Bigl[\!\sum_{ \substack{ q \in \mathbb N \\P\textrm{-smooth}\\ \textrm{squarefree} } }\!\!\sum_{n=1}^{\infty}   \frac{ \mu ( \gcd(n,q))}{q \phi \left( \frac{q}{\gcd(n,q)}\right)}\frac{ a_n (E)}{\sqrt{n}} \int_0^\infty  2\pi W(u)  \sqrt{u} J_1\!\left( 4 \pi \frac{\sqrt{u n}}{q} \right)  du \Bigr]\]
\[= \lim_{X\to\infty}\!\!\!\! \sum_{ \substack{ q \in \mathbb N \\P\textrm{-smooth}\\ \textrm{squarefree}}}\!\!\!  \sum_{n=1}^{\infty}   \frac{ \mu ( \gcd(n,q))}{q \phi \left( \frac{q}{\gcd(n,q)}\right)}\frac{1}{\sqrt{n}} \int_0^\infty\!\!  2\pi W(u) \sqrt{u}  J_1\!\left(\! 4 \pi \frac{\sqrt{u n}}{q} \right)\!du\ \mathbb E_{ \substack{ \{E:H(E)\le X\} \\ p\nmid N(E)\textrm{ for } p \leq P}} \!\! \left[ a_n (E)\right]\]
\[ = \sum_{ \substack{ q \in \mathbb N \\P\textrm{-smooth}\\ \textrm{squarefree}}} \sum_{n=1}^{\infty} \frac{ \mu ( \gcd(n,q))}{q \phi \left( \frac{q}{\gcd(n,q)}\right)\sqrt{n}} \int_0^\infty 2\pi W(u)  \sqrt{u} J_1\!\left(\! 4 \pi \frac{\sqrt{u n}}{q} \right) du  \lim_{X\to\infty} \mathbb E_{ \substack{ \{E:H(E)\le X\} \\ p\nmid N(E)\textrm{ for } p \leq P}} \left[ a_n (E)\right].\]
 We now evaluate  \[ \lim_{X\to\infty} \mathbb E_{ \substack{ \{E:H(E)\le X\} \\ p\nmid N(E)\textrm{ for } p \leq P}}  \left[ a_n (E)\right] = \lim_{X \to\infty}  \mathbb E_{ \substack{A, B\in \mathbb Z \\ \max (4 \abs{A}^3, 27 \abs{B}^2) \le X\\ p^4 \nmid A \textrm{ or } p^6 \nmid B \textrm{ for all } p\\ 4 A^3+27 B^2 \neq 0 \\  p \nmid N(E_{A,B} )\textrm{ for }p\leq P  } } \left[ a _n (E_{A,B}) \right] ,\] This differs from the limit considered in Proposition \ref{all} only in the condition that $p\nmid N(E)$ for $p \leq P$, which amounts to a congruence condition modulo a product of primes $\leq P$. This has the effect of restricting the sum over congruence classes to the congruence classes with good reduction, and since the same restriction appears on the denominator, we obtain
\[ \lim_{X\to\infty} \mathbb E_{ \substack{ \{E:H(E)\le X\} \\ p\nmid N(E)\textrm{ for } p \leq P}}  \left[a_n (E)\right] = \prod_{\substack{ p \mid n \\ p \leq P}} \tilde{\ell}_{p,{v_p(n)}} \prod_{\substack{p \mid n\\ p > P}} \ell_{p,{v_p(n)}},\]
which implies
\begin{equation}\label{primes-P}\begin{aligned}
&\lim_{X\to\infty}  \mathbb E_{ \substack{ \{E:H(E)\le X\} \\ p\nmid N(E)\textrm{ for } p \leq P}}  \Bigl[  \frac{\prod_{p \leq P} (1-1/p)^{-1}  }{  N(E)  }\sum_{\substack{ n \in \mathbb N \\ p \nmid n \textrm{ for } p \leq P}} W \left( \frac{n}{ N(E)} \right) \epsilon(E) a_n (E)  \Bigr]\\
=&\sum_{ \substack{ q \in \mathbb N \\P\textrm{-smooth}\\ \textrm{squarefree}}}   \sum_{n=1}^{\infty} \frac{ \mu ( \gcd(n,q))}{q \phi \left( \frac{q}{\gcd(n,q)}\right)\sqrt{n}} \int_0^\infty  2\pi W(u)  \sqrt{u} J_1\left( 4 \pi \frac{\sqrt{u n}}{q} \right)  du  \prod_{\substack{ p \mid n \\ p \leq P}} \tilde{\ell}_{p,{v_p(n)}} \prod_{\substack{p \mid n\\ p > P}} \ell_{p,{v_p(n)}}.\end{aligned}\end{equation}
We next check that this sum and integral are together absolutely convergent uniformly in~$P$. We need to show that
\[\sum_{ \substack{q \in \mathbb N \\ \textrm{squarefree}}}\!\!\!\!  \sum_{n=1}^{\infty} \frac{1}{q \phi \left( \frac{q}{\gcd(n,q)}\right)\!\sqrt{n}}  \int_0^\infty \!\! 2\pi\abs{ W(u)} \sqrt{u}  \abs{J_1\!\left(\! 4 \pi \frac{\sqrt{u n}}{q} \right) } du\!   \prod_{p\mid n}\! \max\bigl( \abs{\ell_{p,{v_p(n)}}}, \abs{\tilde{\ell}_{p,{v_p(n)}}} \bigr)< \infty.\]
Inserting the bound $J_1\left( 4 \pi \frac{\sqrt{u n}}{q} \right) \ll \left( 4 \pi \frac{\sqrt{u n}}{q} \right)^{-1/2} $ and using the compact support of $W$, it suffices to prove that
\[\sum_{ \substack{q \in \mathbb N \\ \textrm{squarefree}}}  \sum_{n=1}^{\infty} \frac{1}{q \phi \left( \frac{q}{\gcd(n,q)}\right)\sqrt{n}}   \frac{ \sqrt{q}}{n^{\frac{1}{4}} } \prod_{p\mid n} \max\bigl( \abs{\ell_{p,{v_p(n)}}}, \abs{\tilde{\ell}_{p,{v_p(n)}}} \bigr)< \infty.\]
Every term is multiplicative in $q$ and $n$, so the sum splits as an Euler product, writing $q =\prod_p p^{d_p}$ and $n =\prod_p p^{\nu_p} $,
\[\sum_{ \substack{q \in \mathbb N \\ \textrm{squarefree}}}  \sum_{n=1}^{\infty} \frac{1}{q \phi \left( \frac{q}{\gcd(n,q)}\right)\sqrt{n}}   \frac{ \sqrt{q}}{n^{\frac{1}{4}}}  \prod_{p\mid n} \max\bigl( \abs{\ell_{p,{v_p(n)}}}, \abs{\tilde{\ell}_{p,{v_p(n)}}} \bigr)\]
\[ =\prod_{p \textrm{ prime} } \left( \sum_{d=0}^1 \sum_{\nu=0}^{\infty} \frac{1}{ p^{d}\phi \left( \frac{p^d}{\gcd(p^\nu, p^d)}\right) p^{\frac{\nu}{2}}  }\frac{ p^{\frac{d}{2}}}{ p^{\frac{\nu}{4}}} \max \bigl( \abs{\ell_{p,\nu}}, \abs{\tilde{\ell}_{p,\nu}} \bigr)\right)\]
\[ =\prod_{p \textrm{ prime} } \left( \sum_{d=0}^1 \sum_{\nu=0}^{\infty} \frac{1}{ p^{\frac{d}{2}} p^{\frac{\nu}{4}}  } \frac{ \max \bigl( \abs{\ell_{p,\nu}}, \abs{\tilde{\ell}_{p,\nu}} \bigr)}{ p^{\frac{\nu}{2}}}\begin{cases} \frac{1}{p-1} & \textrm{if } d=1 \textrm{ and } \nu=0 \\ 1 & \textrm{otherwise} \end{cases} \right).\]
To check that this product is finite, it suffices to check that each term in the product is finite and the terms are $1 + O ( p^{-3/2})$ for $p$ sufficiently large. Both $\ell_{p,\nu}$ and $\tilde{\ell}_{p,\nu}$ are bounded by  $(\nu+1) p^{ \frac{\nu}{2}}$ since they are averages of $a_{p^\nu}$ that satisfy that bound, and so the $p^{\frac{\nu}{4}}  $ in the denominator ensures the sum is finite.

The term $d=0,\nu=0$ contributes $1$, so it suffices to show the remaining terms are $O(p^{-3/2})$. The term $d=1,\nu=0$ contributes $\frac{1}{ \sqrt{p} (p-1)}= O( p^{-3/2})$. Beyond this, any term with $d=1$ is bounded by the corresponding term with $d=0$ so we may assume $d=0$. Then all terms with $\nu>6$ are handled by the $(\nu+1) p^{ \frac{\nu}{2}}$ bound for the local factors, so we may assume $\nu \leq 6$. In this range, since $p>3$ and there are no cusp forms of level $1$ and weight $\leq 8$, Lemma \ref{all-lf} gives $\ell_{p,\nu}=0$ and Lemma \ref{primes-lf} gives $\tilde{\ell}_{p,\nu} = \frac{ 1+ (-1)^\nu}{2p}$. In this case, \[\frac{1}{ p^{\frac{d}{2}} p^{\frac{\nu}{4}}  } \frac{ \max ( \abs{\ell_{p,\nu}}, \abs{\tilde{\ell}_{p,\nu}} )}{ p^{\frac{\nu}{2}}}= \frac{ 1+ (-1)^\nu} { 2 p^{ 1+ \frac{3\nu}{4}}} =O(p ^{- \frac{7}{4}}), \]  since $\nu \geq 1$, completing the proof of absolute convergence.

Using this absolute convergence, we may exchange the sum with the limit in $P$ and then rearrange, obtaining
\[ \lim_{P\to\infty} \lim_{X\to\infty}  \mathbb E_{ \substack{ \{E:H(E)\le X\} \\ p\nmid N(E)\textrm{ for } p \leq P}}  \Bigl[  \frac{\prod_{p \leq P} (1-1/p)^{-1}  }{  N(E)  }\sum_{\substack{ n \in \mathbb N \\ p \nmid n \textrm{ for } p \leq P}} W \left( \frac{n}{ N(E)} \right) \epsilon(E) a_n (E)  \Bigr]\]
\[=  \lim_{P \to \infty} \sum_{ \substack{ q \in \mathbb N \\P\textrm{-smooth}\\ \textrm{squarefree}}}   \sum_{n=1}^{\infty} \frac{ \mu ( \gcd(n,q))}{q \phi \left( \frac{q}{\gcd(n,q)}\right)\sqrt{n}} \int_0^\infty  2\pi W(u) \sqrt{u}  J_1\left( 4 \pi \frac{\sqrt{u n}}{q} \right)  du  \prod_{\substack{ p \mid n \\ p \leq P}} \tilde{\ell}_{p,{v_p(n)}} \prod_{\substack{p \mid n\\ p > P}} \ell_{p,{v_p(n)}}\]
\[= \sum_{ \substack{ q \in \mathbb N \\ \textrm{squarefree}}}   \sum_{n=1}^{\infty} \frac{ \mu ( \gcd(n,q))}{q \phi \left( \frac{q}{\gcd(n,q)}\right)\sqrt{n}} \int_0^\infty  2\pi W(u)\sqrt{u}  J_1\left( 4 \pi \frac{\sqrt{u n}}{q} \right)  du   \lim_{P \to \infty} \prod_{\substack{ p \mid n \\ p \leq P}} \tilde{\ell}_{p,{v_p(n)}} \prod_{\substack{p \mid n\\ p > P}} \ell_{p,{v_p(n)}}\]
\[=\sum_{ \substack{ q \in \mathbb N \\ \textrm{squarefree}}}   \sum_{n=1}^{\infty} \frac{ \mu ( \gcd(n,q))}{q \phi \left( \frac{q}{\gcd(n,q)}\right)\sqrt{n}} \int_0^\infty  2\pi W(u)  \sqrt{u} J_1\left( 4 \pi \frac{\sqrt{u n}}{q} \right)  du    \prod_{p\mid n } \tilde{\ell}_{p,{v_p(n)}} \]
\[ = \int_0^\infty W(u)  \sqrt{u}\Biggl( 2\pi \sum_{ \substack{ q \in \mathbb N \\ \textrm{squarefree}}}   \sum_{n=1}^{\infty}  \frac{ \mu ( \gcd(n,q))}{q \phi \left( \frac{q}{\gcd(n,q)}\right)} \frac{\prod_{p\mid n } \tilde{\ell}_{p,{v_p(n)}}}{\sqrt{n}} J_1\left( 4 \pi \frac{\sqrt{u n}}{q} \right) \Biggr) du . \]
Lemma \ref{all-lf} implies that $\tilde{\ell}_{p,{v_p(n)}}$ vanishes when $v_p(n)$ is odd, so the product vanishes unless~$n$ is a perfect square. We thus introduce the change of variables $n=m^2$, giving the statement of Proposition~\ref{primes}.
\end{proof}

\begin{remark} An argument almost identical to the first part of the above proof shows that for any finite set $S$ of primes we have
\[ \lim_{X\to\infty}  \mathbb E_{ \substack{ \{E:H(E)\le X\} \\ p\nmid N(E)\textrm{ for } p \in S}}  \Bigl[  \frac{\prod_{p \in S} (1-1/p)^{-1}  }{  N(E)  }\sum_{\substack{ n \in \mathbb N \\ p \nmid n \textrm{ for } p\in S}} W \left( \frac{n}{ N(E)} \right) \epsilon(E) a_n (E) \Bigr]\]
\[ =\sum_{ \substack{ q \in \mathbb N \\ p \nmid q \textrm{ for } p \notin S\\ \textrm{squarefree}}}   \sum_{n=1}^{\infty} \frac{ \mu ( \gcd(n,q))}{q \phi \left( \frac{q}{\gcd(n,q)}\right)\sqrt{n}} \int_0^\infty  2\pi W(u)  \sqrt{u} J_1\left( 4 \pi \frac{\sqrt{u n}}{q} \right)  du  \prod_{\substack{ p \mid n \\ p\in S}} \tilde{\ell}_{p,{v_p(n)}} \prod_{\substack{p \mid n\\ p \notin S}} \ell_{p,{v_p(n)}}.\] \end{remark}

Before proving Theorem \ref{prime}, we check that the sum appearing in it is absolutely convergent. 

\begin{lemma}\label{prime-absolute-convergence} The sum \[ \sum_{\substack{ q\in\mathbb N \\ \emph{squarefree} }} \sum_{\substack{m \in \mathbb N }} \frac{ \mu(\gcd(m,q))}{q m \phi\left( \frac{q}{ \gcd(m,q) } \right) }  J_1  \left(4 \pi  \frac{\sqrt{u}m }{ q}  \right)   \prod_{p\mid q} \hat{\ell}_{p,2 v_p(m)} \prod_{p\mid m, p\nmid q} \ell_{p,{2v_p(m)}} \ \] is absolutely convergent, with the convergence uniform in $u$ as long as $u$ is bounded away from $0$. \end{lemma}

\begin{proof} Since $u$ is bounded away from $0$, we have \[ J_1\left( 4 \pi \frac{\sqrt{u} m}{q} \right) \ll \left( 4 \pi \frac{\sqrt{u} m}{q} \right)^{-1/2} \ll  \left( \frac{q}{m} \right)^{1/2}.\] Hence it suffices to establish that the sum
\[  \sum_{\substack{ q\in\mathbb N \\ \textrm{squarefree} }} \sum_{\substack{m \in \mathbb N }} \frac{ 1}{q m \phi\left( \frac{q}{ \gcd(m,q) } \right) } \left( \frac{q}{m} \right)^{1/2}  \prod_{p\mid q} \abs{ \hat{\ell}_{p,2 v_p(m)}} \prod_{p\mid m, p\nmid q} \abs{\ell_{p,{2v_p(m)}} }  \]
is finite. 
This splits as an Euler product, with Euler factor at $p$ given by
\[  \sum_{d=0}^1 \sum_{v=0}^\infty \frac{1}{ p^d p^{v} \phi \left( \frac{ p^d}{ \gcd(p^d,p^v)}\right)} \frac{ p^{d/2}}{p^{v/2}} \begin{cases} \abs{\hat{\ell}_{p, 2\nu}} & \textrm{if } d=1 \\ \abs{\ell_{p,2\nu}} & \textrm{if }d=0 \end{cases}\]
\[ = \sum_{v=0}^\infty \frac{1}{ p^{ \frac{3v}{2}}}  \abs{\ell_{p,2v}} +  \frac{1}{p^{\frac{1}{2}} (p-1)} \abs{\hat{\ell}_{p,0}} + \sum_{v=1}^\infty \frac{1}{p^{\frac{1}{2} + \frac{3v}{2}}} \abs{\hat{\ell}_{p, 2v}}.\]
It suffices to check that this Euler factor is finite for all $p$ and $1 + O(p^{-\frac{3}{2}})$ for large $p$, as then the Euler product is finite.

The definitions of $\ell_{p,2v}$ and $\hat{\ell}_{p,2v}$ as integrals of $a_{p^{2v}}$ ensure that each is $O( (v+1) p^{v} ) $ for fixed $p$ and the $p^{\frac{3v}{4}}$ in the denominator ensures the sum for each $p$ is finite. 

So we may assume $p>3$. We have $\ell_{p,0}=1$ so this term gives $1$, and we must check the remaining terms are $O(p^{-\frac{3}{2}})$. We have $\hat{\ell}_{p,0}=O(1)$, so $\frac{1}{p^{\frac{1}{2}} (p-1)} \abs{\hat{\ell}_{p,0}}= O (p^{-3/2})$. 

The formula for $\ell_{p,2v}$ in terms of modular forms gives $\ell_{p,2v} = O(  v p^{ v- \frac{1}{2} })$ for $v>0$ since the number of modular forms is $O(v)$, each contributes at most $p^{ v + \frac{1}{2} } $ to the sum, and the sum is then multiplied by a rational function of size $O(p^{-1})$. This shows that $\frac{1}{ p^{ \frac{3v}{2}}}  \abs{\ell_{p,2v}} = O (v p^{ - \frac{v+1}{2}  })$. Summing this over $v\geq 1$ and noting that $\ell_{p,2v}=0$ for $v<5$, we find that the terms in $\sum_{v=0}^\infty \frac{1}{ p^{ \frac{3v}{2}}}  \abs{\ell_{p,2v}}$  with $v\neq 0$ sum to $O (p^{-3})$.

We have $\abs{\hat{\ell}_{p,2}}= O(1)$, so the $v=1$ term of $ \sum_{v=1}^\infty \frac{1}{p^{\frac{1}{2} + \frac{3v}{2}}} \abs{\hat{\ell}_{p, 2v}}$ is $O(p^{-2})$. For $v>1$, we have $\hat{\ell}_{p,2v} =\ell_{p,2v} + O(1)$ and we saw that $\ell_{p,2v} = O ( v p^{ v- \frac{1}{2}} ) $, which since $1 \leq v  p^{v - \frac{1}{2}}$ gives $\hat{\ell}_{p,2v} =O ( v p^{v- \frac{1}{2}} )$. Summing $ \frac{1}{p^{\frac{1}{2} + \frac{3v}{2}}} O( v p^{v - \frac{1}{2}})$ over $v \geq 2$,  we obtain $O ( p^{-2})$, which is at most $O(p^{-\frac{3}{2}})$ as desired.\end{proof}

\begin{proof}[Proof of Theorem \ref{prime}]The argument is similar to Proposition \ref{primes}, but more complicated. We start the analysis with the contribution of a single elliptic curve $E$ of conductor $N$. We begin by introducing normalized Fourier coefficients. We use different strategies to detect the conditions $p \nmid n$ for primes of good reduction (i.e. $p \nmid N$), multiplicative reduction (i.e. $p\mid N$ but $p^2\nmid N$), and additive reduction (i.e. $p^2\mid N$). For primes of additive reduction, we observe that $a_n=0$ for $p\mid n$ already and nothing needs to be done. For primes of multiplicative reduction, we use a simple inclusion-exclusion, subtracting off multiples of $p$, and taking advantage of the relation $a_{np}=a_p a_n$ that holds for these primes. For primes of good reduction, we again use additive characters. (It would be possible to use additive characters to detect all these primes and then use a more intricate form of the Voronoi summation formula, but the approach taken here seems simpler.)  We have
\[\frac{\prod_{p \leq P} (1-1/p)^{-1}  }{  N  }\sum_{\substack{ n \in \mathbb N \\ p \nmid n \textrm{ for } p \leq P}} W \left( \frac{n}{ N} \right) \epsilon(E) a_n (E )  \]
\[ = \frac{\prod_{p \leq P} (1-1/p)^{-1}  }{  N }\sum_{\substack{ n \in \mathbb N \\ p \nmid n \textrm{ for } p \leq P \textrm{ such that } p^2\nmid N}} W \left( \frac{n}{ N} \right) \epsilon(E) a_n (E )  \]
\[ = \frac{\prod_{p \leq P} (1-1/p)^{-1}  }{ N } \sum_{ \substack{q_m \in \mathbb N \\ P\textrm{-smooth} \\ \textrm{squarefree} \\ \gcd( q_m^2, N)= q_m}}  \mu(q_m) \sum_{\substack{ n \in \mathbb N \\ p \nmid n \textrm{ for } p \leq P \textrm{ such that } p\nmid N \\ q_m \mid n }} W \left( \frac{n}{ N} \right)  \epsilon(E)a_n (E )   \]
\[ = \frac{\prod_{p \leq P} (1-1/p)^{-1}  }{ N}  \sum_{ \substack{q_m \in \mathbb N \\ P\textrm{-smooth} \\ \textrm{squarefree} \\ \gcd( q_m^2, N)= q_m}}  \mu(q_m)  \sum_{\substack{ n \in \mathbb N \\ p \nmid n \textrm{ for } p \leq P \textrm{ such that } p\nmid N  }} W \left( \frac{q_m n }{ N} \right)  \epsilon(E)  a_{n q_m}  (E )  \]
\[ = \frac{\prod_{p \leq P} (1-1/p)^{-1}  }{ N}  \sum_{ \substack{q_m \in \mathbb N \\ P\textrm{-smooth} \\ \textrm{squarefree} \\ \gcd( q_m^2, N)= q_m}}  \mu(q_m)  a_{q_m}(E) \sum_{\substack{ n \in \mathbb N \\ p \nmid n \textrm{ for } p \leq P \textrm{ such that } p\nmid N  }} W \left( \frac{q_m n }{ N} \right) \epsilon(E) a_{n}  (E )   \]
\[ = \frac{\prod_{\substack{ p \leq P \\ p \mid N } }(1-\nicefrac{1}{p})^{-1}  }{  N }\!\!\!\!\!\!\!\!  \sum_{\substack{ q_g\in\mathbb N \\ P\textrm{-smooth} \\ \textrm{squarefree} \\ \gcd( q_g, N)=1}}  \sum_{ \substack{q_m \in \mathbb N \\ P\textrm{-smooth} \\ \textrm{squarefree} \\ \gcd( q_m^2, N)= q_m}}\!\!\!\!\!\!\!\!\!\!   \frac{ \mu(q_g q_m)   a_{q_m}(E) }{\phi(q_g)} \!\!\!\!\!\! \sum_{ a \in (\nicefrac{\mathbb Z}{q_g\mathbb Z})^\times}  \sum_{\substack{ n \in \mathbb N   }}\! e\! \left( \frac{ a n}{ q_g} \right)\!\! W\!\! \left( \frac{q_m n }{ N} \right)\!\! \epsilon(E) a_n(E).   \]

Applying the Voronoi summation formula (Lemma~\ref{elliptic-Voronoi}) followed by a change of variables replacing $u$ with $q_m^{-1}  u$ yields
\begin{align*} \frac{1}{N}
\sum_{\substack{ n \in \mathbb N }} \ & e \left( \frac{ a n}{ q_g} \right) W \left( \frac{q_m n }{ N} \right)\epsilon(E)   a_{n}  (E )\\
&= \frac{1}{q_g}  \sum_{n=1}^{\infty} \frac{  a_{n}  (E )}{\sqrt{n}}  e\left( \frac{\overline{a N} n}{q_g} \right) \int_0^\infty 2 \pi  W( q_m u ) \sqrt{u} J_1 \left( \frac{ 4 \pi \sqrt{ u n }} { q_g} \right) du\\
&=    q_g^{-1} q_m^{-3/2} \sum_{n=1}^{\infty} \frac{  a_{n}  (E )}{\sqrt{n}}  e\left( \frac{\overline{a N} n}{q_g} \right) \int_0^\infty 2 \pi  W( u ) \sqrt{u} J_1 \left( \frac{ 4 \pi \sqrt{ u n }} {\sqrt{q_m}  q_g} \right) du.
\end{align*}
Writing $\hat{W} (y) = \int_0^\infty 2 \pi  W( u ) \sqrt{u} J_1 \left( 4 \pi \sqrt{ u y} \right) du $ and plugging this in, followed by a Ramanujan sum evaluation, a change of variables substituting $n$ with $q_m^{-1} n $, and the combination $ q = q_g q_m$, we obtain
\[\frac{\prod_{p \leq P} (1-1/p)^{-1}  }{  N  }\sum_{\substack{ n \in \mathbb N \\ p \nmid n \textrm{ for } p \leq P}} W \left( \frac{n}{ N} \right) \epsilon(E) a_n (E )  \]
\[=  \Bigl(\prod_{\substack{ p \leq P\\ p \mid N } } (1-1/p)^{-1}  \Bigr)\!\!\!\!\sum_{\substack{ q_g\in\mathbb N \\ P\textrm{-smooth} \\ \textrm{squarefree} \\ \gcd( q_g, N)=1}}  \sum_{ \substack{q_m \in \mathbb N \\ P\textrm{-smooth} \\ \textrm{squarefree} \\ \gcd( q_m^2, N)= q_m}}   \!\!\!\!\!\!\!\frac{ \mu(q_g q_m)   a_{q_m}(E) }{q_g q_m^{3/2} \phi(q_g)}\!\!  \sum_{ a \in (\mathbb Z/q_g\mathbb Z)^\times}  \sum_{\substack{ n \in \mathbb N   }}\frac{  a_{n}  (E )}{\sqrt{n}}  e\!\left( \frac{\overline{a N} n}{q_g} \right)\! \hat{W}\! \left( \frac{n}{ q_m q_g^2} \right) \]
\[=\Bigl(\prod_{\substack{ p \leq P\\ p \mid N } } (1-1/p)^{-1}  \Bigr)  \sum_{\substack{ q_g\in\mathbb N \\ P\textrm{-smooth} \\ \textrm{squarefree} \\ \gcd( q_g, N)=1}}  \sum_{ \substack{q_m \in \mathbb N \\ P\textrm{-smooth} \\ \textrm{squarefree} \\ \gcd( q_m^2, N)= q_m}}   \sum_{\substack{ n \in \mathbb N   }}  \frac{ \mu(\gcd(n,q_g) q_m)    }{q_g q_m^{3/2} \phi\left( \frac{q_g}{ \gcd(n,q_g) } \right) }  \frac{  a_{n q_m }  (E )}{\sqrt{n}}  \hat{W} \left( \frac{n}{ q_m q_g^2} \right) \]
\[= \Bigl(\prod_{\substack{ p \leq P\\ p \mid N } } (1-1/p)^{-1}  \Bigr) \sum_{\substack{ q_g\in\mathbb N \\ P\textrm{-smooth} \\ \textrm{squarefree} \\ \gcd( q_g, N)=1}}  \sum_{ \substack{q_m \in \mathbb N \\ P\textrm{-smooth} \\ \textrm{squarefree} \\ \gcd( q_m^2, N)= q_m}}   \sum_{\substack{ n \in \mathbb N \\ q_m \mid n   }}  \frac{ \mu(\gcd(n,q_g) q_m)    }{q_g q_m \phi\left( \frac{q_g}{ \gcd(n,q_g) } \right) }  \frac{  a_{n}  (E )}{\sqrt{n}}   \hat{W} \left( \frac{n}{ q_m^2 q_g^2} \right). \]
\[ =\Bigl(\prod_{\substack{ p \leq P\\ p \mid N } } (1-1/p)^{-1}  \Bigr)  \sum_{\substack{ q\in\mathbb N \\ P\textrm{-smooth} \\ \textrm{squarefree} \\ \gcd( q^2, N) \mid q }}   \sum_{\substack{ n \in \mathbb N \\ \gcd(q,N) \mid n   }}  \frac{ \mu(\gcd(n,q)  )  }{ q \phi\left( \frac{q}{ \gcd(n,q) } \right) }  \frac{  a_{n}  (E )}{\sqrt{n}}   \hat{W} \left( \frac{n}{ q^2} \right)\]
\[ = \sum_{\substack{ q\in\mathbb N \\ P\textrm{-smooth} \\ \textrm{squarefree} }} \sum_{\substack{n \in \mathbb N }} \frac{ \mu(\gcd(n,q))}{\sqrt{n}q  \phi\left( \frac{q}{ \gcd(n,q) } \right) }  \operatorname{ LT}_{n,q} (E)  \hat{W} \left( \frac{n}{ q^2} \right),\] where
\[ \operatorname{LT}_{n,q}(E) =  \begin{cases} \Bigl(\prod_{\substack{ p \leq P\\ p \mid N } } (1-1/p)^{-1}  \Bigr)  a_n(E) & \textrm{if }\gcd(q^2,N)\mid q \textrm{ and } \gcd(q,N) \mid n  \\ 0 & \textrm{otherwise} \end{cases}.\]
The point of this final expression is that $\operatorname{LT}_{n,q}(E)$ is the only part that depends on the elliptic curve~$E$.

We still have absolute convergence of the sum, so we may write
\[ \lim_{X\to\infty}  \mathbb E_{ \{E:H(E)\le X\} }  \Bigl[  \frac{\prod_{p \leq P} (1-1/p)^{-1}  }{  N(E)  }\sum_{\substack{ n \in \mathbb N \\ p \nmid n \textrm{ for } p \leq P}} W \left( \frac{n}{ N(E)} \right) \epsilon(E) a_n (E)  \Bigr]\]
\[ = \lim_{X\to\infty}  \mathbb E_{ \{E:H(E)\le X\} }  \Bigl[  \sum_{\substack{ q\in\mathbb N \\ P\textrm{-smooth} \\ \textrm{squarefree} }} \sum_{\substack{n \in \mathbb N }} \frac{ \mu(\gcd(n,q))}{q \sqrt{n} \phi\left( \frac{q}{ \gcd(n,q) } \right) }  \operatorname{ LT}_{n,q} (E)  \hat{W} \left( \frac{n}{ q^2} \right) \Bigr] \]
\[ =  \sum_{\substack{ q\in\mathbb N \\ P\textrm{-smooth} \\ \textrm{squarefree} }} \sum_{\substack{n \in \mathbb N }} \frac{ \mu(\gcd(n,q))}{ q\sqrt{n} \phi\left( \frac{q}{ \gcd(n,q) } \right) }   \hat{W} \left( \frac{n}{ q^2}  \right) \lim_{X\to\infty}  \mathbb E_{ \{E:H(E)\le X\} }  \Bigl[  \operatorname{ LT}_{n,q} (E) \Bigr] .\]
We also have
\[   \lim_{X\to\infty}  \mathbb E_{ \{E:H(E)\le X\} }  \Bigl[  \operatorname{ LT}_{n,q} (E) \Bigr]  = \lim_{X \to\infty}  \mathbb E_{ \substack{A, B\in \mathbb Z \\ \max (4 \abs{A}^3, 27 \abs{B}^2) \le X\\ p^4 \nmid A \textrm{ or } p^6 \nmid B \textrm{ for all } p } } \Bigl[ \operatorname{ LT}_{n,q} (E_{A,B})  \Bigr].\]

Now $\operatorname{LT}_{n,q}(E_{A,B}) $ is a product over the set of primes $p$ such that $p \leq P$ or $p \mid n$ of a  $p$-adically continuous function of $A,B$. This is because $a_n(E_{A,B})$ is such a $p$-adically continous function, as is the function that is $(1-1/p)^{-1}$ if $p\mid N$ and $0$ otherwise. Furthermore the function which is $1$ if $\gcd(q^2, N) \mid q$ and $\gcd(q,N)\mid n$ and $0$ otherwise can be written as a product of $p$-adically continuous characteristic functions, with the function for $p$ checking that $v_p( \gcd(q^2,N)) \leq v_p(q)$ and $v_p( \gcd(q,N))\leq v_p(n)$.

This leads to
\[  \lim_{X\to\infty}  \mathbb E_{ \{E:H(E)\le X\} }  \Bigl[  \operatorname{ LT}_{n,q} (E) \Bigr] = \prod_{p \leq P} \ell^*_{p,v_p(n), v_p(q)} \prod_{p > P, p\mid n} \ell_{p,{v_p(n)}}, \]
where $\ell^*_{p, \nu, \gamma}$ is defined as $(1- p^{-10})^{-1}$ times the integral of the relevant $p$-adically continuous function, that is
\[ \ell^*_{p, \nu,0} \colonequals  (1-p^{-10} )^{-1} \int_{ \substack{A, B\in \mathbb Z_p \\ p^4\nmid A \textrm{ or } p^6 \nmid B}}  a_{p^\nu} (E_{A,B})  \cdot \begin{cases} (1-1/p )^{-1} & \textrm{if } p \mid N(E_{A,B}) \\ 1 & \textrm{if } p\nmid N(E_{A,B}) \end{cases} \]
\[ \ell^*_{p, 0,1} = (1-p^{-10} )^{-1}   \int_{ \substack{A, B\in \mathbb Z_p \\ p^4\nmid A \textrm{ or } p^6 \nmid B \\ p \nmid N(E_{A,B}) }} 1 \]
and for $\nu>0$
\[ \ell^*_{p,\nu,1} \colonequals  (1-p^{-10} )^{-1}  \int_{ \substack{A, B\in \mathbb Z_p \\ p^4\nmid A \textrm{ or } p^6 \nmid B \\ p^2 \nmid N(E_{A,B}) }} a_{p^\nu} (E_{A,B})\cdot \begin{cases} (1-1/p )^{-1} & \textrm{if } p \mid N(E_{A,B}) \\ 1 & \textrm{if } p\nmid N(E_{A,B}) \end{cases}. \]

Thus we have
\[ \lim_{X\to\infty}  \mathbb E_{ \{E:H(E)\le X\}}  \Bigl[  \frac{\prod_{p \leq P} (1-1/p)^{-1}  }{  N(E)  }\sum_{\substack{ n \in \mathbb N \\ p \nmid n \textrm{ for } p \leq P}} W \left( \frac{n}{ N(E)} \right) \epsilon(E) a_n (E)  \Bigr]\]
\[ =  \sum_{\substack{ q\in\mathbb N \\ P\textrm{-smooth} \\ \textrm{squarefree} }} \sum_{\substack{n \in \mathbb N }} \frac{ \mu(\gcd(n,q))}{q \sqrt{n}  \phi\left( \frac{q}{ \gcd(n,q) } \right) }  \hat{W} \left( \frac{n}{ q^2}  \right)   \prod_{p \leq P} \ell^*_{p,v_p(n), v_p(q)} \prod_{p > P, p\mid n} \ell_{p,{v_p(n)}}. \]
We now make an adjustment to simplify the expression of the local factors. We observe that adding an arbitrary real number $\Delta$ to $\ell^*_{p, \nu,0}$ and adding  $p^2\Delta$ to $\ell^*_{p, \nu+2,1}$ does not affect the sum. Indeed, $\ell^*_{p,\nu,0}$ contributes to the summand associated to pairs $n,q$ with $p\nmid q$, and increasing $\ell^*_{p,\nu,0}$ by $\Delta$ increases that term by
\[ \frac{ \mu(\gcd(n,q))}{ q\sqrt{n} \phi\left( \frac{q}{ \gcd(n,q) } \right) }   \hat{W} \left( \frac{n}{ q^2}  \right) \Delta  \prod_{p'<P,  p'\neq p} \ell^*_{p',v_{p'} (n), v_{p'} (q)} \prod_{p' \geq P, p'\mid n} \ell_{p',{v_{p'}(n)}}, \]
while increasing $\ell^*_{p,\nu+2,1}$ by $p^2\Delta$ increases the term $(np^2, qp)$ by 
\[ \frac{ \mu(\gcd(np^2,qp))}{  pq \sqrt{np^2} \phi\left( \frac{qp }{ \gcd(np^2,qp) } \right) }  \hat{W} \left( \frac{np^2 }{ (qp)^2}  \right)  p^2\Delta  \prod_{p'<P,  p'\neq p} \ell^*_{p',v_{p'} (np^2), v_{p'} (qp)} \prod_{p' \geq P, p'\mid n} \ell_{p',{v_{p'}(np^2)}}, \]
and these two contributions cancel since $v_{p'} (np^2)= v_{p'}(n), v_{p'} (qp) = v_{p'}(n) $, $ \frac{np^2 }{ (qp)^2 } =  \frac{n}{q^2}$.  Since $q$ is prime to $p$ we have $\gcd(np^2,qp)= p \gcd(n,q)$, thus $\frac{qp }{ \gcd(np^2,qp) } = \frac{q}{ \gcd(n,q)}$ and $\mu(\gcd(np^2, qp))= - \mu(\gcd(n,q))$.

We now apply this to \[ \Delta = (1-p^{-10} )^{-1} \int_{ \substack{A, B\in \mathbb Z_p \\ p^4\nmid A \textrm{ or } p^6 \nmid B}}a_{p^\nu} (E_{A,B}) \cdot \begin{cases} 1- (1-1/p )^{-1} & \textrm{if } p \mid N(E_{A,B}) \\ 0 & \textrm{if } p\nmid N(E_{A,B}) \end{cases}.\]

This allows us to replace $\ell^*_{p, \nu,0}$ with $\ell_{p,\nu}$ and $\ell^*_{p,\nu,1}$ with $\hat{\ell}_{p, \nu}$, defined as 
\[ \hat{\ell}_{p, 0 } \colonequals (1-p^{-10} )^{-1}   \int_{ \substack{A, B\in \mathbb Z_p \\ p^4\nmid A \textrm{ or } p^6 \nmid B \\ p \nmid N(E_{A,B}) }} 1 \]
\[ \hat{\ell}_{p,1} \colonequals (1-p^{-10} )^{-1}  \int_{ \substack{A, B\in \mathbb Z_p \\ p^4\nmid A \textrm{ or } p^6 \nmid B \\ p^2 \nmid N(E_{A,B}) }} a_p(E_{A,B}) \cdot \begin{cases}  (1-1/p )^{-1} & \textrm{if } p \mid N(E_{A,B}) \\ 1 & \textrm{if } p\nmid N(E_{A,B}) \end{cases}\]
and for $\nu \geq 2$ 
\[ \hat{\ell}_{p,\nu } \colonequals (1-p^{-10} )^{-1}\!\! \int_{ \substack{A, B\in \mathbb Z_p \\ p^4\nmid A \textrm{ or } p^6 \nmid B  }}\!\!\! \cdot \begin{cases}   \frac{ a_{p^\nu} (E_{A,B}) }{1-\nicefrac{1}{p}}  + p^2  a_{p^{\nu-2} } (E_{A,B})  ( 1\!-\!(1\!-\!\nicefrac{1}{p})^{-1} ) & \textrm{if } p\! \mid\! N(E_{A,B}) \\ a_{p^\nu} (E_{A,B}) & \textrm{if } p\!\nmid\! N(E_{A,B})\end{cases}.\]
If $E_{A,B}$ has multiplicative reduction at $p$, then $a_{p^{\nu-2}} = a_{p^\nu}$ and we have
\begin{align*}
 \frac{a_{p^\nu} (E_{A,B})  }{1-\nicefrac{1}{p} } \!+  a_{p^{\nu-2} } (E_{A,B}) p^2 ( 1 - (1\!-\nicefrac{1}{p})^{-1} ) &= a_{p^\nu} \left( (1-\nicefrac{1}{p})^{-1}\!+ p^2 ( 1- (1\!-\nicefrac{1}{p})^{-1} ) \right)\\
&= - p a_{p^\nu},
\end{align*}
while if $E_{A,B}$ has additive reduction at $p$ then we have $a_{p^{\nu-2}} =0$ for $\nu>2$. This gives the simpler formula
\[ \hat{\ell}_{p,2} =  (1-p^{-10} )^{-1}  \int_{ \substack{A, B\in \mathbb Z_p \\ p^4\nmid A \textrm{ or } p^6 \nmid B  }} \cdot \begin{cases} a_{p^2} & \textrm{if }E_{A,B}\textrm{ has good reduction} \\ - p a_{p^2}  & \textrm{if }E_{A,B}\textrm{ has multiplicative reduction} \\  - \frac{p^2}{p-1} & \textrm{if }E_{A,B} \textrm{ has additive reduction} \end{cases} \]
and for $\nu>2$
\[\hat{\ell}_{p,\nu} = (1-p^{-10} )^{-1}  \int_{ \substack{A, B\in \mathbb Z_p \\ p^4\nmid A \textrm{ or } p^6 \nmid B \\ p^2 \nmid N(E_{A,B}) }}  a_{p^\nu} (E_{A,B}) \cdot \begin{cases} 1 & \text{if } p \nmid N(E_{A,B}) \\ -p & \textrm{if } p \mid N(E_{A,B}) \end{cases} .\]
This agrees with Definition \ref{lphat}. We thus obtain
\begin{equation} \begin{aligned}& \lim_{X\to\infty}  \mathbb E_{ \{E:H(E)\le X\}}  \Bigl[  \frac{\prod_{p \leq P} (1-1/p)^{-1}  }{  N(E)  }\sum_{\substack{ n \in \mathbb N \\ p \nmid n \textrm{ for } p \leq P}} W \left( \frac{n}{ N(E)} \right) \epsilon(E) a_n (E)  \Bigr]\\
=&  \sum_{\substack{ q\in\mathbb N \\ P\textrm{-smooth} \\ \textrm{squarefree} }} \sum_{\substack{n \in \mathbb N }} \frac{ \mu(\gcd(n,q))}{q \sqrt{n} \phi\left( \frac{q}{ \gcd(n,q) } \right) }  \hat{W} \left( \frac{n}{ q^2}  \right)   \prod_{p\mid q} \hat{\ell}_{p, v_p(n)} \prod_{p\mid n, p\nmid q} \ell_{p,{v_p(n)}}.  \end{aligned}\end{equation}
Lemmas \ref{all-lf} and \ref{prime-lf} imply that $\ell_{p,{v_p(n)}}$ and $\hat{\ell}_{p,{v_p(n)}}$ vanish when $v_p(n)$ is odd, so the product vanishes unless $n$ is a perfect square. We thus introduce the change of variables $n=m^2$, obtaining
\begin{equation}\label{prime-P} \begin{aligned}& \lim_{X\to\infty}  \mathbb E_{ \{E:H(E)\le X\}}  \Bigl[  \frac{\prod_{p \leq P} (1-1/p)^{-1}  }{  N(E)  }\sum_{\substack{ n \in \mathbb N \\ p \nmid n \textrm{ for } p \leq P}} W \left( \frac{n}{ N(E)} \right) \epsilon(E) a_n (E)  \Bigr]\\
=&  \sum_{\substack{ q\in\mathbb N \\ P\textrm{-smooth} \\ \textrm{squarefree} }} \sum_{\substack{m \in \mathbb N }} \frac{ \mu(\gcd(m,q))}{q m \phi\left( \frac{q}{ \gcd(m,q) } \right) }  \hat{W} \left( \frac{m^2}{ q^2}  \right)   \prod_{p\mid q} \hat{\ell}_{p, 2v_p(m)} \prod_{p\mid m, p\nmid q} \ell_{p,{2v_p(m)}}.  \end{aligned}\end{equation}
 Since \[\hat{W} \left(\frac{m^2}{q^2} \right) = \int_0^\infty 2 \pi  W( u ) \sqrt{u} J_1 \left( 4 \pi  \frac{\sqrt{ u }m}{q} \right) du\] is an integral of $J_1\! \left(\! 4 \pi  \frac{\sqrt{ u }m}{q} \right) $ against a bounded function on a compact set, Lemma \ref{prime-absolute-convergence} implies that the sum above is absolutely convergent. Hence when we take $\lim_{P \to \infty}$ of each side, on the right-hand side we simply remove the $P$-smooth condition. Then the same absolute convergence allows us to exchange the sum with the integral in the definition of $\hat{W}$, yielding the statement of Theorem~\ref{prime}. \end{proof}

We now turn to the calculation of the local factors $\ell_{p,\nu},$  $\tilde{\ell}_{p,\nu}$, and $\hat{\ell}_{p,\nu}$. We begin with some relatively standard but lengthy preparatory lemmas, the first expressing the measure of the set of $p$-adic elliptic curves whose reduction mod $p$ is a given mod $p$ elliptic curve and the second evaluating a sum over mod $p$ elliptic curves in terms of modular forms. From these, our formulas for $\ell_{p,\nu}$, $\tilde{\ell}_{p,\nu}$, $\hat{\ell}_{p,\nu}$ are obtained by summing over elliptic curves.

We first introduce some convenient notation for mod $p$ elliptic curves. Let $\mathcal M_{1,1}(\mathbb F_p)$ be the set of isomorphism classes of elliptic curves over $\mathbb F_p$, and let $\overline{\mathcal M}_{1,1}(\mathbb F_p)$ be the set of isomorphism classes of curves over $\mathbb F_p$ that are either an elliptic curve or a nodal cubic. (The notation arises from the Deligne-Mumford stacks $\mathcal M_{1,1}$ and $\overline{\mathcal M}_{1,1}$, respectively the moduli spaces of smooth curves of genus one with a marked point and stable curves of genus one with a marked point, but we do not need their stack structure, only their sets of $\mathbb F_p$-points, which can be described in an elementary way.)

We define $a_{p^\nu}(E)$ for $E\in \overline{\mathcal M}_{1,1}(\mathbb F_p)$, as $p^{ \frac{\nu}{2}} U_\nu \left( \frac{p+1 - \abs{E (\mathbb F_p)}}{2\sqrt{p}}\right) $ if $E$ is smooth, $1^\nu$ if $E$ is nodal with smooth locus a split form of $\mathbb G_m$, and $(-1)^\nu$ if $E$ is nodal with smooth locus a non-split form of $\mathbb G_m$.  It is clear from the definitions that $a_{p^\nu}(E)= a_{p^\nu}(E')$ where $E'$ is any elliptic curve over $\mathbb Q_p$ whose reduction mod $p$ is given by $E$.

\begin{lemma}\label{lf-underlying} Fix a prime $p$ and $E \in \overline{\mathcal M}_{1,1}(\mathbb F_p)$.  The measure of the set of pairs $A,B \in \mathbb Z_p^2$ such that $E_{A,B}$ has reduction mod $p$ isomorphic to $E$ and either $p^4\nmid A$ or $p^6\nmid B$, according to the uniform measure on $\mathbb Z_p^2$ with total mass $1$, is
\[ \frac{ p^{-1} -p^{-2}}{ \abs{ \Aut(E)}} \cdot \begin{cases} 1 & \textrm{if } p >3 \\ 3 & \textrm{if } p=3,  E \textrm{ supersingular} \\ 3^{-9} & \textrm{if } p=3, E \textrm{ ordinary} \\ 2^{-8} & \textrm{if } p=2\end{cases} .\]
\end{lemma}

\begin{proof} Note that a curve over $\mathbb Q_p$ has reduction mod $p$ isomorphic to $E$ if and only if it is $\mathbb Q_p$-isomorphic to a curve over $\mathbb Z_p$ whose mod $p$ fiber is isomorphic to $E$.

We first handle the case $p>3$.  The key point in this case is that any elliptic curve over $\mathbb Z_p$ may be put in the form $y^2 = x^3 + Ax + B$ over $\mathbb Z_p$, and any curve over $\mathbb Q_p$ may be put in the form $y^2 =x^3+Ax +B$ in a unique (with $p^4\nmid A$ or $p^6 \nmid B$ way) so the curve $E_{A,B}$ has reduction modulo $p$ isomorphic to $E$ if and only if the equation $y^2=x^3+Ax + B$, reduced modulo $p$, defines a curve isomorphic to $E$. Furthermore, any isomorphism between two curves over $\mathbb F_p$ with equations of the form $y^2 =x^3 +A x+ B$ has the form $ x \mapsto \lambda^2 x, y \mapsto \lambda^3 y$ for some $\lambda \in \mathbb F_p^\times$. It follows that the number of pairs $A, B \in \mathbb F_p$ defining a curve isomorphic to $E \in \overline{\mathcal M}_{1,1}(\mathbb F_p)$ is $(p-1) / \abs{\Aut(E)}$ so the $p$-adic measure of curves with reduction mod $p$ isomorphic to $E$ is $\frac{ (p-1) / \abs{\Aut(E)}}{ p^2} =  (p^{-1} - p^{-2} ) \frac{1}{ \abs{\Aut(E)}}$.

We next handle the case $p=3$ and finally $p=2$, which we do by a series of lemmas. We always work with the uniform $p$-adic measure on $\mathbb Q_p$ that assigns $\mathbb Z_p$ mass one.

\begin{lemma}\label{3-local-criterion} For $E \in \overline{\mathcal M}_{1,1}(\mathbb F_3)$, the curve $y^2 =x^3+Ax+B$ has reduction mod $3$ isomorphic to $E$ if and only if there exists $\lambda \in 3^{\mathbb Z}$ and $r \in \mathbb Q_3$ such that $y^2 = (x+r)^3 + \lambda^4 A (x+r) + \lambda^6 B$ is a polynomial with coefficients in $\mathbb Z_3$, and reduced mod $3$ defines $E$. \end{lemma}

We  will write $A' $ for $ \lambda^4 A$ and $B'$ for $ \lambda^6 B$.

\begin{proof} Any elliptic curve over $\mathbb Z_3$ may be put in the form $y^2=x^3 + a_2 x^2 + a_4 x + a_6$ with $a_2,a_4,a_6 \in \mathbb Z_3$, so a curve has reduction isomorphic to $E$ if and only if it can be put in that form with $y^2=x^3 + a_2 x^2 + a_4 x + a_6$ defining $E$ modulo $3$. Any isomorphism between the curves $y^2 =x^3 +A x+B$ and $y^2=x^3 + a_2 x^2 + a_4 x + a_6$ is necessarily a linear change of variables of the form $x \mapsto \lambda^{-2}( x + r) $ for $r \in \mathbb Q_3$ and $\lambda \in \mathbb Q_3^\times$, and because multiplying $x$ by a unit preserves the property that the reduction is isomorphic to $E$, we may assume $\lambda \in 3^{\mathbb Z}$. \end{proof}

\begin{lemma}\label{3-local-balanced} Fix  $E \in \overline{\mathcal M}_{1,1}(\mathbb F_3)$. The measure of the set of triples $a_2,a_4,a_6 \in \mathbb Z_3$ such that mod $3$ the equation $y^2 = x^3 + a_2 x^2 + a_4 x+ a_6$ defines a curve isomorphic to $E$ is $(3^{-1} - 3^{-2}) \frac{1}{\abs{\Aut(E)}}$. \end{lemma}

\begin{proof} Any isomorphism between two curves over $\mathbb F_3$ with equations of the form $y^2 = x^3 + a_2 x^2 + a_4 x+ a_6$ has the form $x\mapsto \lambda^2 x+s , y\mapsto \lambda^3 y$ for some $\lambda \in \mathbb F_3^\times, s\in \mathbb F_3$. It follows that the number of triples $a_2,a_4,a_6$ defining a curve isomorphic to $E \in \overline{\mathcal M}_{1,1}(\mathbb F_3)$ is $3(3-1) / \abs{\Aut(E)}$ so the $3$-adic measure of the set of triples $a_2,a_4,a_6\in\mathbb Z_3$  with $y^2=x^3+a_2x^2+a_4 x+a_6$ isomorphic to $E$ is $ \frac{ 3(3-1)/\abs{\Aut(E)}}{3^3}= (3^{-1} - 3^{-2}) \frac{1}{\abs{\Aut(E)}}$. \end{proof}

\begin{lemma}\label{3-bijection} The map $\mathbb Q_3^3 \to \mathbb Q_3^3$ that sends $(a_2,A',B') $ to $(a_2,a_4,a_6)$ where  $x^3 + a_2 x^2 + a_4 x + a_6 = (x+ a_2/3)^3 + A' (x+a_2/3) + B'$ is a measure-preserving bijection.\end{lemma}

 \begin{proof} It is easy to check that there is a unique $a_4,a_6$ for each $A,B$ and vice versa so this is indeed a bijection. Moreover, $A'$ is equal to $a_4$ plus a polynomial in $a_2$ and $B'$ is equal to $a_6$ plus a polynomial in $a_2, a_4$ so this bijection preserves the $3$-adic measure (for example, since its Jacobian is upper-triangular unipotent and hence has determinant $1$).\end{proof}
 
\begin{lemma}\label{3-scaled} Fix  $E \in \overline{\mathcal M}_{1,1}(\mathbb F_3)$. The measure of the set of $A',B' \in \mathbb Q_3^2$ such that there exists $r\in \mathbb Q_3$ with $(x+r)^3 + A' (x+r) + B'$ a polynomial with coefficients in $\mathbb Z_3$ whose reduction mod $3$ gives an elliptic curve isomorphic to $E$, is $(1-3^{-1})  \frac{1}{\abs{\Aut(E)}}$. \end{lemma}
 
 \begin{proof}  Lemma \ref{3-local-balanced} together with the measure-preserving bijection of Lemma \ref{3-bijection} together imply that the measure of the set of $(a_2,A',B')$ such that $y^2 = (x+ a_2/3)^3 + A' (x+a_2/3) + B'$ is a polynomial with coefficients in $\mathbb Z_3$ whose reduction mod $3$ gives an elliptic curve isomorphic to $E$ is $(3^{-1} - 3^{-2}) \frac{1}{\abs{\Aut(E)}}$.  For each $A',B'$, if any $a_2$ satisfies this condition, then the measure of the set of $a_2$ that do is $3^{-1}$: Given an $a_2$ satisfying the condition, any other solution must differ by a multiple of $3$, and everything that differs by a multiple of $3$ gives a solution.
 
 So the measure of the set of triples $(a_2,A', B')$ satisfying this condition is $3^{-1}$ times the measure of the set of pairs $A',B'$ where at least one $a_2$ exists, giving the statement.\end{proof}
 
 We now conclude the $p=3$ case. We see from Lemma \ref{3-local-criterion} that the set of $A,B$ with $y^2=x^3+Ax +B$ having reduction isomorphic to $E$ is the set of $A,B$ such that there exists $\lambda\in 3^{\mathbb Z}$ with $\lambda^4 A, \lambda^6 B$ in the set whose measure is computed in Lemma \ref{3-scaled}. We now split into cases depending on whether or not $E$ is supersingular. Note that an elliptic curve $y^2=f(x)$ in characteristic $p$ is supersingular if and only if the coefficient of $x^{p-1}$ in $f(x)^{\frac{p-1}{2}}$ is zero, so in characteristic $3$ if and only if the coefficient of $x^2$ in $f(x)$ is zero.

If $E$ is supersingular then every equation $y^2= x^3+a_2 x^2+a_4x+a_6$ defining $E$ has $a_2\equiv 0\bmod 3$ so that $r\in \mathbb Z_3$, thus $A' \in \mathbb Z_3$ and $B'\in \mathbb Z_3$, but $A'$ and $B'$ are not both divisible by $3$ as then $(x+r)^3 + A' (x+r) + B'$ would have a single root mod $3$ and $y^2= (x+r)^3 + A' (x+r) + B'$ would not be nodal. Thus the only way to have $A' = \lambda^4 A$ and $B'=\lambda^6 B$ with $A,B\in \mathbb Z_3$ and $3^4\nmid A$ or $3^6\nmid B$ is if $\lambda=1$ and $A'=A, B'=B$. Hence the measure of the set of $A,B$ is equal to the measure of the set of $A',B'$ and thus is $(1-3^{-1})  \frac{1}{\abs{\Aut(E)}}$.

On the other hand, if $E$ is not supersingular, then every equation $y^2= x^3+a_2 x^2+a_4x+a_6$ defining $E$ has $a_2\not\equiv  0\bmod 3$ so $r\in \frac{1}{3} \mathbb Z_3 \setminus \mathbb Z_3$ and hence $A' \notin \mathbb Z_3$ but $3^2 A' \in \mathbb Z_3$ and $3^3 B' \in \mathbb Z_3$. Thus, the only way to have  $A' = \lambda^4 A$ and $B'=\lambda^6 B$ with $A,B\in \mathbb Z_3$ and $3^4\nmid A$ or $3^6\nmid B$ is if $\lambda=3^{-1}$ and $A'=3^{-4}A, B'=3^{-6} B$. Thus the measure of the set of $A, B$ is $3^{-10}$ times the measure of the set of $A', B'$ and thus is $3^{-10} (1-3^{-1})  \frac{1}{\abs{\Aut(E)}}$.

We finally handle the case $p=2$, which is similar, but notationally more complicated. (We could have saved space by handling these cases together, but it seems clearer to have the easier $p=3$ case first.)

\begin{lemma}\label{2-local-criterion} For  $E \in \overline{\mathcal M}_{1,1}(\mathbb F_2)$, the curve $y^2 =x^3+Ax+B$ has reduction mod $2$ isomorphic to $E$ if and only if there exists $\lambda \in 2^{\mathbb Z}$ and $r,s,t \in \mathbb Q_2$ such that $(y+sx+t) ^2 = (x+r)^3 + \lambda^4 A (x+r) + \lambda^6 B$ is a polynomial with coefficients in $\mathbb Z_2$, and, reduced mod $2$ defines $E$. \end{lemma}

We  will write $A' $ for $ \lambda^4 A$ and $B'$ for $\lambda^6 B$.

\begin{proof} Any elliptic curve over $\mathbb Z_2$ may be put in the form $y^2 + a_1 xy+ a_3 y =x^3 + a_2 x^2 + a_4 x + a_6$ with $a_1, a_2,a_3, a_4,a_6 \in \mathbb Z_2$, so a curve has reduction isomorphic to $E$ and only if it can be put in that form with $y^2 + a_1 xy+ a_3 y =x^3 + a_2 x^2 + a_4 x + a_6$  defining $E$ modulo $2$. Any isomorphism between the curves $y^2 =x^3 +A x+B$ and $y^2 + a_1 xy+ a_3 y =x^3 + a_2 x^2 + a_4 x + a_6$  is necessarily a linear change of variables of the form $x \mapsto \lambda^{-2} ( x + r) $ and $y \mapsto \lambda^{-3} (y + sx+t)$ for $r,s,t  \in \mathbb Q_2$ and $\lambda \in \mathbb Q_2^\times$, and because multiplying $\lambda$ by a unit preserves the property that the reduction of $(y+sx+t) ^2 = (x+r)^3 + \lambda^4 A (x+r) + \lambda^6 B$ is isomorphic to $E$, we may assume $\lambda \in 2^{\mathbb Z}$. \end{proof} 

\begin{lemma}\label{2-local-balanced} Fix  $E \in \overline{\mathcal M}_{1,1}(\mathbb F_2)$. The measure of the set of tuples $a_1, a_2,a_3, a_4,a_6 \in \mathbb Z_2$ such that mod $2$ the equation $y^2+ a_1 xy+ a_3 y = x^3 + a_2 x^2 + a_4 x+ a_6$ defines a curve isomorphic to $E$ is $ \frac{ 1 }{2^2\abs{\Aut(E)}}$. \end{lemma}

\begin{proof} Any isomorphism between two curves over $\mathbb F_2 $ with equations of the form $y^2 + a_1 xy + a_3 y = x^3 + a_2 x^2 + a_4 x+ a_6$ has the form $x\mapsto x+r , y\mapsto y+ sx+t$ for some $ r,s,t \in \mathbb F_2$. It follows that the number of tuples $a_1,a_2,a_3, a_4,a_6$ defining a curve isomorphic to $E \in \overline{\mathcal M}_{1,1}(\mathbb F_2)$ is $2^3 / \abs{\Aut(E)}$ so the $2$-adic measure of the set of tuples $a_1, a_2,a_3, a_4,a_6\in\mathbb Z_2$  with $y^2+ a_1 xy+ a_3 y=x^3+a_2x^2+a_4 x+a_6$ isomorphic to $E$ is $ \frac{ 2^3/\abs{\Aut(E)}}{2^5}= \frac{1}{2^2\abs{\Aut(E)}}$. \end{proof}

\begin{lemma}\label{2-bijection} The map $\mathbb Q_2^5 \to \mathbb Q_2^5$ that sends $(a_1, a_2,a_3, A',B') $ to $(a_1, a_2,a_3, a_4,a_6)$ where  
\[ x^3 + a_2 x^2 + a_4 x + a_6-  y^2 - a_1xy - a_3 y  = (x+ a_2/3  - a_1^2/12 )^3 + A' (x+a_2/3) + B' - (y + a_1 x/2 + a_3/2)^2\] is a measure-preserving bijection.\end{lemma}

 \begin{proof} It is easy to check that there is a unique $a_4,a_6$ for each $A,B$ and vice versa so this is indeed a bijection. Moreover, $A'$ is equal to $a_4$ plus a polynomial in $a_1,a_2,a_3$ and $B'$ is equal to $a_6$ plus a polynomial in $a_1,a_2, a_3,a_4$ so this bijection preserves the $2$-adic measure.\end{proof}

\begin{lemma}\label{2-scaled} Fix  $E \in \overline{\mathcal M}_{1,1}(\mathbb F_2)$. The measure of the set of $A',B' \in \mathbb Q_2^2$ such that there exists $r,s,t\in \mathbb Q_2$ with $(x+r)^3 + A' (x+r) + B' - (y+sx + t)^2 $ a polynomial with coefficients in $\mathbb Z_2$ whose reduction mod $2$ gives an elliptic curve isomorphic to $E$, is $  \frac{1}{\abs{\Aut(E)}}$. \end{lemma}
 
\begin{proof}  Lemma \ref{2-local-balanced} and the measure-preserving bijection of Lemma \ref{2-bijection} together imply that the measure of the set of $(a_1,a_2,a_3, A',B')$ such that $(x+ a_2/3  - a_1^2/12 )^3 + A' (x+a_2/3) + B' - (y + a_1 x/2 + a_3/2)^2$ is a polynomial with coefficients in $\mathbb Z_2$ whose reduction mod $2$ gives an elliptic curve isomorphic to $E$ is $ \frac{1}{\abs{\Aut(E)}}$.  For each $A',B'$, if any triple $a_1,a_2,a_3$ satisfies this condition, then the measure of the set of triples $a_1,a_2,a_3$ that do is $2^{-2}$: Given an $a_1,a_2,a_3$ satisfying this condition, adding multiples of $2$ to $a_1$ and $a_3$ and an element of $\mathbb Z_2$ to $a_2$ produces another example, and any solution must arise this way.
  
So the measure of the set of tuples $(a_1,a_2,a_3,A', B')$ satisfying this condition is $2^{-2}$ times the measure of the set of pairs $A',B'$ where at least one triple $a_1,a_2,a_3$ exists, giving the statement.\end{proof}
 
We now conclude the $p=2$ case. We see from Lemma \ref{2-local-criterion} that the set of $A,B$ with $y^2=x^3+Ax +B$ having reduction isomorphic to $E$ is the set of $A,B$ such that there exists $\lambda\in 2^{\mathbb Z}$ with $\lambda^4 A, \lambda^6 B$ in the set whose measure is computed in Lemma \ref{2-scaled}.  Note that any equation $y^2 + a_1 xy + a_3 y = x^3+a_2 x^2 + a_4 x +a_ 6$ with $a_1,a_3\equiv 0 \bmod 2$ has a cusp singularity over $\mathbb F_2$ at the point $( a_4, a_2 a_4+a_6)$ and thus cannot define the curve $E$.  This means we must have $s$ or $t$ in $\frac{1}{2} \mathbb Z_2 \setminus \mathbb Z_2$ so that $(x+r)^3  + A' (x+r) + B'$ has all coefficients in $\frac{1}{4} \mathbb Z_2$ and either the coefficient of $x^2$ or the coefficient of $1$ in $\frac{1}{4} \mathbb Z_2 \setminus \mathbb Z_2$. It is then impossible to have $A',B'\in \mathbb Z_2$ as this would force the Newton polygon of $(x+r)^3  + A' (x+r) + B'$ to be a straight line with slope the $2$-adic valuation of $r$ and in particular with integer slope. However, we must have $r \in \frac{1}{4} \mathbb Z_2$ so that $A' \in \frac{1}{16} \mathbb Z_2, B' \in \frac{1}{64} \mathbb Z_2$. Thus the only way to have $A' = \lambda^4 A$ and $B'= \lambda^6 B$ with  $A\in \mathbb Z_2$ and $B\in \mathbb Z_2$ but either $2^4\nmid A$ or $2^6 \nmid B$ is if $ \lambda= 2^{-1}$. Hence the measure of the set of $A, B$ is $2^{-10}$ times the measure of the set of $A',B'$ and thus is $2^{-10} \frac{1}{ \abs{\Aut(E)}} = 2^{-8} (2^{-1}- 2^{-2} ) \frac{1}{ \abs{\Aut(E)}}$.\end{proof}

\begin{lemma}\label{lf-modular} For any prime $p$ and positive integer $\nu$, we have 
\[ \sum_{E \in \overline{\mathcal M}_{1,1}(\mathbb F_p)} \frac{ a_{p^\nu}(E)}{ \abs{\Aut(E)}} = - \sum_{ \substack{ f \in S^{ \nu+2}_0 (\SL_2(\mathbb Z))\\ \emph{eigenform} \\ a_1(f)=1 }} a_p (f) .\] \end{lemma}
\begin{proof} We prove this via the trace formula. The sum $\sum_{ \substack{ f \in S^{ \nu+2}_0 (\SL_2(\mathbb Z))\\ \textrm{eigenform} \\ a_1(f)=1 }} a_p (f) $ is $p^{\frac{\nu}{2}}$ times the trace of the Hecke operator $T(p)$ on $S^{ \nu+2}_0( SL_2(\mathbb Z))$, which is calculated by the trace formula \cite[Theorem 0.1]{Hijikata} as 
\[- \sum_{ \substack{ s\in \mathbb Z \\ s^2-4p  \leq 0 \textrm{ or } s^2-4p =t^2 \textrm{ for }t\in \mathbb Z}} a(s) \sum_{\substack{ f\in \mathbb N \\ f^2 \mid s^2-4p}} b(s,f) \]
where $\Phi_s(X) = X^2 -s X + p$ has roots $x,y$, we have $a(s)= \frac{1}{2} \frac{x^{\nu+1} - y^{\nu+1}}{x-y} p^{- \frac{\nu}{2}} $ if $s^2-4p <0$ or $a(s)= \frac{ \min ( \abs{x}, \abs{y}) ^{\nu-1}}{ \abs{x-y}}  (\operatorname{sgn}(x))^\nu n^{- \frac{\nu}{2}}$ if $s^2-4p >0$, and $b(s,f)$ is the class number of the order with discriminant $(s^2-4p)/f^2$ divided by half the order of the unit group if $s^2-4p<0$ or $b(s,f) =  \phi(\abs{t}/f)/2$ if $s^2-4p>0$.  (The other terms in \cite[Theorem 0.1]{Hijikata} can be ignored since, by definition, $\delta(\chi)=\delta(\sqrt{p})=0$ and $c(s,f)=1$ in our case.)

We can only have $s^2-4p=t^2$ if $s^2-t^2 =4p$ so $(s+t)$ and $(s-t)$ are both even and thus one is $\pm 2p$ and the other is $\pm 2$ so $s=\pm (p+1)$ and $t=\pm (p-1)$. In this case, $\sum_{f^2 \mid t^2} \phi( \abs{t}/f)/2= \abs{t}/2= (p-1)/2$. Furthermore  $x$ and $y$ are either $p$ and $1$ or $-p$ and $-1$ so that $\frac{ \min ( \abs{x}, \abs{y}) ^{\nu-1}}{ \abs{x-y}} = \frac{1}{p-1}$. Thus the contribution to the sum over $s$ from $s$ with $s^2-4p>0$ is $ \frac{1}{2} (\operatorname{sgn}(s))^\nu n^{- \frac{\nu}{2}}$. After multiplying by $n^{\frac{\nu}{2}}$, this matches the contribution of the nodal elliptic curves, of which there are two isomorphism classes, each with automorphisms of order $2$, one with $a_{p^\nu}(E)=1^\nu$ matching $s=p+1$ and one with $a_{p^{\nu}}(E) = -1^{\nu}$ matching $s=-(p+1)$.

The terms with $s^2-4p <0$ match the contributions of smooth elliptic curves, in particular the term $s$ matches the contribution of elliptic curves satisfying $a_p = s$ since $xy=p$ implies that  $\frac{x^{\nu+1} - y^{\nu+1}}{x-y} = p^{\frac{\nu}{2}} U_\nu \left( \frac{x+y}{2 \sqrt{p}}\right) = p^{ \frac{\nu}{2}} U_\nu \left( \frac{s}{2\sqrt{p}} \right)$ which is the value of $a_{p^\nu}$ for these elliptic curves.  The $ p^{- \frac{\nu}{2}}$ in the formula for $a(s)$ cancels the $p^{\frac{\nu}{2}}$ we multiply the trace of the Hecke operator by and the $\frac{1}{2}$ cancels the $2$ that was divided from the order of the unit group. Any such elliptic curve has endomorphism group the order with discriminant $(s^2-4p)/f^2$ for some $f$, the number of elliptic curves with endomorphism group a given order is the class number, and the order of the automorphism group is the order of the unit group of the order, so every term matches.

We also sketch an alternate geometric proof. One can consider $f$ the morphism from the universal elliptic curve to $\overline{\mathcal M}_{1,1}$ and $f_* \mathbb Q_\ell$ the pushforward of the constant sheaf in the sense of \'{e}tale cohomology. The trace of Frobenius on the stalk of this sheaf at a point $E $ is given by $a_p(E)$. The stalk is rank $2$ if $E$ is a smooth elliptic curve, with the determinant of Frobenius equal to $p$, and of rank $1$ if $E$ is nodal. It follows that the trace of Frobenius on $\operatorname{Sym}^{\nu} (f_* \mathbb Q_\ell)$ is $a_{p^\nu}(E)$ in either case. The Lefschetz fixed point formula for stacks then implies that $\sum_{E \in \overline{\mathcal M}_{1,1}(\mathbb F_p)} \frac{ a_{p^\nu}(E)}{ \abs{\Aut(E)}} $ is $ \sum_{i=0}^2(-1)^i  \operatorname{tr} ( \operatorname{Frob}_p, H^i( \overline{\mathcal M}_{1,1,\overline{\mathbb F}_p}, \operatorname{Sym}^{\nu} (f_* \mathbb Q_\ell))$. The cohomology groups in degree $0$ and $2$ vanish since $\nu>0$ while the work of Deligne~\cite{Deligne} can be used to deduce that $H^1(\overline{\mathcal M}_{1,1,\overline{\mathbb F}_p}, \operatorname{Sym}^{\nu} (f_* \mathbb Q_\ell))$ is the sum of two-dimensional Galois representations associated to cuspidal eigenforms of weight $\nu+2$, the trace of Frobenius on each one being the coefficient of $p$ in the corresponding modular form. This gives the same formula. \end{proof}

\begin{lemma}\label{all-lf-detailed} Fix a prime $p$ and positive integer $\nu$. Let $\ell_{p,\nu}$ be the quantity defined in Definition \ref{lp}. If $p>3$ we have
\[ \ell_{p,\nu}= \frac{p^{-1}- p^{-2}}{1-p^{-10}}  \sum_{E \in \overline{\mathcal M}_{1,1}(\mathbb F_p)} \frac{ a_{p^\nu}(E)}{ \abs{\Aut(E)}} = - \frac{p^{-1}- p^{-2}}{1-p^{-10}}  \sum_{ \substack{ f \in S^{ \nu+2}_0 (\SL_2(\mathbb Z))\\ \emph{eigenform} \\ a_1(f)=1 }} a_p (f) .\] 
If $p=3$ we have
\[ \ell_{3,\nu}= \frac{ 1- 3^{-1} }{ 1- 3^{-10}}  \Bigl( \sum_{\substack{ E \in \overline{\mathcal M}_{1,1}(\mathbb F_3)\\ E \emph{ supersingular}}} \frac{ a_{3^\nu}(E)}{ \abs{\Aut(E)}} + 3^{-10}  \sum_{\substack{ E \in \overline{\mathcal M}_{1,1}(\mathbb F_3)\\ E \emph{ ordinary}}} \frac{ a_{3^\nu}(E)}{ \abs{\Aut(E)}}  \Bigr)\] \[ =  3^{\frac{\nu}{2} -2} \Bigl(   U_\nu \left( \frac{3}{ 2 \sqrt{3} }\right) +4 U_\nu \left( 0\right) +  U_\nu \left( \frac{-3}{ 2 \sqrt{3}} \right) \Bigr)  -   \frac{3^{-10}- 3^{-11}}{1-3^{-10}}  \sum_{ \substack{ f \in S^{ \nu+2}_0 (\SL_2(\mathbb Z))\\ \emph{eigenform} \\ a_1(f)=1 }} a_3 (f) . \]
If $p=2$ we have
\[ \ell_{2,\nu}= \frac{2^{-10}}{1-2^{-10}}  \sum_{E \in \overline{\mathcal M}_{1,1}(\mathbb F_2)} \frac{ a_{2^\nu}(E)}{ \abs{\Aut(E)}} =-  \frac{2^{-10}}{1-2^{-10}}   \sum_{ \substack{ f \in S^{ \nu+2}_0 (\SL_2(\mathbb Z)) \\ \emph{eigenform} \\ a_1(f)=1 }} a_2 (f) .\] 
\end{lemma}
\begin{proof} For the case $p>3$, we sum Lemma \ref{lf-underlying} over all possible $E$ and then apply Lemma \ref{lf-modular}.

For the case $p=3$, summing Lemma \ref{lf-underlying} over all possible $E$ gives 
\[ \ell_{3,\nu}(E)= \frac{ 1- 3^{-1} }{ 1- 3^{-10}}  \Bigl( \sum_{\substack{ E \in \overline{\mathcal M}_{1,1}(\mathbb F_3)\\ E \textrm{ supersingular}}} \frac{ a_{3^\nu}(E)}{ \abs{\Aut(E)}} + 3^{-10}  \sum_{\substack{ E \in \overline{\mathcal M}_{1,1}(\mathbb F_3)\\ E \textrm{ ordinary}}} \frac{ a_{3^\nu}(E)}{ \abs{\Aut(E)}}  \Bigr)\]
\[= ( 1- 3^{-1})\sum_{\substack{ E \in \overline{\mathcal M}_{1,1}(\mathbb F_3)\\ E \textrm{ supersingular}}} \frac{ a_{3^\nu}(E)}{ \abs{\Aut(E)}} + \frac{ 3^{-10} - 3^{-11} }{ 1- 3^{-10}}   \sum_{\substack{ E \in \overline{\mathcal M}_{1,1}(\mathbb F_3)}} \frac{ a_{3^\nu}(E)}{ \abs{\Aut(E)}}.\]
We can evaluate the second term using Lemma \ref{lf-modular} and evaluate the first term by explicitly listing all the isomorphism classes of supersingular elliptic curves, which are $y^2=x^3-x$, $y^2=x^3-x-1$, $y^2=x^3-x-2$, and $y^2=x^3+x$ with automorphism groups of orders $6,6,6,2$ respectively and with $a_3 = 0 , 3, -3, 0 $ respectively. This gives
\[ ( 1- 3^{-1})\sum_{\substack{ E \in \overline{\mathcal M}_{1,1}(\mathbb F_3)\\ E \textrm{ supersingular}}} \frac{ a_{3^\nu}(E)}{ \abs{\Aut(E)}}\]
\[ = \frac{2}{3}\Bigl (  \frac{ 3^{\frac{\nu}{2}} U_\nu\left(0\right) }{6}  + \frac{ 3^{\frac{\nu}{2}} U_\nu\left(\frac{3}{ 2\sqrt{3}} \right) }{6} + \frac{ 3^{\frac{\nu}{2}} U_\nu\left(\frac{-3}{ 2\sqrt{3}} \right) }{6}+ \frac{ 3^{\frac{\nu}{2}} U_\nu\left(0\right) }{2}  \Bigr)\]
\[ =3^{\frac{\nu}{2} -2} \Bigl(   U_\nu \left( \frac{3}{ 2 \sqrt{3} }\right) +4 U_\nu \left( 0\right) +  U_\nu \left( \frac{-3}{ 2 \sqrt{3}} \right) \Bigr) .\]
Combining these gives Lemma \ref{all-lf} in the case $p=3$.

For the case $p=2$, we sum Lemma \ref{lf-underlying} over all possible $E$ and apply Lemma \ref{lf-modular}. \end{proof}

\begin{lemma}\label{primes-lf-detailed} Fix a prime $p$ and positive integer $\nu$. Let $\tilde{\ell}_{p,\nu}$ be the quantity defined in Definition~\ref{lptilde}. If $p\neq3$ we have
\[ \tilde{\ell}_{p,\nu}= p^{-1}  \sum_{E \in \mathcal M_{1,1}(\mathbb F_p)} \frac{ a_{p^\nu}(E)}{ \abs{\Aut(E)}} =  - p^{-1} \Bigl( \frac{ 1 + (-1)^\nu}{2}  + \sum_{ \substack{ f \in S^{ \nu+2}_0 (\SL_2(\mathbb Z))\\ \emph{eigenform} \\ a_1(f)=1 }} a_p (f) \Bigr).\] 
If $p=3$ we have
\[ \tilde{\ell}_{3,\nu}= \frac{ 1}{ 1 + 2 \cdot 3^{-10}}  \Bigl( \sum_{\substack{ E \in \mathcal M_{1,1}(\mathbb F_3)\\ E \emph{ supersingular}}} \frac{ a_{3^\nu}(E)}{ \abs{\Aut(E)}} + 3^{-10}  \sum_{\substack{ E \in \mathcal M_{1,1}(\mathbb F_3)\\ E \emph{ ordinary}}} \frac{ a_{3^\nu}(E)}{ \abs{\Aut(E)}}  \Bigr)\] 
\[ =  \begin{cases}  \frac{ 3^{\frac{\nu}{2}}}{ 3(1 + 2 \cdot 3^{-10})} \Bigl(  U_\nu \left( \frac{3}{ 2 \sqrt{3} }\right) +2 U_\nu \left( 0\right)  +3^{-9} \Bigl(  U_\nu \left( \frac{2}{ 2 \sqrt{3} }\right)+U_\nu \left( \frac{1}{ 2 \sqrt{3} }\right)  \Bigr)\Bigr)  & \textrm{if } \nu \textrm{ even} \\ 0 & \textrm{if }\nu \textrm{ odd} \end{cases}. \]
\end{lemma}
\begin{proof} For the case $p>3$, we sum Lemma \ref{lf-underlying} over all possible $E$ to obtain
\[\int_{ \{ A, B\in \mathbb Z_p \mid E_{A,B} \textrm{ has good reduction and }p^4 \nmid A \textrm{ or } p^6 \nmid B\} } a_{p^\nu}(E_{A,B}) = (p^{-1} - p^{-2}) \sum_{E \in \mathcal M_{1,1}(\mathbb F_p)} \frac{ a_{p^\nu}(E)}{ \abs{\Aut(E)}} \]
and the measure of the set $\{ A, B\in \mathbb Z_p \mid E_{A,B} \textrm{ has good reduction and }p^4 \nmid A \textrm{ or } p^6 \nmid B\}$ is $(p^{-1} -p^{-2}) \sum_{E \in \mathcal M_{1,1}(\mathbb F_p)} \frac{1}{ \abs{\Aut(E)}} = (p^{-1} - p^{-2} ) p$, since each $j$ invariant in $\mathbb F_p$ contributes $1$. Dividing, we obtain
\[ \tilde{\ell}_{p,\nu} = \frac{  (p^{-1} - p^{-2}) \sum_{E \in \mathcal M_{1,1}(\mathbb F_p)} \frac{ a_{p^\nu}(E)}{ \abs{\Aut(E)}} }{ (p^{-1} - p^{-2} ) p }  = p^{-1} \sum_{E \in \mathcal M_{1,1}(\mathbb F_p)} \frac{ a_{p^\nu}(E)}{ \abs{\Aut(E)}} \]
and furthermore Lemma \ref{lf-modular} gives
\[ \sum_{E \in \mathcal M_{1,1}(\mathbb F_p)} \frac{ a_{p^\nu}(E)}{ \abs{\Aut(E)}}=  \sum_{E \in \overline{\mathcal{M}}_{1,1}(\mathbb F_p)} \frac{ a_{p^\nu}(E)}{ \abs{\Aut(E)} }- \frac{ 1 + (-1)^\nu}{2} = - \frac{1 + (-1)^\nu}{2} -  \sum_{ \substack{ f \in S^{ \nu+2}_0 (\SL_2(\mathbb Z))\\ \textrm{eigenform} \\ a_1(f)=1 }} a_p (f)\]
since there are two isomorphism classes of nodal curves of genus $1$ over $\mathbb F_p$, the split and non-split, which have $a_{p^\nu}$ equal to $1$ and $(-1)^\nu$ respectively and automorphism group of order $2$. 

For the case $p=3$, the sum of Lemma \ref{lf-underlying} over all possible $E$ gives 
\[\int_{ \{ A, B\in \mathbb Z_3 \mid E_{A,B} \textrm{ has good reduction and }3^4 \nmid A \textrm{ or } 3^6 \nmid B\} } a_{3^\nu}(E_{A,B})\] \[ = (3^{-1} - 3^{-2}) \Bigl(  \sum_{\substack{ E \in \mathcal M_{1,1}(\mathbb F_3)\\ E \textrm{ supersingular}}} \frac{ a_{3^\nu}(E)}{ \abs{\Aut(E)}} + 3^{-10}  \sum_{\substack{ E \in \mathcal M_{1,1}(\mathbb F_3)\\ E \textrm{ ordinary}}} \frac{ a_{3^\nu}(E)}{ \abs{\Aut(E)}}\Bigr)\]
and the measure of $ \{ A, B\in \mathbb Z_3 \mid E_{A,B} \textrm{ has good reduction and }3^4 \nmid A \textrm{ or } 3^6 \nmid B\} $ is
\[ (3^{-1} - 3^{-2}) \Bigl( \sum_{\substack{ E \in \mathcal M_{1,1}(\mathbb F_3)\\ E \textrm{ supersingular}}} \frac{1}{ \abs{\Aut(E)}} + 3^{-10}  \sum_{\substack{ E \in \mathcal M_{1,1}(\mathbb F_3)\\ E \textrm{ ordinary}}} \frac{ 1}{ \abs{\Aut(E)}}\Bigr) = (3^{-1} - 3^{-2}) ( 1 + 2 \cdot 3^{-10} ) \] because the one supersingular $j$ invariant, $j=0$, contributes $1$ and the two ordinary $j$ invariants, $j=1$ and $j=2$, contribute $3^{-10}$ each.  Dividing, we obtain
\[ \tilde{ \ell}_{3,\nu}(E)= \frac{ 1}{ 1 + 2 \cdot 3^{-10}}  \Bigl( \sum_{\substack{ E \in \mathcal M_{1,1}(\mathbb F_3)\\ E \textrm{ supersingular}}} \frac{ a_{3^\nu}(E)}{ \abs{\Aut(E)}} + 3^{-10}  \sum_{\substack{ E \in \mathcal M_{1,1}(\mathbb F_3)\\ E \textrm{ ordinary}}} \frac{ a_{3^\nu}(E)}{ \abs{\Aut(E)}}  \Bigr).\] 

If $\nu$ is odd, then $a_{3^{\nu}}$ takes opposite values on each elliptic curve and its quadratic twist, cancelling those terms and giving the sum a value of $0$. If $\nu$ is even, then $a_{3^\nu}$ takes equal values on the curve and its quadratic twist, so we may merge those terms and (as long as the elliptic curve is not equal to its quadratic twist) cancel the $2$ in the order of the automorphism group. (The same facts are true in every characteristic, but only here do they significantly simplify the formula).

Having done this, the supersingular quadratic twist pair $y^2=x^3-x-1$ and $y^2=x^3-x-2$ contributes $\frac{1}{3}  3^{\frac{\nu}{2}} U_\nu \left( \frac{3}{ 2 \sqrt{3} }\right) $ to the sum over supersingular $E$,  the supersingular curve $y^2=x^3-x$ isomorphic to its own quadratic twist contributes $\frac{1}{6}3^{\frac{\nu}{2}} U_\nu(0)$, and the supersingular curve $y^2=x^3+x$ isomorphic to its own quadratic twist contributes $\frac{1}{2} 3^{\frac{\nu}{2}} U_\nu(0)$. The ordinary quadratic twist pair $y^2 =x^3+x^2 +1$ and $y^2 = x^3 -x^2-1$ contribute $3^{-10} 3^{\frac{\nu}{2}} U_\nu \left( \frac{2}{ 2 \sqrt{3} }\right)  $ and the ordinary quadratic twist pair $y^2 = x^3 + x^2-1$ and $y^2=x^3 - x^2 +1$ contribute $3^{-10}3^{\frac{\nu}{2}} U_\nu \left( \frac{1}{ 2 \sqrt{3} }\right)  $. Pulling out the factor of $3$ proves the case $p=3$.

The case $p=2$ is identical to the case $p>3$ except the additional factor $2^{-8}$ appears on both the numerator and denominator and is canceled. \end{proof}

\begin{lemma}\label{prime-lf-detailed} Fix a prime $p$ and a nonnegative integer $\nu$. Let $\hat{\ell}_{p,\nu}$ be the quantity defined in Definition \ref{lphat}. For $p>3$ we have
\[ \hat{\ell}_{p, 0 } =   \frac{1- p^{-1}}{ 1- p^{-10}}  \]
\[ \hat{\ell}_{p,1} =  0  \]
\[ \hat{\ell}_{p,2} =  - \frac{p-p^{-1} +p^{-2}- p^{-8} }{ (1-p^{-10}) (p-1) }   \]
and for $\nu>2$
\begin{align*}
\hat{\ell}_{p,\nu} &= \frac{p^{-1}- p^{-2}}{1-p^{-10}} \Bigl(  \sum_{E \in {\mathcal M}_{1,1}(\mathbb F_p)} \frac{ a_{p^\nu}(E)}{ \abs{\Aut(E)}} - p  \frac{ 1 + (-1)^\nu}{2}  \Bigr)\\
&= -  \frac{p^{-1}- p^{-2}}{1-p^{-10}} \Bigl(  \frac{p+1}{2} ( 1 + (-1)^\nu )  + \!\!\!\!\! \sum_{ \substack{ f \in S^{ \nu+2}_0 (\SL_2(\mathbb Z))\\ \emph{eigenform} \\ a_1(f)=1 }}\!\!\!\! a_p (f) \Bigr).
\end{align*}
If $p=3$ we have
\[ \hat{\ell}_{3, 0 } =   (1-3^{-10})^{-1}  \left( \frac{2}{3} + \frac{4}{3^{11}} \right) \]
\[ \hat{\ell}_{3,1} =  0\]
\[ \hat{\ell}_{3,2} = -\frac{ 3 - 3^{-7}+16 \cdot 3^{-11}   } {2  (1-3^{-10} )} \]
and for $\nu>2$
\[\hat{\ell}_{3,\nu} = (1-3^{-10} )^{-1} \frac{2}{3} \Bigl( \sum_{\substack{ E \in \mathcal M_{1,1}(\mathbb F_3)\\ E \emph{ supersingular}}} \frac{ a_{3^\nu}(E)}{ \abs{\Aut(E)}} + 3^{-10}  \sum_{\substack{ E \in \mathcal M_{1,1}(\mathbb F_3)\\ E \emph{ ordinary}}} \frac{ a_{3^\nu}(E)}{ \abs{\Aut(E)}}  -  3^{-9} \frac{ 1 + (-1)^\nu}{2}  \Bigr)\] 
\[ =  \begin{cases} \frac{2}{9} (1-3^{-10} )^{-1}  \Bigl( 3^{ \frac{\nu}{2}} \Bigl(  U_\nu\! \left( \frac{3}{ 2 \sqrt{3} }\right)\! +2 U_\nu\! \left( 0\right)  +3^{-9} \Bigl(  U_\nu\! \left( \frac{2}{ 2 \sqrt{3} }\right)\!+U_\nu\! \left( \frac{1}{ 2 \sqrt{3} }\right) \Bigr) \Bigr)  -3^{-8}  \Bigr)  & \textrm{if } \nu \textrm{ even} \\ 0 & \textrm{if }\nu \textrm{ odd} \end{cases}. \]
If $p=2$ we have
\[ \hat{\ell}_{2, 0 } =  \frac{2^{-9}}{ 1- 2^{-10}} \]
\[ \hat{\ell}_{2,1} =  0\]
\[ \hat{\ell}_{2,2} =  - \frac{ 4 - 2^{-6}  + 3  \cdot 2^{-10} }{1- 2^{-10}} \]
and for $\nu>2$
\[\hat{\ell}_{2,\nu} = \frac{1}{ 2^{10}-1}  \Bigl(  \sum_{E \in {\mathcal M}_{1,1}(\mathbb F_2)} \frac{ a_{2^\nu}(E)}{ \abs{\Aut(E)}} - 1 - (-1)^\nu   \Bigr)   = -  \frac{1}{ 2^{10}-1} \Bigl(  \frac{3}{2} ( 1 + (-1)^\nu)  +\!\!\!\!\!\!\!\! \sum_{ \substack{ f \in S^{ \nu+2}_0 (\SL_2(\mathbb Z))\\ \emph{eigenform} \\ a_1(f)=1 }}\!\!\!\!\!\! a_2 (f) \Bigr).\]
\end{lemma}

\begin{proof}
For the case $p>3$, we sum Lemma \ref{lf-underlying} over all possible $E$. We split further into cases depending on the value of $\nu$.

For $\nu=0$ we obtain

\[ \hat{\ell}_{p,0} = (1-p^{-10})^{-1} \sum_{E \in \mathcal M_{1,1}(\mathbb F_p)} \frac{p^{-1} - p^{-2}}{ \abs{\Aut(E)}}  = \frac{1- p^{-1}}{1-p^{-10}} \] since each $j$ invariant contributes $1$ to $\sum_{E \in \mathcal M_{1,1}(\mathbb F_p)} \frac{1}{ \abs{\Aut(E)}}$.

For $\nu=1$, the contribution of each elliptic curve cancels with its quadratic twist and we get zero.

For $\nu=2$, there is a subtlety as elliptic curves with additive reduction contribute. To calculate the measure of the set of $A,B$ with $p^4\nmid A$ or $p^6\nmid B$ such that $E_{A,B}$ has additive reduction at $p$, we take the total measure $1- p^{-10}$ and subtract the measure $(p^{-1} -p^{-2} ) (p+1) = 1 - p^{-2}$ of elliptic curves with good or multiplicative reduction, obtained via Lemma~\ref{lf-underlying}, to see that curves with additive reduction have measure $p^{-2} - p^{-10}$. This implies that
\[ \begin{aligned}  \hat{\ell}_{p,2} &= (1-p^{-10} )^{-1}  \int_{ \substack{A, B\in \mathbb Z_p \\ p^4\nmid A \textrm{ or } p^6 \nmid B \\ p^2 \nmid N(E_{A,B}) }} \cdot \begin{cases}  a_{p^2} & \textrm{if }E_{A,B}\textrm{ has good reduction} \\ -p a_{p^2} & \textrm{if }E_{A,B}\textrm{ has multiplicative reduction} \\  - \frac{p^2}{p-1} & \textrm{if }E_{A,B} \textrm{ has additive reduction} \end{cases} \\  &= (1-p^{-10})^{-1} \Bigl( \sum_{E \in \mathcal M_{1,1}(\mathbb F_p)}\!\!\!\!(p^{-1} - p^{-2}) \frac{ a_{p^2}}{\abs{\Aut(E)}}  - p (p^{-1} -p^{-2} )  - \frac{p^2}{p-1} (p^{-2} - p^{-10} ) \Bigr)\\ &= (1-p^{-10})^{-1} \Bigl( \sum_{E \in \overline{ \mathcal M}_{1,1}(\mathbb F_p) }\!\!\!\!(p^{-1} -p^{-2} ) \frac{ a_{p^2}}{\abs{\Aut(E)}}  - (p+1) (p^{-1} -p^{-2} )  - \frac{p^2}{p-1} (p^{-2} - p^{-10} ) \Bigr)\\
&= (1-p^{-10})^{-1} \Bigl( - (p+1) (p^{-1} - p^{-2} ) - \frac{p^2}{p-1} (p^{-2} - p^{-10} ) \Bigr) \\  &= \frac{  - (1 -p^{-1} - p^{-2} + p^{-3} ) - (p^{-1} - p^{-9} )  }{( 1- p^{-10}) (1-p^{-1})} = - \frac{1 -p^{-2} + p^{-3} - p^{-9} } {( 1- p^{-10}) (1-p^{-1})},\end{aligned}\]
where we use that $a_{p^2}=1$ for curves with multiplicative reduction (twice) and then apply Lemma \ref{lf-modular}.

For $\nu>2$, the strategy is similar except that we do not need to consider curves with additive reduction, obtaining
\[\begin{aligned} \hat{\ell}_{p,\nu} &= (1-p^{-10} )^{-1}  \int_{ \substack{A, B\in \mathbb Z_p \\ p^4\nmid A \textrm{ or } p^6 \nmid B \\ p^2 \nmid N(E_{A,B}) }}   a_{p^\nu} (E_{A,B})  \cdot \begin{cases} 1 & \text{if } p \nmid N(E_{A,B}) \\ -p& \textrm{if } p \mid N(E_{A,B}) \end{cases}\\ &=  (1-p^{-10})^{-1} \Bigl( \sum_{E \in \mathcal M_{1,1}(\mathbb F_p)}(p^{-1} - p^{-2}) \frac{ a_{p^\nu}}{\abs{\Aut(E)}}  - p(p^{-1} -p^{-2} ) \frac{ 1 + (-1)^\nu}{2} \Bigr)\\ &=\frac{p^{-1} - p^{-2}}{1-p^{-10}}  \Bigl( \sum_{E \in \overline{ \mathcal M}_{1,1}(\mathbb F_p) }\frac{ a_{p^\nu}}{\abs{\Aut(E)}}  - (p+1)  \frac{ 1 + (-1)^\nu}{2} \Bigr)\\ &= - \frac{p^{-1}- p^{-2}}{1-p^{-10}} \Bigl(  (1 + (-1)^\nu) \frac{p+1}{2}  + \sum_{ \substack{ f \in S^{ \nu+2}_0 (\SL_2(\mathbb Z))\\ \textrm{eigenform} \\ a_1(f)=1 }} a_p (f) \Bigr).\end{aligned}\]

We now consider the case $p=2$. For every value of $\nu$ except $2$, every term appearing is simply multiplied by the additional $2^{-8}$ factor that appears in Lemma \ref{lf-underlying} if $p=2$, so we can take the formulas for $p>3$, specialize $p$ to $2$, and then multiply by $2^{-8}$. For $\nu=2$, we must calculate the measure of the set of $A,B$ such that $E_{A,B}$ has additive reduction at $2$. This can be obtained by subtracting from $1-2^{-10}$ the measure $(2^{-1} - 2^{-2}) 2^{-8} (2+1) = 3 \cdot 2^{-10}$ of elliptic curves with good or multiplicative reduction. This implies that curves with additive reduction have measure $1 - 2^{-8}$. Hence
\[ \begin{aligned}  \hat{\ell}_{2,2} &=  (1-2^{-10} )^{-1}  \int_{ \substack{A, B\in \mathbb Z_2 \\ 2^4\nmid A \textrm{ or } 2^6 \nmid B \\ 2^2 \nmid N(E_{A,B}) }} \cdot \begin{cases}  a_{2^2} & \textrm{if }E_{A,B}\textrm{ has good reduction} \\ -2 a_{2^2} & \textrm{if }E_{A,B}\textrm{ has multiplicative reduction} \\  - \frac{2^2}{2-1} & \textrm{if }E_{A,B} \textrm{ has additive reduction} \end{cases} \\  &= (1-2^{-10})^{-1} \Bigl( \sum_{E \in \mathcal M_{1,1}(\mathbb F_2)} 2^{-10} \frac{ a_{2^2}}{\abs{\Aut(E)}}  - 2 \cdot 2^{-10}   - \frac{2^2}{2-1} (1 - 2^{-8}  ) \Bigr)\\ &= (1-2^{-10})^{-1} \Bigl( \sum_{E \in \overline{ \mathcal M}_{1,1}(\mathbb F_p) }2^{-10} \frac{ a_{p^2}}{\abs{\Aut(E)}}  - (2+1) \cdot 2^{-10}  - \frac{2^2}{2-1} (1 - 2^{-8}) \Bigr)\\ 
&= (1-2^{-10})^{-1} \Bigl( -3 \cdot 2^{-10}  -4 \cdot (1 - 2^{-8}  ) \Bigr) \\  &= -\frac{  4 - 2^{-6} + 3\cdot 2^{-10}}{ 1- 2^{-10}} .\end{aligned}\]

We finally consider the case $p=3$. First we take $\nu=0$. Lemma \ref{lf-underlying} gives   \[\begin{aligned} \hat{\ell}_{3,0}(E) &= \frac{ 1- 3^{-1} }{ 1- 3^{-10}}  \Bigl( \sum_{\substack{ E \in \overline{\mathcal M}_{1,1}(\mathbb F_3)\\ E \textrm{ supersingular}}}\frac{1}{ \abs{\Aut(E)}} + 3^{-10}  \sum_{\substack{ E \in \overline{\mathcal M}_{1,1}(\mathbb F_3)\\ E \textrm{ ordinary}}} \frac{ 1}{ \abs{\Aut(E)}}  \Bigr)\\
&=  \frac{ 1- 3^{-1} }{ 1- 3^{-10}}  \Bigl( 1 + 2 \cdot 3^{-10} \Bigr),\end{aligned}\]
since there is one supersingular $j$-invariant over $\mathbb F_3$ and two ordinary $j$-invariants over $\mathbb F_3$, with each $j$-invariant contributing $1$ to the corresponding sum.

For $\nu=1$, the contribution of each elliptic curve cancels with its quadratic twist and we get zero.

For $\nu=2$, we again calculate the density of additive reduction by subtraction. In this case, the total measure of curves with supersingular reduction is $1-3^{-1}$, the two ordinary $j$-invariants contribute $2\cdot (1-3^{-1}) \cdot 3^{-10}$, and the curves with multiplicative reduction contribute $ (1-3^{-1}) \cdot 3^{-10}$, for a total of $(1-3^{-1}) ( 1 + 3^{-9}) = 2 \cdot 3^{-1} + 2 \cdot 3^{-10} $. Subtracting this from the total measure $1-3^{-10}$, we see that the curves with additive reduction have measure $1-  2 \cdot 3^{-1} - 3^{-10} - 2 \cdot 3^{-10} = 3^{-1} - 3^{-9}$. This gives
\[(1-3^{-10} )\hat{\ell}_{3,2} = (1- 3^{-1} ) \Bigl(\!\!\!\!\sum_{\substack{ E \in \mathcal M_{1,1}(\mathbb F_3)\\ E \textrm{ supersingular}}}\!\!\!\! \frac{ a_{3^2}(E)}{ \abs{\Aut(E)}} + 3^{-10} \!\!\!\!\!\!\!\! \sum_{\substack{ E \in \mathcal M_{1,1}(\mathbb F_3)\\ E \textrm{ ordinary}}}\! \frac{ a_{3^2}(E)}{ \abs{\Aut(E)}}  -  3 \cdot 3^{-10} \Bigr)  - \frac{3^2}{3\!-\!1} ( 3^{-1} - 3^{-9}).\]
We now evaluate the sums by enumerating isomorphism classes of elliptic curves. We again use the observation that $a_{3^\nu}$ takes equal values on the curve and its quadratic twist, so we may merge those terms and (as long as the elliptic curve is not equal to its quadratic twist) cancel the $2$ in the order of the automorphism group. We also note the formula $a_{3^2}= a_3^2 - 3$.

Having done this, the supersingular quadratic twist pair $y^2=x^3-x-1$ and $y^2=x^3-x-2$ has $a_3 = \pm 3$, so  $a_{3^2} =6$, and thus contributes  $\frac{1}{3}  6=2 $ to the sum over supersingular $E$,  the supersingular curve $y^2=x^3-x$ isomorphic to its own quadratic twist has $a_3=0$, so $a_{3^2}=-3$, and thus contributes $-\frac{1}{6}3 =- \frac{1}{2}$, and the supersingular curve $y^2=x^3+x$ isomorphic to its own quadratic twist has $a_3=0$, so $a_{3^2}=-3$, and so contributes $-\frac{1}{2} 3$. Hence the sum over supersingular curves vanishes. The ordinary quadratic twist pair $y^2 =x^3+x^2 +1$ and $y^2 = x^3 -x^2-1$ have $a_3 = \pm 2$ so $a_{3^2}=1$ and so contribute  $1 $ to the sum over ordinary $E$ and the ordinary quadratic twist pair $y^2 = x^3 + x^2-1$ and $y^2=x^3 - x^2 +1$ has $a_3=\pm 1$ so $a_{3^2} = -2$ and so contribute $-2$. Thus the sum over ordinary $E$ is $-1$. Plugging in these values we obtain
\[(1-3^{-10} ) \hat{\ell}_{3,2} = (1- 3^{-1} ) \Bigl(- 3^{-10}   -  3 \cdot 3^{-10} \Bigr) - \frac{3^2}{3-1} ( 3^{-1} - 3^{-9})    =-  \frac{ 16 \cdot 3^{-11} + 3 - 3^{-7} } {2  }  .\]

Finally, we handle the case $\nu>2$. Summing Lemma \ref{lf-underlying} gives
 
\[\begin{aligned}(1-3^{-10} ) \hat{\ell}_{3,\nu} &=   (1- 3^{-1} )  \Bigl(\!\! \sum_{\substack{ E \in \mathcal M_{1,1}(\mathbb F_3)\\ E \textrm{ supersingular}}}\!\!\!\! \frac{ a_{3^\nu}(E)}{ \abs{\Aut(E)}} + 3^{-10} \!\!\!\!\!\!\!  \sum_{\substack{ E \in \mathcal M_{1,1}(\mathbb F_3)\\ E \textrm{ ordinary}}} \frac{ a_{3^\nu}(E)}{ \abs{\Aut(E)}}  -  3 \cdot 3^{-10} \frac{ 1+ (-1)^\nu}{2}  \Bigr)\\
&=  \frac{2}{3} \Bigl( \sum_{\substack{ E \in \mathcal M_{1,1}(\mathbb F_3)\\ E \textrm{ supersingular}}} \frac{ a_{3^\nu}(E)}{ \abs{\Aut(E)}} + 3^{-10}  \sum_{\substack{ E \in \mathcal M_{1,1}(\mathbb F_3)\\ E \textrm{ ordinary}}} \frac{ a_{3^\nu}(E)}{ \abs{\Aut(E)}}  -  3^{-9} \frac{ 1 + (-1)^\nu}{2}  \Bigr)\end{aligned}\]

For $\nu$ odd, the sum vanishes because each curve cancels with its quadratic twist. We again evaluate the sums by summing over quadratic twist pairs. The supersingular quadratic twist pair $y^2=x^3-x-1$ and $y^2=x^3-x-2$ contributes $\frac{1}{3}  3^{\frac{\nu}{2}} U_\nu \left( \frac{3}{ 2 \sqrt{3} }\right) $ to the sum over supersingular $E$,  the supersingular curve $y^2=x^3-x$ isomorphic to its own quadratic twist contributes $\frac{1}{6}3^{\frac{\nu}{2}} U_\nu(0)$, and the supersingular curve $y^2=x^3+x$ isomorphic to its own quadratic twist contributes $\frac{1}{2} 3^{\frac{\nu}{2}} U_\nu(0)$. This gives
\[  \sum_{\substack{ E \in \mathcal M_{1,1}(\mathbb F_3)\\ E \textrm{ supersingular}}}\!\! \frac{ a_{3^\nu}(E)}{ \abs{\Aut(E)}}  = \frac{1}{3}  3^{\frac{\nu}{2}} U_\nu\! \left( \frac{3}{ 2 \sqrt{3} }\right) + \frac{1}{6}3^{\frac{\nu}{2}} U_\nu(0)+\frac{1}{2} 3^{\frac{\nu}{2}} U_\nu(0)= \frac{1}{3}  3^{\frac{\nu}{2}} U_\nu\! \left( \frac{3}{ 2 \sqrt{3} }\right) + \frac{2}{3}3^{\frac{\nu}{2}} U_\nu(0) .\]

The ordinary quadratic twist pair $y^2 =x^3+x^2 +1$ and $y^2 = x^3 -x^2-1$ contribute $ 3^{\frac{\nu}{2}} U_\nu \left( \frac{2}{ 2 \sqrt{3} }\right)$, and the ordinary quadratic twist pair $y^2 = x^3 + x^2-1$ and $y^2=x^3 - x^2 +1$ contribute $3^{\frac{\nu}{2}} U_\nu \left( \frac{1}{ 2 \sqrt{3} }\right)  $. This gives
\[ \sum_{\substack{ E \in \mathcal M_{1,1}(\mathbb F_3)\\ E \textrm{ ordinary}}} \frac{ a_{3^\nu}(E)}{ \abs{\Aut(E)}}  = 3^{\frac{\nu}{2}} U_\nu \left( \frac{2}{ 2 \sqrt{3} }\right) +3^{\frac{\nu}{2}} U_\nu \left( \frac{1}{ 2 \sqrt{3} }\right).\]
Plugging these in, we obtain for $\nu>2 $ even
\[
(1-3^{-10} )\hat{\ell}_{3,\nu} = \frac{2}{3} \! \left( \frac{1}{3}  3^{\frac{\nu}{2}} U_\nu\! \left( \frac{3}{ 2 \sqrt{3} }\right)\!+\! \frac{2}{3}3^{\frac{\nu}{2}} U_\nu(0)\!+\! 3^{-10} 3^{\frac{\nu}{2}} U_\nu\! \left( \frac{2}{ 2 \sqrt{3} }\right)\!+\! 3^{-10} 3^{\frac{\nu}{2}} U_\nu\! \left( \frac{1}{\! 2 \sqrt{3} }\right)\!-\! 3^{-9}\!  \right).
\]

We simplify by pulling out a factor of $3^{-1}$ and then pulling out a factor of $3^{\frac{\nu}{2}}$, giving
\[  (1-3^{-10} )\hat{\ell}_{3,\nu}=\frac{2}{9}  \left(   3^{\frac{\nu}{2}} \left( U_\nu \left( \frac{3}{ 2 \sqrt{3} }\right) + 2  U_\nu(0) + 3^{-9}U_\nu \left( \frac{2}{ 2 \sqrt{3} }\right) +3^{-9}  U_\nu \left( \frac{1}{ 2 \sqrt{3} }\right) \right)- 3^{-8}  \right). \qedhere\]
\end{proof}

\section{Computations}

In this section we summarize the computations used to produce Figures~\ref{fig:pred}--\ref{fig:primes} and the methods used to efficiently compute the left-hand side (LHS) and right-hand side (RHS) of the formulas in \eqref{main-prediction}, \eqref{primes-P}, \eqref{prime-P}. Here \eqref{main-prediction} is the prediction given by Conjecture~\ref{main-conjecture}, while \eqref{primes-P} and \eqref{prime-P} correspond to the formulas in Proposition~\ref{primes} and Theorem~\ref{prime}, respectively, as $P\to\infty$.  For $P=1$, formulas \eqref{primes-P} and \eqref{prime-P} coincide and agree with Proposition~\ref{all}.

We assume throughout this section that we are working with normalized indicator functions $W_j(u)$ of area 1 whose supports partition the interval $(0,u_{\max}]$ into $r$ subintervals $I_j\colonequals (j\delta,(j+1)\delta]$ of width $\delta\colonequals u_{\max}/r$.  In the computations for Figures~\ref{fig:pred} and~\ref{fig:all} we used $u_{\max}=1$, with $r=2000$ on the LHS and $r=100000$ on the RHS; see Section~\ref{conv} for the rationale behind these choices.  In the context of Conjecture~\ref{main-conjecture}, we assume $(C_1,C_2]\subseteq (0,u_{\max}]$ is equal to a union of $I_j$ (possibly a single $I_j$), and note that normalizing the $W_j(u)$ to have area one amounts to multiplying both the LHS and RHS of \eqref{main-prediction} by the same factor.

\subsection{Approximating the limit on the left}

To approximate the limit as $X\to\infty$ that appears on the LHS of \eqref{main-prediction}, \eqref{primes-P}, \eqref{prime-P}, we consider increasing values of $X=2^n$ and the set of elliptic curves $\{E: H(E)\le X\}$.
For each $E$ we must compute $a_n(E)$ for $n\le u_{\max}N(E)$, or some subset of these $n$, depending on the value of $P$.  We also want to consider increasing values of $P$, and $r$ different functions $W_j$, so we will compute $a_n(E)$ for $n\le N \colonequals u_{\max}N(E)$ just once, accumulating integer sums
\[
s_{j,P}(E)\colonequals \sum_{\substack{n\le N\\p\nmid n\text{ for }p\le P\\\frac{n}{N(E)}\in I_j}} \epsilon(E)a_n(E)
\]
as we go. For \eqref{main-prediction} we take $P=\infty$ and change the constraint ``$p\nmid n\text{ for }p\le P$'' to ``$n$ prime''.

By accumulating the sums $s_{j,P}$ as we compute $a_n(E)$ we avoid the need to store all the $a_n(E)$, which would require a prohibitive amount of storage.  In our computations we maintained approximately 24000 sums $s_{j,P}(E)$ for each $E$, corresponding to $12$ values of $P$ and $r=2000$ values of $j$. This is much less than the number of $a_n(E)$ to be computed; note that $N(E)$ may be as large as $32H(E)$, and with $X=2^{28}$ we will need to compute more than $10^{9}$ values of $a_n(E)$ for each of several million $E$.

To compute $a_p(E)$ we use the \textsc{smalljac} library \cite{KS08}, which provides an efficient implementation of a generic group algorithm \cite{SutherlandThesis} for computing $E(\mathbb{F}_p)=p+1-a_p(E)$ at $p\nmid N(E)$ (it also computes $a_p(E)$ for $p|N(E)$).
It takes $N^{5/4+o(1)}$ time to do this for all $p\le N$, which is asymptotically worse than both the $N(\log N)^{3+o(1)}$ complexity one could obtain using the average polynomial-time approach of \cite{HS14} and the $N(\log N)^{4+o(1)}$ complexity given by Schoof's algorithm (under the Generalized Riemann Hypothesis one can improve this to $N(\log N)^{3+o(1)}$ expected time, see \cite{SS15}, but is still slower than the average polynomial-time approach in practice).
However, in practice the generic group approach is faster than both these approaches for all feasible values of $N$; for $N \le 2^{28}$ this can be seen in the timings of \cite[Table 3]{KS08} and \cite[Table 1]{Sut20}, and in fact this advantage extends to $N\le 2^{40}$ and beyond.  The generic group algorithm also has a better space complexity than the average polynomial-time algorithm: $N^{1/4+o(1)}$ versus $N^{1+o(1)}$.

We compute $a_n(E)$ for $n\le N$ (updating the sums $s_{j,P}(E)$ as we go) as follows:
\begin{enumerate}[label=\arabic*.]
\setlength{\itemsep}{6pt}
\item Compute $a_p(E)$ for $p\le \sqrt{N}$.
\item Compute $a_n(E)$ for $\sqrt{N}$-smooth $n\le N$ using $a_{p^k}=a_pa_{p^{k-1}}-pa_{p^{k-2}}$ and $a_{mn}=a_{m}a_{n}$ for $\gcd(m,n)=1$, and keep the $a_n$ with $n\le \sqrt{N}$ in memory.
\item Compute $a_p(E)$ for $\sqrt{N}<p\le N$ and $a_n(E)=a_{n/p}(E)a_p(E)$ for $p|n\le N$.
\end{enumerate}
In step 2 we use a space-efficient implementation of Bernstein's algorithm \cite[Ch. 2]{Bernstein} to enumerate the $\sqrt{N}$-smooth $n\le N$ along with their prime factorizations using less than $N^{1/2+o(1)}$ space.  This space-efficient approach makes it feasible to perform these computations for many $E$ in parallel on a large scale, which would not have been possible using a naive approach that requires $N^{1+o(1)}$ space.

Once the integers $s_{j,P}(E)$ have been computed for all $E\in \{E:H(E)\le X\}$, we compute our approximation to the LHS for $P$ and the indicator function $W_j(u)$ via
\[
\mathrm{LHS}(j,P,X)\colonequals \frac{1}{\delta}\cdot\frac{1}{\#\mathcal H(X)}\sum_{E\in \mathcal H(X)} \frac{c_{j,P}(E)}{N(E)}s_{j,P}(E),
\]
where $\mathcal H(X) \colonequals \{E:H(E)\le X\}$ and $c_{j,P}(E)\colonequals \prod_{p\le P}(1-\nicefrac{1}{p})^{-1}$ for $P<\infty$.  For $P=\infty$ we instead use $c_{j,P}(E)\colonequals \log (\frac{2j+1}{2}\delta N(E))$ and when considering \eqref{primes-P} and Proposition~\ref{primes} we replace $\mathcal H(X)$ with $\mathcal H_P(X)\colonequals \{E \in \mathcal H(X): p\nmid N(E)\text{ for }p\le P\}$.

As an example, for $X=2^{28}$ there are 5,122,428 elliptic curves $E$ with $H(E)\le X$ with conductors $N(E)$ ranging up to $2^{33}$.  We used a large cloud-based parallel computation to compute approximately $10^{15.75}$ values of $a_n(E)$ and 24,000 sums $s_{j,P}(E)$ for each $E$.
This involved about 140 CPU-years of compute time.

\subsection{Approximating the integral on the right}
To approximate the integral on the RHS of \eqref{main-prediction}, \eqref{main-evaluation}, \eqref{primes-P}, \eqref{prime-P} we evaluate the integrand at the midpoint $u_j\colonequals (j+\tfrac{1}{2})\delta$ of the interval $I_j$, under the assumption that $\delta$ is small enough to make this a good approximation; we discuss the choice of a sufficiently small $\delta$ in Section~\ref{conv}.

For $m\in \mathbb N$ we define multiplicative functions
\begin{align*}
\ell(m) &\colonequals \prod_{p|m}\ell_{p,2v_p(m)}\\
\ell'(m) &\colonequals \begin{cases}
\prod_{p|m}\tilde\ell_{p,2v_p(m)}& \text{ for } \eqref{primes-P}\\
\prod_{p|m}\hat\ell_{p,2v_p(m)}& \text{ for } \eqref{prime-P}\\
\end{cases}\\
\end{align*}
\begin{align*}
\psi(m) &\colonequals
\begin{cases}
\phi(m)^{-1} & \text{ for } \eqref{primes-P} \\
\phi(m)^{-1}\prod_{p|m}\hat\ell_{p,0} & \text{ for } \eqref{prime-P}\\
\end{cases}\\
\varphi_{d,P}(m) &\colonequals \begin{cases}
\prod_{p\le P}p^{v_p(m)} & \text{ for } \eqref{primes-P} \\
\prod_{p|d} p^{v_p(m)} & \text{ for } \eqref{prime-P}\\
\end{cases}\\
\end{align*}
and for $0\le j < r$ and $P\ge 1$ we define
\[
\mathrm{RHS}(j,P,B)\colonequals 2\pi \sqrt{u_j}\!\!\!\sum_{\substack{q\le B\\P\textrm{-smooth }\\\textrm{squarefree}}}\!\!\!\!\frac{\psi(q)}{q}\sum_{d|q}\frac{\mu(d)}{\psi(d)}\!\!\!\!\sum_{\substack{m\le B\\\gcd(m,q)=d}}\!\!\!\!\frac{1}{m}\ell'\Bigl(\varphi_{d,P}(m)\Bigr)\ell\Bigl(\tfrac{m}{\varphi_{d,P}(m)}\Bigr)J_1\Bigl(4\pi\sqrt{u_j}\tfrac{m}{q}\Bigr).
\]
The RHS of \eqref{primes-P}, \eqref{prime-P}  is then equal to $\lim_{B\to\infty}\textrm{RHS}(j,P,B)$, where we set $P=\infty$ in \eqref{primes-P} and \eqref{prime-P} to get the RHS of Proposition~\ref{primes} and Conjecture \ref{main-conjecture}, respectively, and we set $P=1$ (in either \eqref{primes-P} or \eqref{prime-P}) to get the RHS of Proposition \ref{all}.

The multiplicativity of $\ell(m),\ell'(m),\psi(m),\mu(m)$ allows us to efficiently precompute a table of all values of these functions for positive integers $m\le B$ using a sieving approach.  For each $P$-smooth squarefree $d\le B$ we can similarly compute a table of all values of $\varphi_{d,P}(m)$ for $m\le B$, and for \eqref{primes-P} it suffices to just compute $\varphi_{1,P}(m)$ for $m\le B$, since $\varphi_{d,P}(m)$ does not depend on $d$ in these cases.  We also note that for $p>3$ we have $\ell_{p,\nu}=0$  for $\nu < 10$ (since $\dim S_0^{\nu+2}(\SL_2(\mathbb Z))=0$ for $\nu<10$), which means we can ignore all $m$ in the inner sum that are divisible by a prime $p>3$ for which either $p>P$ or $p\nmid d$ (depending on which case we are in) and $v_p(m) < 5$.

With this approach computing $\mathrm{RHS}(j,P,B)$ involves $O(B^{2+o(1)})$ arithmetic operations  and at most $B^2$ evaluations of the Bessel function $J_1(x)$. When performed at any fixed precision, the evaluations of $J_1(x)$ will dominate the computation, and the space required is $O(B)$.  In our implementation we computed $J_1(x)$ to 53 bits of precision via the GNU Scientific Library function \texttt{gsl\_sf\_bessel\_J1} \cite{GSL} and we used 80 bits of precision (GCC type \texttt{long double}) for all arithmetic operations.  For $B\le 2^{20}$ this guarantees at least 20 bits of precision in the values of $\mathrm{RHS}(j,P,B)$ (more when $P$ is small), and in practice the precision is much better than this worst case estimate (we compared a subset of our computations against results computed using 106 bits of precision throughout).

\subsection{Convergence}\label{conv}
In this section we fix $u_{\max}=1$ so that $\delta=\nicefrac{1}{r}$ throughout, and unless otherwise specified we are considering the LHS and RHS of \eqref{prime-P} for various $P$, except in the $P=\infty$ case where we consider the LHS and RHS of Conjecture \ref{main-prediction}.

To determine a suitable choice of $B$ we computed $\mathrm{RHS}(j,P,B)$ for various $P$ and increasing values of $B$ for $0\le j < r=100$ with $\delta=\nicefrac{1}{100}$; we use much larger values of $r$ below, but for empirical tests of convergence, taking 100 uniformly spaced values of $u\in (0,1]$ suffices and allows us to explore a larger range of $B$ at lower cost.
Table~\ref{tab:Blimit} tabulates values of
\[
\max_{0\le j<100}\bigl|\mathrm{RHS}(j,P,B)-\mathrm{RHS}(j,P,2B)\bigr|
\]
for ten different choices of $P$ with $B$ ranging from $2^{10}$ to $2^{20}$.
For reference, in Figures 1--3 the vertical height of each pixel is approximately $\nicefrac{1}{40}=0.025$, and Table~\ref{tab:Blimit} suggests we should take $B\ge 2^{15}$ to achieve this resolution.  Entries in Table~\ref{tab:Blimit} measure the difference between $B$ and $2B$, not $B$ and the limit as $B\to\infty$, but the signs of the differences are roughly equidistributed, so the differences are not cumulative.
To give ourselves some margin for error, we take $B=10^5$, which is used in all our figures and tables other than Table~\ref{tab:Blimit}.

\begin{table}

\small
\begin{tabular}{cccccccccccc}
\multicolumn{1}{c}{} && \multicolumn{10}{c}{$P$} \\
\cline{3-12}
\noalign{\vskip 4pt}
$B$ && $2^{1}$& $2^{2}$& $2^{3}$& $2^{4}$& $2^{5}$& $2^{6}$& $2^{7}$& $2^{8}$& $2^{9}$& $2^{10}$\\
\cline{1-1}\cline{3-12}
\noalign{\vskip 4pt}
$2^{10}$ & & 0.0344 & 0.0512 & 0.0550 & 0.0554 & 0.0568 & 0.0587 & 0.0597 & 0.0600 & 0.0601 & 0.0600\\
$2^{11}$ & & 0.0062 & 0.0064 & 0.0151 & 0.0164 & 0.0167 & 0.0169 & 0.0171 & 0.0170 & 0.0171 & 0.0172\\
$2^{12}$ & & 0.0169 & 0.0210 & 0.0198 & 0.0208 & 0.0216 & 0.0217 & 0.0219 & 0.0220 & 0.0221 & 0.0221\\
$2^{13}$ & & 0.0114 & 0.0169 & 0.0161 & 0.0162 & 0.0166 & 0.0168 & 0.0168 & 0.0170 & 0.0170 & 0.0170\\
$2^{14}$ & & 0.0080 & 0.0122 & 0.0127 & 0.0129 & 0.0134 & 0.0135 & 0.0135 & 0.0135 & 0.0135 & 0.0135\\
$2^{15}$ & & 0.0079 & 0.0117 & 0.0122 & 0.0123 & 0.0127 & 0.0127 & 0.0128 & 0.0128 & 0.0128 & 0.0128\\
$2^{16}$ & & 0.0013 & 0.0016 & 0.0024 & 0.0027 & 0.0028 & 0.0027 & 0.0028 & 0.0027 & 0.0027 & 0.0027\\
$2^{17}$ & & 0.0034 & 0.0039 & 0.0045 & 0.0050 & 0.0051 & 0.0051 & 0.0051 & 0.0051 & 0.0051 & 0.0051\\
$2^{18}$ & & 0.0034 & 0.0039 & 0.0044 & 0.0045 & 0.0046 & 0.0046 & 0.0046 & 0.0046 & 0.0047 & 0.0047\\
$2^{19}$ & & 0.0025 & 0.0029 & 0.0029 & 0.0029 & 0.0029 & 0.0029 & 0.0029 & 0.0029 & 0.0029 & 0.0028\\
$2^{20}$ & & 0.0019 & 0.0021 & 0.0021 & 0.0022 & 0.0022 & 0.0021 & 0.0022 & 0.0022 & 0.0022 & 0.0022\\
\bottomrule
\end{tabular}
\medskip

\caption{Convergence as $B\to\infty$ for varying $P$ in \eqref{prime-P}.  Each table entry is $\max_j|\mathrm{RHS}(j,P,B)-\mathrm{RHS}(j,P,2B)|$ over $0\le j < r= 100$.\vspace{-8pt}}\label{tab:Blimit}
\end{table}

We now consider the choice of $r$ (how many subintervals $I_j$ to use).  For the computations of $\mathrm{LHS}(j,P,X)$ we do not want to make $r$ too large. We want each subinterval $I_j$ to contain many values of $n/N(E)$ as $E$ ranges over elliptic curves with $H(E)\le X$, as we vary $X$ over a range of values that we can feasibly compute, and the value of $X$ constrains $N(E)\le 32X$ and hence the number of $n/N(E)\in I_j$.  The images shown in Figures 1--3 have a width of approximately 2000 pixels, which leads us to choose $r=2000$ when computing $\mathrm{LHS}(j,P,X)$.

For the computation of $\mathrm{RHS}(j,P,B)$ it is both feasible and desirable to use a larger value of~$r$.  As $r$ increases the piecewise linear function defined by the points $\bigl((j+\tfrac{1}{2})\delta,\, \mathrm{RHS}(j,P,B)\bigr)$ approaches a continuous function $f_{P,B}(u)$ that converges to the murmuration density function (the integrand on the RHS of \eqref{main-prediction}) in the limit as $P,B\to\infty$.  We would like to approximate this function on $(0,1]$ as accurately as we can, as this is the green curve shown in Figure~\ref{fig:pred}.

\begin{table}

\small
\begin{tabular}{ccccccccccccc}
\multicolumn{1}{c}{} && \multicolumn{10}{c}{$P$} \\
\cline{3-12}
\noalign{\vskip 4pt}
$r$ && $2^{1}$& $2^{2}$& $2^{3}$& $2^{4}$& $2^{5}$& $2^{6}$& $2^{7}$& $2^{8}$& $2^{9}$& $2^{10}$\\
\cline{1-1}\cline{3-12}
\noalign{\vskip 4pt}
$2^{10}$ & & 0.1817 & 0.2098 & 0.2240 & 0.2264 & 0.2289 & 0.2297 & 0.2299 & 0.2300 & 0.2300 & 0.2300\\
$2^{11}$ & & 0.1218 & 0.1459 & 0.1644 & 0.1669 & 0.1682 & 0.1687 & 0.1689 & 0.1689 & 0.1689 & 0.1689\\
$2^{12}$ & & 0.0952 & 0.1145 & 0.1231 & 0.1237 & 0.1242 & 0.1243 & 0.1243 & 0.1243 & 0.1243 & 0.1243\\
$2^{13}$ & & 0.0635 & 0.0826 & 0.0877 & 0.0883 & 0.0884 & 0.0884 & 0.0885 & 0.0885 & 0.0885 & 0.0885\\
$2^{14}$ & & 0.0507 & 0.0606 & 0.0624 & 0.0627 & 0.0627 & 0.0627 & 0.0627 & 0.0627 & 0.0627 & 0.0627\\
$2^{15}$ & & 0.0339 & 0.0388 & 0.0411 & 0.0413 & 0.0413 & 0.0413 & 0.0413 & 0.0413 & 0.0413 & 0.0413\\
$2^{16}$ & & 0.0237 & 0.0275 & 0.0286 & 0.0287 & 0.0287 & 0.0287 & 0.0287 & 0.0287 & 0.0287 & 0.0287\\
$2^{17}$ & & 0.0168 & 0.0209 & 0.0215 & 0.0216 & 0.0216 & 0.0216 & 0.0216 & 0.0216 & 0.0216 & 0.0216\\
\bottomrule
\end{tabular}
\medskip

\caption{Convergence as $r\to\infty$ for $B=10^5$ and varying $P$ in \eqref{prime-P}.  Each table entry in the row for $r=2^n$ is the maximum absolute difference between the two sides of \eqref{rapprox} over $1\le j < 2r-1$.}\label{tab:rlimit}
\end{table}

Let $\mathrm{RHS}_r(j,P,B)$ denote the value of $\mathrm{RHS}(j,P,B)$ with $\delta=\frac{1}{r}$ determining $u_j=(j+\tfrac{1}{2})\delta$ in the formula for $\mathrm{RHS}(j,P,B)$.  
We can linearly interpolate values of $\mathrm{RHS}_{2r}(j,P,B)$ from values of $\mathrm{RHS}_{r}(j,P,B)$ via
\begin{equation}\label{rapprox}
\mathrm{RHS}_{2r}(j,P,B)\approx \frac{(2+(-1)^j)\mathrm{RHS}_{r}(\lfloor\tfrac{j}{2}\rfloor,P,B) + (2-(-1)^j)\mathrm{RHS}_{r}(\lfloor\tfrac{j}{2}\rfloor+1,P,B)}{4}.
\end{equation}

Table~\ref{tab:rlimit} provides data on the accuracy of this interpolation for various values of $P$ and increasing values of $r=2^n$.  If our goal is to approximate the murmuration density function on $(0,1]$ to within $0.025$ (the vertical height of a single pixel in our figures), the data in Table~\ref{tab:rlimit} suggests that $r=2000$ is too small; we instead use $r=10^5$ for computing $\mathrm{RHS}(j,P,B)$.

Having fixed $B=10^5$ and $r=10^5$, we now consider the rate of convergence as $P\to\infty$.
For any fixed $B$ and $j$, the value of $\mathrm{RHS}(j,P,B)$ is constant for all $P\ge B$, so we may take $P=B$ as the limit of $P\to\infty$.  Table~\ref{tab:Plimit} shows the maximum value of the difference between $\mathrm{RHS}(j,P,B)$ and $\mathrm{RHS}(j,\infty,B)=\mathrm{RHS}(j,B,B)$ over $0\le j < 10^5$ for increasing values of~$P$.  One can see that for $P\ge 64$ the difference is already too small to be noticeable in our figures.

\begin{table}
\small
\begin{tabular}{cccccccccc}
\multicolumn{10}{c}{$P$} \\
\cline{1-10}
\noalign{\vskip 4pt}
$2^{1}$& $2^{2}$& $2^{3}$& $2^{4}$& $2^{5}$& $2^{6}$& $2^{7}$& $2^{8}$& $2^{9}$& $2^{10}$\\
\cline{1-10}
\noalign{\vskip 4pt}
2.1133 & 0.9206 & 0.2260 & 0.1228 & 0.0539 & 0.0155 & 0.0074 & 0.0029 & 0.0012 & 0.0009\\
\bottomrule
\end{tabular}
\medskip

\caption{Convergence to the murmuration density function in \eqref{main-prediction} as $P\to\infty$ in \eqref{prime-P} with $B=10^5$ fixed.  Each entry is $\max_j |\mathrm{RHS}(j,P,B)-\mathrm{RHS}(j,\infty,B)|$ taken over $0\le j < r=100000$.}\label{tab:Plimit}
\end{table}

We now want to compare values of $\mathrm{LHS}(j,P,X)$ over $0\le j < r= 2000$ as $X\to\infty$ to corresponding values of $\mathrm{RHS}(j',P,B)$ over $0\le j' < r'=10^5$ for various values of $P$. 
We are thus led to define
\[
\mathrm{RHS}'(j,P,B) = \frac{1}{50}\sum_{j'=50j}^{50j+49}\mathrm{RHS}(j',P,B),
\]
which is the average value of $\mathrm{RHS}(j',P,B)$ over the interval $I_j=\bigcup_{50j\le j' < 50(j+1)} I_{j'}$.

Table~\ref{tab:LR} compares $\mathrm{LHS}(j,P,X)$ to $\mathrm{RHS}'(j,P,B)$ for various values of $P$ (including $P=\infty$) and increasing values of $X=2^n$.  The bottom entry in the $P=\infty$ column corresponds to Figure~\ref{fig:pred}, while the bottom entry in the $P=1$ column corresponds to Figure~\ref{fig:all}.

\begin{table}

\small
\begin{tabular}{crcccccccccc}
&&\multicolumn{10}{c}{$P$} \\
\cline{4-12}
\noalign{\vskip 4pt}
$X$ & $\#\mathcal{H}(X)$ && $1$ & $2$ & $4$ & $8$ & $16$ & $32$ & $64$ & $128$ & $\infty$\\
\cline{1-2}\cline{4-12}
\noalign{\vskip 4pt}
$2^{16}$ & 5042 & & 0.2271 & 0.4525 & 0.5655 & 0.6908 & 0.7665 & 0.8788 & 0.9484 & 1.0382 & 1.1693\\
$2^{17}$ & 9014 & & 0.1807 & 0.3601 & 0.4448 & 0.5383 & 0.5910 & 0.6715 & 0.7351 & 0.8088 & 0.9110\\
$2^{18}$ & 15936 & & 0.1382 & 0.2759 & 0.3416 & 0.4115 & 0.4420 & 0.4999 & 0.5492 & 0.6001 & 0.7252\\
$2^{19}$ & 28138 & & 0.1080 & 0.2155 & 0.2642 & 0.3207 & 0.3466 & 0.3878 & 0.4305 & 0.4680 & 0.5751\\
$2^{20}$ & 50886 & & 0.0838 & 0.1673 & 0.2065 & 0.2490 & 0.2667 & 0.3009 & 0.3327 & 0.3602 & 0.4597\\
$2^{21}$ & 89570 & & 0.0654 & 0.1307 & 0.1643 & 0.1959 & 0.2086 & 0.2305 & 0.2535 & 0.2765 & 0.3731\\
$2^{22}$ & 159988 & & 0.0527 & 0.1053 & 0.1306 & 0.1554 & 0.1659 & 0.1825 & 0.1983 & 0.2171 & 0.3067\\
$2^{23}$ & 286254 & & 0.0429 & 0.0858 & 0.1046 & 0.1236 & 0.1305 & 0.1426 & 0.1545 & 0.1717 & 0.2548\\
$2^{24}$ & 508826 & & 0.0351 & 0.0700 & 0.0851 & 0.0987 & 0.1036 & 0.1128 & 0.1233 & 0.1385 & 0.2139\\
$2^{25}$ & 906302 & & 0.0284 & 0.0566 & 0.0692 & 0.0806 & 0.0845 & 0.0917 & 0.1004 & 0.1119 & 0.1769\\
$2^{26}$ & 1615826 & & 0.0227 & 0.0452 & 0.0549 & 0.0639 & 0.0669 & 0.0730 & 0.0794 & 0.0877 & 0.1443\\
$2^{27}$ & 2873164 & & 0.0184 & 0.0367 & 0.0443 & 0.0505 & 0.0530 & 0.0572 & 0.0624 & 0.0696 & 0.1205\\
$2^{28}$ & 5122428 & & 0.0147 & 0.0294 & 0.0354 & 0.0403 & 0.0427 & 0.0463 & 0.0501 & 0.0561 & 0.1013\\
\bottomrule
\end{tabular}
\medskip

\caption{Comparison of $\mathrm{LHS}(j,P,X)$ and $\mathrm{RHS}'(j,P,B)$ in \eqref{prime-P} for varying~$P$, and in \eqref{main-prediction} for $P=\infty$, as $X\to\infty$. Each entry is the average of $|\mathrm{LHS}(j,P,X)-\mathrm{RHS}'(j,P,B)|$ over $0\le j < 2000$.}\label{tab:LR}
\end{table}

Table~\ref{tab:LRpc} shows analogous data in the $P=\infty$ case when we restrict to curves of prime conductor.
Now $\mathrm{LHS}(j,P,X)$ corresponds to the LHS of \eqref{main-prediction} restricted to curves of prime conductor, while $\mathrm{RHS}(j',P,B)$ corresponds to murmuration density function in Proposition~\ref{primes} with the sum truncated at $B=10^5$ and $\mathrm{RHS}'(j,P,B)$ defined as above.
The entry in Table~\ref{tab:LRpc} for $X=2^{40}$ corresponds to Figure~\ref{fig:primes}.

\begin{table}
\small
\begin{tabular}{crccccccccc}
\toprule
$X$ & $2^{31}$ & $2^{32}$ & $2^{33}$ & $2^{34}$ & $2^{35}$ & $2^{36}$ & $2^{37}$ & $2^{38}$ & $2^{39}$ & $2^{40}$\\
$\#\mathcal{H}(X)$ & 7909 & 13500 & 22867 & 39004 & 66799 & 114039 & 195027 & 334454 & 573010 & 987442\\
& 0.9963& 0.7913& 0.6131& 0.4852& 0.3812& 0.2983& 0.2320& 0.1802& 0.1414& 0.1097\\
\bottomrule
\end{tabular}
\medskip

\caption{Comparison of $\mathrm{LHS}(j,P,X)$ and $\mathrm{RHS}'(j,P,B)$ for $P=\infty$ in \eqref{primes-P}, as $X\to\infty$ restricted to prime conductor curves. Each entry in the bottom row is the average of $|\mathrm{LHS}(j,P,X)-\mathrm{RHS}'(j,P,B)|$ over $0\le j < 2000$.}\label{tab:LRpc}
\end{table}

We conclude with Figure~\ref{fig:Ps}, which plots the LHS and RHS of \eqref{prime-P} for $P=2,4,8,16$ with $B=10^{5}$ and $X=2^{28}$, which can be viewed as the first four in a sequence of plots as $P\to\infty$ interpolating between Figure~\ref{fig:all} $(P=1)$ and Figure~\ref{fig:pred} $(P=\infty)$ as $P$ increases.  These plots are already closer to Figure~\ref{fig:pred} than Figure~\ref{fig:all}, which demonstrates the large impact of the local factors at small primes.  As can be seen in Table~\ref{tab:Plimit} above, the plot for $P=16$ is already fairly close to the plot for $P=\infty$.

Animated versions of these plots showing convergence of $\mathrm{LHS}(j,P,X)$ as $X\to\infty$ for various fixed values of $P$, and as $P\to\infty$ for a fixed value of $X$, are available at
\bigskip

\begin{center}
\url{https://math.mit.edu/~drew/ssplots}
\end{center}

\FloatBarrier

\begin{figure}
\includegraphics[width=\textwidth]{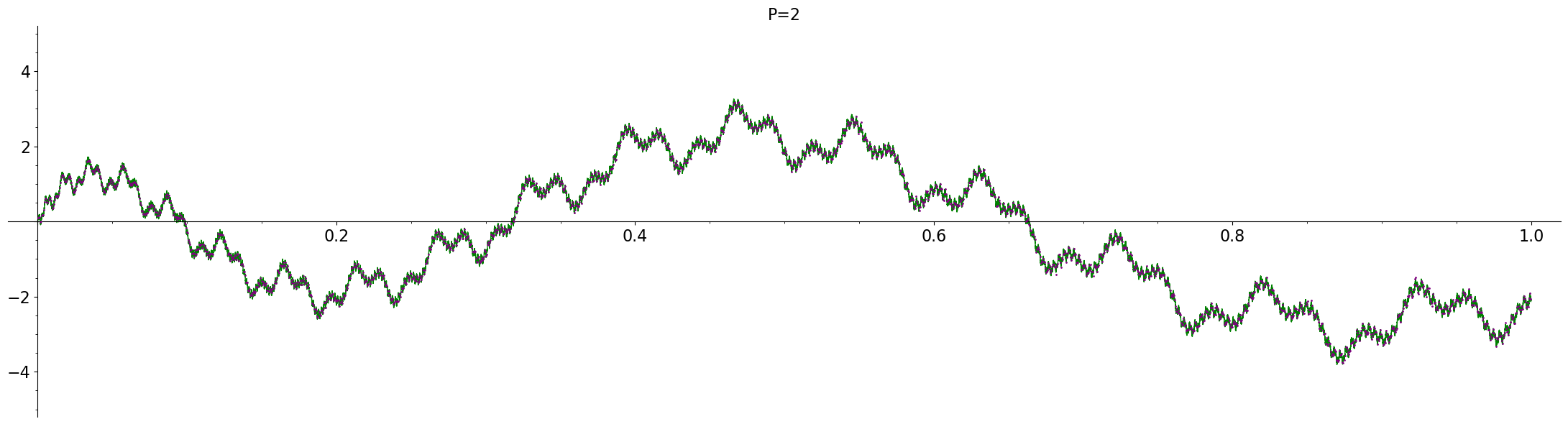}
\medskip

\includegraphics[width=\textwidth]{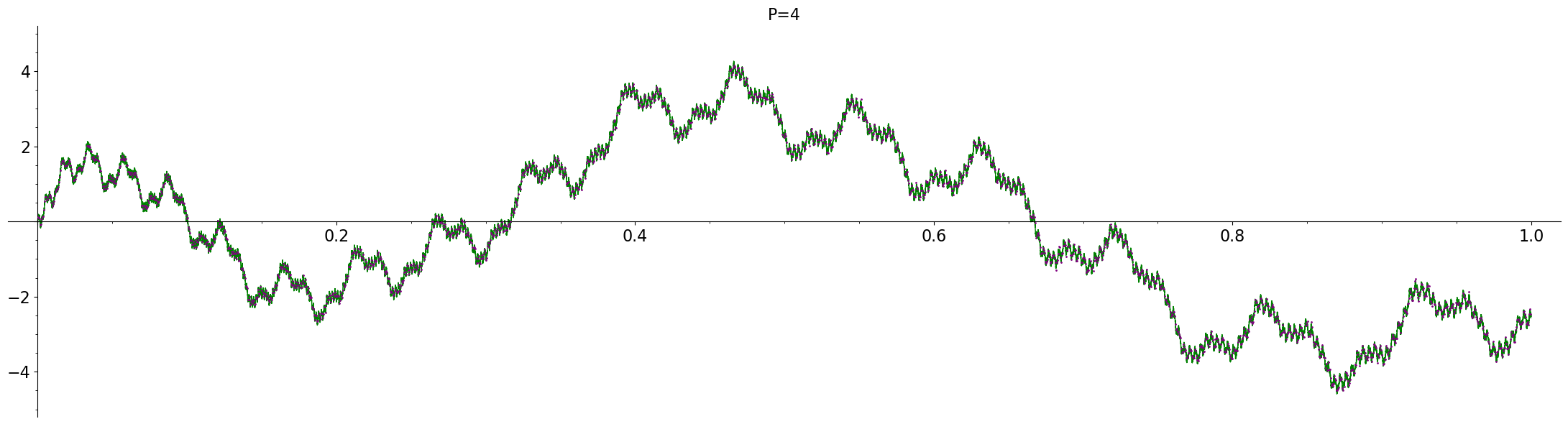}
\medskip

\includegraphics[width=\textwidth]{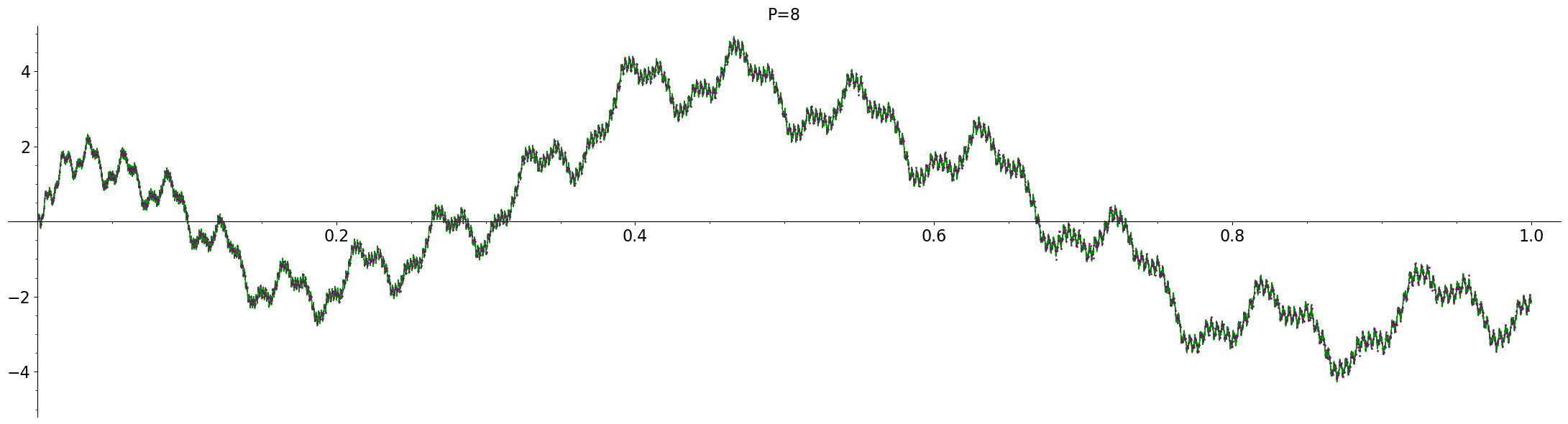}
\medskip

\includegraphics[width=\textwidth]{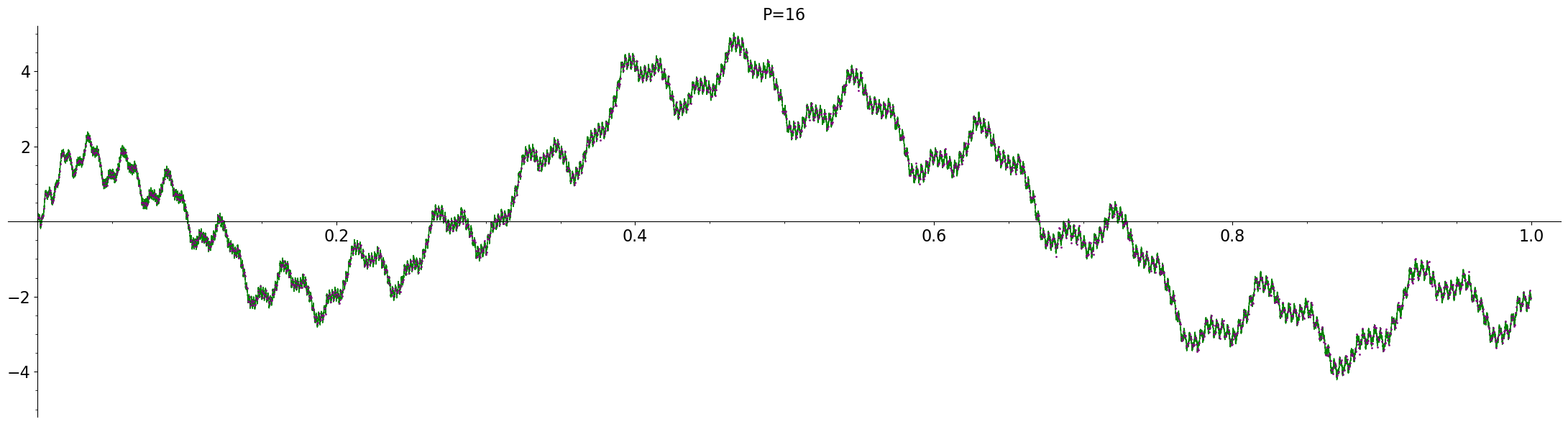}
\caption{\small Plots of $\mathrm{LHS}(j,P,2^{28})$  (purple dots) and $\mathrm{RHS}(j',P,10^5)$ (green curve) for $P=2,4,8,16$.}\label{fig:Ps}
\end{figure}

\FloatBarrier


\begin{thebibliography}{99}

\bibitem{BZ} Stephan Baier and Liangyi Zhao, \emph{On the low-lying zeroes of Hasse-Weil $L$-functions for elliptic curves}, Advances in Mathematics 219 (2008) pp. 952--985, \url{https://doi.org/10.1016/j.aim.2008.06.006}


\bibitem{BHKSSW}
Jennifer S. Balakrishnan, Wei Ho, Nathan Kaplan, Simon Spicer, William Stein, James Weigandt,
\emph{Databases of elliptic curves ordered by height and distributions of Selmer groups and ranks}, LMS J. Comput. Math. \textbf{19}(A) (2016), pp. 351--370. \url{https://doi.org/10.1112/S1461157016000152}

\bibitem{Bernstein} Daniel J. Bernstein, \emph{Detecting perfect powers in essentially linear time, and other studies in computational number theory}, PhD Thesis, University of California, Berkeley, 1995. \url{https://www.proquest.com/docview/304193115}

\bibitem{BSD} Bryan Birch and Peter Swinnerton-Dyer, \emph{Notes on elliptic curves. II}, J. Reine Angew. Math. {165} (1965), pp. 79--108. \url{https://doi.org/10.1515/crll.1965.218.79}

\bibitem{BFKMM} Valentin Blomer, \'{E}tienne Fouvry, Emmanuel Kowalski, Philippe Michel, and Djordje Mili\'cevi\'c, \emph{On moments of twisted $L$-functions}, American Journal of Mathematics {139} no. 3 (2017), pp. 707--768. \url{https://www.jstor.org/stable/44508999}

\bibitem{BFKMMS}Valentin Blomer, \'{E}tienne Fouvry, Emmanuel Kowalksi, Philippe Michel, Djordje Mili\'cevi\'c, and Will Sawin, \emph{The Second Moment Theory of Families of {$L$}-functions--The Case of Twisted Hecke $L$-Functions}, Memoirs of the American Mathematical Society \textbf{282} (2023). \url{https://doi.org/10.1090/memo/1394}

\bibitem{BBLLD} Jonathan Bober, Andrew R. Booker, Min Lee, and David Lowry-Duda, \emph{Murmurations of modular forms in the weight aspect}, Algebra \& Number Theory, in press (2025). \url{https://arxiv.org/abs/2310.07746}

\bibitem{BLLDSHZ} Andrew R. Booker, Min Lee, David Lowry-Duda, Andrei Seymour-Howell, and Nina Zubrilina \emph{Murmurations of Maass forms}, arXiv preprint (2024). \url{https://arxiv.org/abs/2409.00765}

\bibitem{BCDT} Christophe Breuil, Brian Conrad, Fred Diamond, and Richard Taylor, \emph{On the modularity of elliptic curves over $\mathbb{Q}$: wild 3-adic exercises}, J. Amer. Math. Soc. \textbf{14} (4), pp. 843--939. \url{https://doi.org/10.1090/S0894-0347-01-00370-8}

\bibitem{ratios} Brian Conrey, David W. Farmer and Martin R. Zirnbauer, \emph{Autocorrelation of ratios of {$L$}-functions}, Communications in number theory and physics Volume 2, Number 3, (2008), pp. 593--636. \url{https://link.intlpress.com/JDetail/1805800242891145218}

\bibitem{CKR} Edgar Costa, Kiran S. Kedlaya, and David Roe, \emph{Hypergeometric $L$-functions in average polynomial time, II},  Research in Number Theory \textbf{11} no. 32 (2025). \url{https://doi.org/10.1007/s40993-024-00593-8}

\bibitem{Cowan2} Alex Cowan, \emph{Murmurations and ratios conjectures}, arXiv preprint (2024). \url{https://arxiv.org/abs/2408.12723}

\bibitem{Cowan3} Alex Cowan, \emph{On the mean value of $\GL_1$ and $\GL_2$ $L$-functions, with applications to murmurations}, arXiv preprint (2025). \url{https://arxiv.org/abs/2504.09944}

\bibitem{Deligne}Pierre Deligne,  Formes modulaires et repr\'{e}sentations $\ell$-adiques, S\'{e}minaire N. Bourbaki, 1971, exp. no 355, pp. 139--172. \url{http://www.numdam.org/item/SB_1968-1969__11__139_0.pdf}

\bibitem{KS08} Kiran S. Kedlaya and Andrew V. Sutherland, 
\emph{Computing L-series of hyperelliptic curves}, in Eighth Algorithmic Number Theory Symposium (ANTS VIII), Lect. Notes Comp. Sci. 5011, Springer, Berlin, 2008, 312--326. \url{https://doi.org/10.1007/978-3-540-79456-1\_21}

\bibitem{ILS} Henryk Iwaniec, Wenzhi Luo, and Peter Sarnak, \emph{Low-lying zeros of families of $L$-functions}, Publications math\'ematiques de l'I.H.\'E.S. 91 (2000), pp. 55--131. \url{http://www.numdam.org/item/PMIHES_2000__91__55_0/}

\bibitem{ZMTables} Izrail Solomonovich Gradshteyn and Iosif Moiseevich Ryzhik,  \emph{Table of Integrals, Series, and Products}, Eight ed., eds. Daniel Zwillinger and Victor Moll, Elsevier / Academic Press (2014). \url{https://doi.org/10.1016/C2010-0-64839-5}

\bibitem{GSL} Brian Gough, \emph{GNU Scientific Library Reference Manual -- Third Edition}, Network Theory Ltd., 2009. \url{https://dl.acm.org/doi/10.5555/1538674}

\bibitem{HS14} David Harvey and Andrew V. Sutherland,
\emph{Computing Hasse-Witt matrices of hyperelliptic curves in average polynomial time}, 
LMS J. Comp. Math. \textbf{17} (2014), 257--273. \url{https://doi.org/10.1112/S1461157014000187}

\bibitem{HLOP} Yang-Hui He, Kyu-Hwan Lee, Thomas Oliver, and Alexey Pozdnyakov, \emph{Murmurations of elliptic curves}, Exp. Math. (2024), pp. 1--13. \url{https://doi.org/10.1080/10586458.2024.2382361}

\bibitem{HLOPS} Yang-Hui He, Kyu-Hwan Lee, Thomas Oliver, Alexey Pozdnyakov, and Andrew V. Sutherland, \emph{Murmurations of $L$-functions}. In preparation.

\bibitem{Helfgott04} Harald Andr\'es Helfgott, \emph{On the behaviour of root numbers in families of elliptic curves
}, arXiv preprint (2004). \url{https://arxiv.org/abs/math/0408141}

\bibitem{Hijikata} Hiroaki Hijikata, \emph{Explicit formula of the traces of Hecke operators for $\Gamma_0(N)$}, J. Math. Soc. Japan {\bf 26} (1974), pp. 56--82. \url{https://doi.org/10.2969/jmsj/02610056}

\bibitem{KMVK} Emmanuel Kowalski, Philippe Michel, and Jeffrey VanderKam, Rankin-Selberg $L$-functions in the level aspect, Duke Mathematical Journal \textbf{114} (2002), no. 1, pp. 123--191. \url{https://doi.org/10.1215/S0012-7094-02-11416-1}

\bibitem{LOP} Kyu-Hwan Lee, Thomas Oliver, and Alexey Pozdnyakov, \emph{Murmurations of Dirichlet characters}, arXiv preprint (2023). \url{https://arxiv.org/abs/2307.00256}

\bibitem{SarnakLetter} Peter Sarnak, \emph{Letter to Drew Sutherland and Nina Zubrilina on murmurations and root numbers}, August 2023. \url{https://publications.ias.edu/sarnak/paper/2726}

\bibitem{SST} Peter Sarnak, Sug Woo Shin, and Nicolas Templier, \emph{Families of $L$-functions and their symmetry}, in Werner M\"uller, Sug Woo Shin, and Nicolas Templier, (eds) Families of Automorphic Forms and the Trace Formula, Simons Symposia, Springer, (2016), pp. 531--578. \url{https://doi.org/10.1007/978-3-319-41424-9_13}

\bibitem{SS15} Igor Shparlinski and Andrew V. Sutherland,
\emph{On the distribution of Atkin and Elkies primes for reductions of elliptic curves on average}, LMS J. Comp. Math. \textbf{18} (2015), pp. 308--322. \url{https://doi.org/10.1112/S1461157015000017}.

\bibitem{Sut20} Andrew V. Sutherland,
\emph{Counting points on superelliptic curves in average polynomial time}, in
Fourteenth Algorithmic Number Theory Symposium (ANTS XIV),
Open Book Series \textbf{4} (2020), pp. 403--422. \url{https://doi.org/10.2140/obs.2020.4.403}

\bibitem{SutherlandLetter} Andrew V. Sutherland. Letter to Michael Rubinstein and Peter Sarnak, August 30, 2022. \url{https://publications.ias.edu/sarnak/paper/2725}

\bibitem{SutherlandTalk} Andrew V. Sutherland, \emph{Murmurations of arithmetic $L$-functions}, Lecture at Institute for Advanced Study Special Number Theory Afternoon, April 21, 2023. \url{https://math.mit.edu/~drew/MurmurationsUnanimated.pdf}

\bibitem{SutherlandThesis} Andrew V. Sutherland, \emph{Order computations in generic groups}, PhD thesis, Massachusetts Institute of Technology, 2007. \url{http://groups.csail.mit.edu/cis/theses/sutherland-phd.pdf}.

\bibitem{Wang} Zeyu Wang, \emph{Murmurations of Hecke $L$-functions of imaginary quadratic fields}, arXiv preprint (2025). \url{https://arxiv.org/abs/2503.17967}

\bibitem{Young} Matthew P. Young, \emph{Low-lying zeros of families of elliptic curves}, Journal of the American Mathematical Society 19 (2006), pp. 205--250, \url{https://www.jstor.org/stable/20161274}

\bibitem{Zubrilina} Nina Zubrilina, \emph{Murmurations}, Inventiones Mathematicae \textbf{241} (2025), pp. 627--680. \url{https://doi.org/10.1007/s00222-025-01347-8}

\end{thebibliography}
\end{document}